\documentclass[a4paper,11pt,reqno]{amsart}

\usepackage{amsmath, amsfonts, amssymb, amsthm, amscd}
\usepackage[foot]{amsaddr} 
\usepackage[a4paper,scale={0.72,0.75},marginratio={1:1},footskip=7mm,headsep=10mm]{geometry}
\usepackage[T1]{fontenc}			
\usepackage[utf8]{inputenc}
\usepackage[italian, english]{babel}
\usepackage{csquotes}
\usepackage[square,numbers]{natbib}	
\usepackage{bbm}
\usepackage{bm}
\usepackage{mathtools}
\usepackage{xspace} 
\usepackage{booktabs}
\usepackage{perpage}
\usepackage{enumerate}
\usepackage[pdftitle={Stochastic filtering and optimal control of pure jump Markov processes with noise-free partial observation},%
pdfauthor={Alessandro Calvia},%
pdftex,hidelinks]{hyperref}
\usepackage{verbatim} 
\usepackage{mathrsfs}

\newcommand{\bbE}{{\ensuremath{\mathbb E}} }

\newcommand{\bbP}{{\ensuremath{\mathbb P}} }

\newcommand{\bbX}{{\ensuremath{\mathbb X}} }
\newcommand{\bbY}{{\ensuremath{\mathbb Y}} }

\newcommand{\cA}{{\ensuremath{\mathcal A}} }
\newcommand{\cB}{{\ensuremath{\mathcal B}} }
\newcommand{\cC}{{\ensuremath{\mathcal C}} }
\newcommand{\cD}{{\ensuremath{\mathcal D}} }

\newcommand{\cF}{{\ensuremath{\mathcal F}} }
\newcommand{\cG}{{\ensuremath{\mathcal G}} }
\newcommand{\cH}{{\ensuremath{\mathcal H}} }
\newcommand{\cI}{{\ensuremath{\mathcal I}} }

\newcommand{\cK}{{\ensuremath{\mathcal K}} }
\newcommand{\cL}{{\ensuremath{\mathcal L}} }
\newcommand{\cM}{{\ensuremath{\mathcal M}} }
\newcommand{\cN}{{\ensuremath{\mathcal N}} }
\newcommand{\cO}{{\ensuremath{\mathcal O}} }
\newcommand{\cP}{{\ensuremath{\mathcal P}} }

\newcommand{\cU}{{\ensuremath{\mathcal U}} }

\newcommand{\cX}{{\ensuremath{\mathcal X}} }
\newcommand{\cY}{{\ensuremath{\mathcal Y}} }
\newcommand{\cZ}{{\ensuremath{\mathcal Z}} }

\newcommand{\bfa}{{\ensuremath{\mathbf a}} }

\newcommand{\bfu}{{\ensuremath{\mathbf u}} }

\newcommand{\fru}{{\ensuremath{\mathfrak u}} }

\newcommand{\dd}{{\ensuremath{\mathrm d}} }

\newcommand{\dB}{{\ensuremath{\mathrm B}} }
\newcommand{\dC}{{\ensuremath{\mathrm C}} }
\newcommand{\dD}{{\ensuremath{\mathrm D}} }

\newcommand{\R}{\mathbb{R}}

\newcommand{\N}{\mathbb{N}}
\newcommand{\Q}{\mathbb{Q}}

\renewcommand{\P}{\mathbb{P}}

\newcommand{\p}{\ensuremath{\mathrm P}}
\newcommand{\e}{\ensuremath{\mathrm E}}

\newcommand{\probsp}{(\Omega, \, \mathcal{F}, \, \mathbbm{P})}

\newcommand{\ind}{\ensuremath{\mathbf{1}}}
\newcommand{\one}{\ensuremath{\mathsf{1}}}
\newcommand{\pre}{h^{-1}}
\newcommand{\inprod}[2]{\langle #1 , #2 \rangle}
\DeclarePairedDelimiterX{\abs}[1]{\lvert}{\rvert}{#1}
\DeclarePairedDelimiterX{\norm}[1]{\lVert}{\rVert}{#1}
\DeclareMathOperator*{\inter}{\bigcap}       

\newcommand{\suptwo}[2]{\sup_{\substack{#1 \\ #2}}} 
\newcommand{\intertwo}[2]{\inter_{\substack{#1 \\ #2}}} 

\DeclareMathOperator{\projY}{proj_Y}
\DeclareMathOperator{\supp}{supp}
\DeclareMathOperator*{\esssup}{\mathrm{ess\,sup}}

\newcommand{\ie}{i.\,e. }

\newcommand{\eg}{e.\,g. }

\renewcommand{\epsilon}{\varepsilon}
\let\temp\theta
\let\theta\vartheta
\let\vartheta\temp

\newcommand{\phivarphi}
{
\let\temp\phi
\let\phi\varphi
\let\varphi\temp
}

\theoremstyle{plain}
\newtheorem{theorem}{Theorem}[section]
\newtheorem*{theorem*}{Theorem}
\newtheorem{lemma}[theorem]{Lemma}
\newtheorem*{lemma*}{Lemma}
\newtheorem{proposition}[theorem]{Proposition}
\newtheorem*{proposition*}{Proposition}

\theoremstyle{definition}
\newtheorem{definition}{Definition}[section]
\newtheorem{assumption}{Assumption}[section]

\newtheorem{example}{Example}[section]

\theoremstyle{remark}
\newtheorem{rem}{Remark}[section]
\newtheorem*{rem*}{Remark}

{\left\lbrace\begin{array}{@{}l@{}}}%
{\end{array}\right.}

\numberwithin{equation}{section}

\title[Noise-free stochastic filtering and optimal control]
	{Stochastic filtering and optimal control of pure jump Markov processes with \mbox{noise-free} partial observation}
\author[A.~Calvia]{Alessandro~Calvia$^\star$}
\address{$^\star$University of Milano-Bicocca, Department of Statistics and Quantitative Methods, via Bicocca degli Arcimboldi 8, 20126 Milano (Italy).}
\email{alessandro.calvia@unimib.it}
\thanks{This research was partially supported by three GNAMPA-INdAM projects in 2015, 2016 and 2017 and by MIUR-PRIN 2015 project \textit{Deterministic and stochastic evolution equations}.}
\date{}

\begin{document}
	
\begin{abstract}
	We consider an infinite horizon optimal control problem for a pure jump Markov process $X$, taking values in a complete and separable metric space $I$, with noise-free partial observation. The observation process is defined as $Y_t = h(X_t)$, $t \geq 0$, where $h$ is a given map defined on $I$. The observation is noise-free in the sense that the only source of randomness is the process $X$ itself. The aim is to minimize a discounted cost functional. In the first part of the paper we write down an explicit filtering equation and characterize the filtering process as a Piecewise Deterministic Process. In the second part, after transforming the original control problem with partial observation into one with complete observation (the separated problem) using filtering equations, we prove the equivalence of the original and separated problems through an explicit formula linking their respective value functions. The value function of the separated problem is also characterized as the unique fixed point of a suitably defined contraction mapping.
\end{abstract}

\maketitle
	
\noindent \textbf{Keywords:} stochastic filtering, partial observation control problem, pure jump processes, piecewise-deterministic Markov processes, Markov decision processes.

\noindent \textbf{AMS 2010:} 93E11, 93E20, 60J25, 60J75

\section{Introduction} \label{sec:intro}
\phivarphi
This paper is devoted to analyze an optimal control problem on infinite time horizon for a continuous-time pure jump Markov process with partial and \mbox{noise-free} observation. The model studied here is specified by a triple of continuous-time stochastic processes $(X,Y,\bfu)=(X_t, Y_t, u_t)_{t \geq 0}$. The process $X$, called \emph{unobserved process}, is a pure jump Markov process, with values in a complete and separable metric space $I$ and initial law $\mu$. The process $Y$, called \emph{observed process}, takes values in another complete and separable metric space $O$. The process $\bfu$, called \emph{control process}, takes values in the set of Borel probability measures on a compact metric space $U$; we call it \emph{admissible} if it is predictable with respect to the filtration $(\cY_t)_{t \geq 0}$ generated by the process $Y$. This process represents the action of a \emph{relaxed control}, a choice motivated by technical reasons. However, we are able to recover classical $U$-valued processes, \ie \emph{ordinary controls}, by standard approximation theorems.

The aim of our problem is to control the rate transition measure of the unobserved process $X$ (otherwise said, to control its infinitesimal generator) via the control process $\bfu$, so that the functional
\begin{equation} \label{eq:introcontrolpb}
	J(\mu, \bfu) = \e_\mu^\bfu \int_0^{+\infty} \int_U e^{-\beta t} f(X_t, \fru) \, u_t(\dd \fru) \, \dd t,
\end{equation} 
is minimized. Here $f$ is a real-valued bounded function, called \emph{running cost function}, $\beta$ is a positive \emph{discount factor} and the expectation is taken with respect to a specific probability measure $\p_\mu^\bfu$, depending on the initial law of the unobserved process $X$ and on the control $\bfu$. The infimum of the functional $J$ among all admissible processes is the \emph{value function} $V(\mu)$.

As is known, the rate transition measure associated to a pure jump process is a transition kernel that, along with the initial distribution, determines its law. In other words, its sojourn times and its post jump locations are random variables whose law can be expressed in terms of the rate transition measure. It is worth noticing that the present framework contains the optimal control problem with noise-free partial observation of a continuous-time homogeneous finite-state Markov chain. This problem, that corresponds to the case where the state space $I$ of the unobserved process $X$ is a finite set, has been dealt with in \citep{calvia:optcontrol}. In that case, the rate transition measure reduces to a matrix (sometimes called Q-matrix (see \eg \citep{norris:markovchains}) and a more precise characterization of the value function can be obtained, thanks to the peculiar structure of the problem. Such a setting may be more familiar to the reader and we invite she/he to keep in mind this situation also in the present setting. 

Let us now present the program that we intend to develop in the present work. The solution of an optimal control problem with partial observation requires a two step procedure. First, one needs to write a \emph{filtering equation}, characterizing at each time $t \geq 0$ the conditional law of $X_t$ given $\cY_t$. More specifically, its solution $\pi = (\pi_t)_{t \geq 0}$, called the \emph{filtering process}, satisfies for each bounded and measurable function $\phi \colon I \to \R$ the following relation
\begin{equation} \label{eq:introfilteq}
\pi_t(\phi) \coloneqq \int_I \phi(x) \, \pi_t(\dd x) = \e_\mu^\bfu[\phi(X_t) \mid \cY_t], \quad \p_\mu^\bfu\text{-a.s.}, \, t \geq 0.
\end{equation}
The second step uses the filtering process $\pi$ to recast the partial observation problem (\ref{eq:introcontrolpb}) into one with complete observation. The new problem, called \emph{separated problem}, has the filtering process $\pi$ as its state process. We will show that the original and the separated problems are equivalent, in the sense that there exists a precise relationship between the original value function $V$ and the value function $v$ of the separated problem.
Thanks to this equivalence, we can concentrate our attention on the latter problem to characterize $v$ (and indirectly the original value function $V$) as the unique fixed point of a contraction mapping.

The specific feature of our model is the \emph{\mbox{noise-free} observation}. With this expression we mean that the stochastic behavior of the observed process comes exclusively from the unobserved one, since no exogenous noise acts on it.
In particular, we suppose that $Y_t = h(X_t)$ for all $t \geq 0$ and for some measurable function $h \colon I \to O$. The function $h$ generates a partition of the set $I$ by its level sets $h^{-1}(y)$, $y \in O$. Therefore, if at time $t$ an observer records the value $Y_t = y$ for some $y \in O$, then she/he knows that $X_t$ takes some value in the set of states $h^{-1}(y)$. We can equivalently say that the observation is given by the level set where $X_t$ lies at any time $t \geq 0$. 

This situation has been studied also under different assumptions on the signal and observed processes; for instance, the filtering problem alone for an unobserved diffusion process has been dealt with (albeit in part) in \citep{bryson:linearfilt, crisan:nonlinfilt, runggaldier:filt, takeuchi:lsqestimation}, while \citep{joannides:nonlinfilt} is devoted entirely to non-linear filtering with noise-free observation (therein called perfect observation). We also mention the book by Xiong \citep[Ch. 11]{xiong:intrtostochfiltth}, where the referenced chapter is devoted to singular filtering. The case of an unobserved process given by a pure-jump $\R^d$-valued Markov process is studied in \citep{cecigerardi:filtering2, cecigerardi:filtering, cecigerardi:pathdependent, cecigerarditardelli:existence}, where the authors consider counting observations. Finally, we recall that the case of a controlled Markov chain with noise free partial observation has been treated in \citep{calvia:optcontrol}. We also mention that a model similar to the one considered in \citep{calvia:optcontrol}, based on the so called \emph{information structures}, has been dealt with in \citep{winter:phdthesis}. These models, that need \emph{ad hoc} results to be analyzed, have received a sporadic treatment in the literature, despite their potential and useful connection with applications; among the others, we mention queuing systems (see \eg \citep{asmussen:applprob, bremaud:pp}) and inventory models (see \eg \citep{bensoussan:inventory}). We point out that our problem is connected to Hidden Markov Models (see \citep{elliott:hmm} for a comprehensive exposition on this subject). 

A work closely related to the filtering problem presented here, is the paper by F.~Confortola and M.~Fuhrman \citep{confortola:filt}. This work, that shares with the present one the very same definition of noise-free observation, is exclusively concerned with the filtering problem for a finite-state Markov chain. There filtering equations (\ref{eq:introfilteq}) are computed; the filtering process $\pi$ is characterized as a \emph{Piecewise Deterministic Process} (or PDP for short), a class of processes introduced by M.~H.~A.~Davis (see \citep{davis:markovmodels}), and its local characteristics are written down explicitly. Here we generalize the results on the filtering problem therein contained, adopting an equivalent description of the signal and observed processes as Marked Point Processes (see \citep{bremaud:pp}). In this way, thanks to the so called \emph{innovations approach}, we are able to compute the filtering equation and characterize the filtering process as a PDP, writing down its local characteristics explicitly. It is quite unusual to obtain such a detailed description of the filtering process in a general framework as ours. To the best of our knowledge, the filtering process considered here represents the first instance of a PDP taking values in a subset of a Banach space (the set of finite signed measures endowed with the total variation norm, as we will see precisely in Section \ref{sec:filterchar}). Hilbert space-valued PDPs have been considered recently in \citep{buckwar:PDP, renault:PDMP}.

The innovations approach allows us to deduce immediately that a similar filtering equation and an analogous characterization of the filtering process as a PDP are valid also in the controlled case. This permits us to reformulate the original control problem with partial observation for the pure jump process $X$ into a control problem with complete observation for a PDP, called the \emph{separated problem}. More specifically, the separated problem will be a discrete-time one related to a specific Markov decision model formulated in terms of a PDP. The reduction of PDP optimal control problems to discrete-time Markov decision processes is exploited \eg in \citep{almudevar:pdmpcontrol, altay:portoptim, calvia:optcontrol, colaneri:optliquid, costadufour:pdmpavgcontrol, davis:pdpcontrol, davis:markovmodels, forwick:pdp, renault:PDMP}.

We also note an important difference between the approach to PDP optimal control problems presented in \citep{davis:markovmodels} (and employed in other works as \citep{almudevar:pdmpcontrol, bandini:constrainedBSDEs, bandini:pdmpoptcontrol, costadufour:pdmpavgcontrol, davis:pdpcontrol, davisfarid:pdpvisc, dempster:pdmpcontrol, renault:PDMP}) and ours. In the book by Davis the class of control processes is represented by \emph{piecewise open-loop controls}, introduced by D.~Vermes in \citep{vermes:optcontrol}. These are processes depending only on the time elapsed since the last jump and the position at the last jump time of the PDP. In our problem, instead, we are forced to use a more general class of control policies depending on the past jump times and jump positions of the PDP. In fact, as we shall later see, we can find a correspondence between controls for the original problem and policies for the separated one only looking at this larger class. In this sense, an approach closer to ours can be traced in \citep{costadufour:PDPoptcontrol}. There the authors consider an optimal control problem for a PDP (with complete observation), where the control parameter acts only on the jump intensity and on the transition measure of the process but not on its deterministic flow.

The paper is organized as follows: in Section \ref{sec:filtering} we formulate the stochastic filtering problem in the uncontrolled case and prove the filtering equation in Theorem \ref{thm:filteringequation}. In Section \ref{sec:filterchar} we characterize the filtering process as a PDP, writing down explicitly its local characteristics. We chose not to include the control process in these two first Sections to make more clear both the notation and the proofs of the results. However, as anticipated earlier, these results remain valid also in the controlled case. In Section \ref{sec:jmpoptcontrol} we solve our optimal control problem with noise-free partial observation. In Subsection \ref{sec:optcontrolform} we introduce it precisely and recapitulate the main results of Sections \ref{sec:filtering} and \ref{sec:filterchar} in the controlled case. Using the filtering process we can transform the optimal control problem with partial observation into a complete observation one, \ie the separated problem. This latter problem is properly formulated in Subsection \ref{sec:pdpoptcontrol} and we prove its equivalence with the original one, obtaining also the characterization of its value function $v$ as the unique fixed point of a contraction mapping. 

We mention that the results of this paper have been presented at the \textit{First Italian Meeting on Probability and Mathematical Statistics}, held at the \textit{University of Torino \& Politecnico di Torino}, 19--22/06/2017 and at the workshop \textit{Stochastic Control, BSDEs and new developments} held in Roscoff (France), 11--15/09/2017.

\subsection{Notation}
We collect here for the reader's convenience the main notation and conventions used in this research article.

Throughout the paper the set $\N$ denotes the set of natural integers $\N = \{1, 2, \dots \}$, whereas $\N_0 = \{0, 1, \dots \}$. We use also the symbols $\bar \N = \N \cup \{\infty\}$ and $\bar \N_0 = \N_0 \cup \{\infty\}$. We indicate by $\cN$ the collection of null sets in some specified probability space.

For a fixed metric space $E$, we denote by $\dB_b(E)$ (resp. $\dC(E), \, \dC_b(E)$) the set of real-valued bounded measurable (resp. continuous, bounded continuous) functions on $E$. The symbol $\cB(E)$ indicates the Borel $\sigma$-algebra on $E$ and we denote by $\cM(E), \, \cM_+(E), \, \cP(E)$ the sets of Borel finite signed, finite and probability measures on $E$.

The symbol $\ind_C$ denotes the indicator function of a set $C$, while $\one$ is the constant function equal to $1$. For a fixed metric space $E$ and $\mu \in \cM(E)$, we indicate by $\supp(\mu)$ the support of $\mu$ and by  $\abs{\mu}$ the total variation measure corresponding to $\mu$. If $f \colon E \to \R$ is an integrable function with respect to $\mu$, we denote by $f \mu \in \cM(E)$ the measure with density $f$ with respect to $\mu$ and by $\mu(f) \coloneqq \int_E f(x) \, \dd \mu(x)$. If $f$ is a function of several variables, we define $\mu(f; \, \cdot) \coloneqq \int_E f(x, \cdot) \, \dd \mu(x)$. 

Finally, as far as measurability is concerned, whenever we write the word \emph{measurable} it is understood that we mean \emph{Borel-measurable}, unless otherwise specified. 

\section{The filtering problem} \label{sec:filtering}
\subsection{Formulation}
We briefly recall some basic aspects of stochastic filtering. Let us fix two complete and separable metric spaces $I$ and $O$ equipped with their respective Borel $\sigma$-algebras $\cI$ and $\cO$. The basic datum of a stochastic filtering problem consists in a pair of stochastic processes $(X,Y) = (X_t, Y_t)_{t \geq 0}$, defined on some complete probability space $\probsp$. The process $X$, called \emph{unobserved} or \emph{signal process}, takes values in the set $I$, while the process $Y$, called \emph{observed} or \emph{data process} takes values in the set $O$.
The aim is to find a $\cP(I)$-valued process $\pi = (\pi_t)_{t \geq 0}$ such that for all $t \geq 0$ and all $\phi \in \dB_b(I)$
\begin{equation}\label{eq:filtproc}
\int_I \phi(x) \, \pi_t(\dd x) = \bbE[\phi(X_t) \mid \cY_t], \quad \bbP\text{-a.s.}
\end{equation}
where $(\cY_t)_{t \geq 0}$ is the natural completed filtration of the observed process $Y$. That is, we are looking for a probabilistic estimate of the unobserved state (or of a measurable function of it) given the observation provided by $Y$. For sake of brevity we will often write 
$$\pi_t(\phi) \coloneqq \int_I \phi(x) \, \pi_t(\dd x).$$

We now introduce the setting of the filtering problem discussed in this paper. Let us fix a complete probability space $\probsp$. The $I$-valued unobserved process defined on this space is a continuous-time homogeneous pure jump Markov process. We are given the law $\mu$ of $X_0$. We can equivalently describe it (see \eg \citep{brandt:MPP}) by recording its jump times and jump locations, \ie by defining for all $n \in \N_0$ the random variables $T_n \colon \Omega \to [0, +\infty]$ and $\xi_n \colon \Omega \to I$, where for each $n \in \N$
\begin{align}
T_0(\omega) &\coloneqq 0, & T_n(\omega) &\coloneqq \inf\{t > T_{n-1}(\omega) \colon X_t(\omega) \ne X_{T_{n-1}(\omega)}(\omega)\}, \label{eq:Xjumptimes} \\
\xi_0(\omega) &\coloneqq X_0(\omega), & \xi_n(\omega) &\coloneqq X_{T_n(\omega)}(\omega).
\end{align}
We denote by $\bbX \coloneqq (\cX_t)_{t \geq 0}$ the natural completed filtration associated to the unobserved process, where $\cX_t \coloneqq \sigma(X_s \colon 0 \leq s \leq t) \vee \cN$ for each $t \geq 0$ and $\cN$ is the collection of $\bbP$-null sets in $\cF$. Its dynamics are described by a \emph{rate transition measure} $\lambda$, \ie a transition kernel from $(I, \cI)$ into itself such that for all $n \in \N_0$ and all $t \geq 0$
\begin{equation}\label{eq:ratetransmeas}
\bbP(T_{n+1} - T_n > t, \, \xi_{n+1} \in A \mid \cX_{T_n}) = \frac{\lambda(\xi_n, A)}{\lambda(\xi_n, I)}e^{-\lambda(\xi_n, I) t}, \quad \bbP\text{-a.s.}
\end{equation}
To have a more synthetic notation it is convenient to define the \emph{jump rate function} $\lambda \colon I \to [0,+\infty)$ as $\lambda(x) \coloneqq \lambda(x,I), \, x \in I.$
It will always be clear to which object the notation $\lambda$ will refer to.

The following Assumption will be in force throughout this Section and ensures some important facts about the process $X$ that we will recall later on.
\begin{assumption}\label{hp:lambda}
	The jump rate function $\lambda$ satisfies $\sup_{x \in I} \lambda(x) < +\infty$.
\end{assumption}
\begin{rem}
	It should be noted that the definition of rate transition measure given in (\ref{eq:ratetransmeas}) implies that $\lambda(x, \{x\}) = 0$ for all $x \in I$. This is evident by looking at the definition of the jump times appearing in (\ref{eq:Xjumptimes}).
\end{rem}

We assume that the $O$-valued observed process is a function of the signal process via a given measurable function $h \colon I \to O$, \ie it satisfies
\begin{equation}
Y_t(\omega) = h\bigl(X_t(\omega)\bigr), \quad \omega \in \Omega, \, t \geq 0.
\end{equation}
This means that in our setting the observation is \emph{noise-free}: the only source of randomness is represented by the unobserved process and no exogenous noise is acting on the observation.
We exclude the cases where the function $h$ is one to one or constant. In the first case we would be able to exactly recover the values assumed by the unobserved process, while in the second one the observation would give no information whatsoever about the signal. We assume, without loss of generality, that the function $h$ is surjective.

It is straightforward to notice that also in this case we can equivalently describe the process $Y$ by defining for all $n \in \N_0$ the random variables $\tau_n \colon \Omega \to [0, +\infty]$ and $\eta_n \colon \Omega \to O$, where for each $n \in \N$
\begin{align}
\tau_0(\omega) &\coloneqq 0, & \tau_n(\omega) &\coloneqq \inf\{t > \tau_{n-1}(\omega) \colon Y_t(\omega) \ne Y_{\tau_{n-1}(\omega)}(\omega)\}, \label{eq:Yjumptimes} \\
\eta_0(\omega) &\coloneqq Y_0(\omega), & \eta_n(\omega) &\coloneqq Y_{\tau_n(\omega)}(\omega).
\end{align}
We denote by $\bbY \coloneqq (\cY_t)_{t \geq 0}$ the natural completed filtration associated to the observed process, where $\cY_t \coloneqq \sigma(Y_s \colon 0 \leq s \leq t) \vee \cN$ for each $t \geq 0$. Finally, we define the \emph{explosion points} of the processes $X$ and $Y$ as the random variables
\begin{align}
T_\infty(\omega) &\coloneqq \lim_{n \to \infty} T_n(\omega), & \tau_\infty(\omega) &\coloneqq \lim_{n \to \infty} \tau_n(\omega).
\end{align}

\subsection{The filtering equation}
To tackle the noise-free filtering problem previously described, we will adopt an innovation approach, basing our analysis on the fact that we can represent $(X,Y)$ as a pair of Marked Point Processes (or MPPs for short). These processes are countable collections of pairs of random variables, describing the occurrence of some random events by recording the time and a related mark  (for more details see \eg \citep{jacod:mpp}, \citep{bremaud:pp}, \citep{brandt:MPP}, \citep{jacobsen:pointprocesstheory}). In this Section we will use some results contained in \citep[Ch. 4]{bremaud:pp} regarding filtering with point process observation.

It is immediate to see that the pairs $(T_n, \xi_n)_{n \in \N}$ and $(\tau_n, \eta_n)_{n \in \N}$ are MPPs. Moreover, thanks to Assumption \ref{hp:lambda} they are $\bbP$-a.s. non-explosive, \ie we have that $T_\infty(\omega) = +\infty$ and $\tau_\infty(\omega) = +\infty, \, \bbP$-a.s.\,.
Together with the initial conditions $\xi_0 = X_0$ and $\eta_0 = Y_0$, they describe completely the unobserved process $X$ and the observed process $Y$ respectively, so when speaking of the signal or data process we can equivalently use the jump process formulation or the MPP one. A third useful description of these two processes is possible. Let us define the following \emph{random counting measures} (or RCMs for short)
\begin{align}
n(\omega, \dd t \, \dd z) &\coloneqq \sum_{n \in \N} \delta_{\bigl(T_n(\omega), \, \xi_n(\omega)\bigr)} (\dd t \, \dd z) \ind_{\{T_n(\omega) < +\infty\}}, \quad \omega \in \Omega \\
m(\omega, \dd t \, \dd y) &\coloneqq \sum_{n \in \N} \delta_{\bigl(\tau_n(\omega), \, \eta_n(\omega)\bigr)} (\dd t \, \dd y) \ind_{\{\tau_n(\omega) < +\infty\}}, \quad \omega \in \Omega \label{eq:mMPP}
\end{align}
where, for any arbitrary point $a$ in some measurable space, $\delta_a$ denotes the Dirac probability measure concentrated at $a$.
For each fixed $\omega \in \Omega$, $n$ is a measure on $\bigl((0, +\infty) \times I, \cB\bigl((0, +\infty)\bigr) \otimes \cI\bigr)$ and is associated to the unobserved process $X$, while $m$ is a measure on $\bigl((0, +\infty) \times O, \cB\bigl((0, +\infty)\bigr) \otimes \cO\bigr)$ and is associated to the observed process $Y$. Random counting measures are particularly useful in connection with their \emph{dual predictable projections} (see \eg \citep[Th. 2.1]{jacod:mpp} for more details), also called \emph{compensators}.

It is a known fact that the $\bbX$-compensator of the RCM $n$ is given by the predictable random measure $\lambda(X_{t-}(\omega), \, \dd z)\, \dd t$, since it is associated to a pure jump Markov process with known rate transition measure. We will denote by $\tilde n$ the corresponding \emph{compensated random measure}, \ie
\begin{equation}\label{eq:Xcompensator}
\tilde n(\omega, \dd t \, \dd z) \coloneqq n(\omega, \dd t \, \dd z) - \lambda(X_{t-}(\omega), \, \dd z)\, \dd t.
\end{equation}
It is important to recall that for all $A \in \cI$ the compensated process $\tilde n\bigl((0, t] \times A)$ is a $\bbX$-martingale.

A preliminary step to the solution of our filtering problem consists in computing the $\bbX$- and $\bbY$-compensators of the RCM $m$.
\begin{lemma}\label{lemma:dualpredproj}
	The $\bbX$- and $\bbY$-dual predictable projections of the random counting measure $m$ are respectively given by the predictable random measures $\mu_t(\omega, \dd y) \, \dd t$ and $\hat\mu_t(\omega, \dd y) \, \dd t$,	where for all $B \in \cO$ and $t > 0$
	\begin{align}
	\mu_t(\omega, B) &\coloneqq \lambda\bigl(X_{t-}(\omega), \, h^{-1}(B \setminus \{Y_{t-}(\omega)\})\bigr) \label{eq:predprojX} \\
	\hat\mu_t(\omega, B) &\coloneqq \int_I \lambda\bigl(x, \, h^{-1}(B \setminus \{Y_{t-}(\omega)\})\bigr) \, \pi_{t-}(\omega \, ; \, \dd x) \label{eq:predprojY}
	\end{align}
	and $\pi = (\pi_t)_{t \geq 0}$ is the filtering process defined in (\ref{eq:filtproc}).
\end{lemma}

\begin{proof}
	Let $B \in \cO$ and $t > 0$ be fixed. Then, we can write:
	\begin{align*}
	&{} m\bigl(\omega, (0, t] \times B\bigr) 
	= \sum_{0 < s \leq t} \ind_{\{Y_s \ne Y_{s-}, \, Y_s \in B\}}(\omega) \\
	&= \sum_{0 < s \leq t} \ind_{\{X_s \in h^{-1}(Y_{s-})^c, \, X_s \in h^{-1}(B)\}}(\omega) \\
	&= \int_{(0,t] \times I} \ind_{h^{-1}(B \setminus \{Y_{s-}(\omega)\})}(z) \, n(\omega, \dd s \, \dd z). 
	\end{align*}
	
	The random field $\bigl(\ind_{h^{-1}(B \setminus \{Y_{s-}(\omega)\})}(z)\bigr)_{s \in (0,t], \, z \in I}$ is bounded and $\bbX$-predictable and under Assumption \ref{hp:lambda} we have that
	\begin{equation*}
	\bbE \!\!\!\int\limits_{(0,t] \times I} \ind_{h^{-1}(B \setminus \{Y_{s-}\})}(z) \, \lambda(X_{s-}, \dd z) \, \dd s = \bbE \int\limits_{(0,t]} \lambda\bigl(X_{s-}, h^{-1}(B \setminus \{Y_{s-}\})\bigr) \, \dd s < \infty.
	\end{equation*}
	Therefore, by \citep[Prop. 5.3]{jacod:mpp}
	\begin{equation*}
	\int\limits_{(0,t] \times I} \ind_{h^{-1}(B \setminus \{Y_{s-}\})}(z) \, \tilde n(\dd s \, \dd z) = m\bigl((0, t] \times B\bigr) - \int\limits_{(0,t]} \lambda(X_{s-}, \, h^{-1}(B \setminus \{Y_{s-}\}))\, \dd s
	\end{equation*}
	is a $\bbX$-local martingale. The first part of the claim follows from \citep[(2.6)]{jacod:mpp} since
	\begin{equation*}
	\biggl(\int_{(0,t] \times B} \mu_s(\dd z) \, \dd s\biggr)_{t \geq 0} = \biggl( \int_{(0,t]} \lambda(X_{s-}, \, h^{-1}(B \setminus \{Y_{s-}\}))\, \dd s \biggr)_{t \geq 0}
	\end{equation*}
	is a $\bbX$-predictable process for all $B \in \cO$.
	
	What we have just shown is that for all $\bbX$-predictable and non-negative random fields $C = (C_t(z))_{t \geq 0, \, z \in I}$ it holds
	\begin{equation*}
	\bbE \int_{(0, +\infty) \times I} C_s(z) m(\dd s \, \dd z) = \bbE \int_{(0, +\infty) \times I} C_s(z) \mu_s(\dd z) \, \dd s.
	\end{equation*}
	To prove the second part of the claim we have to show that the same kind equality holds for all $\bbY$-predictable and non-negative random fields $C$, with $\hat \mu$ replacing $\mu$ on the right hand side. If $C$ is such a random field we have that
	\begin{align*}
	&{ } \bbE \int_{(0, +\infty) \times I} C_s(z) \mu_s(\dd z) \, \dd s = \bbE \int_{(0, +\infty)} \int_I C_s(z) \ind_{h^{-1}(Y_{s-})^c}(z) \lambda(X_{s-}, \, \dd z) \, \dd s \\
	&=\bbE \int_{(0, +\infty)} \int_I \int_I C_s(z) \ind_{h^{-1}(Y_{s-})^c}(z) \lambda(x, \, \dd z) \, \pi_{s-}(\dd x) \, \dd s = \bbE \int_{(0, +\infty) \times I} C_s(z) \hat \mu_s(\dd z) \, \dd s
	\end{align*}
	This is justified by repeatedly using the Fubini-Tonelli theorem and the freezing lemma.
	
	Now, thanks to \citep[(2.4)]{jacod:mpp} it suffices to take $C_s(z) = C_s \ind_B(z)$, for any fixed $B \in \cO$, with $(C_t)_{t \geq 0}$ a $\bbY$-predictable process. Then we get that
	\begin{equation*}
	\bbE \int_{(0, +\infty)} C_s m(\dd s \times B) = \bbE \int_{(0, +\infty)} C_s \hat \mu_s(B) \, \dd s
	\end{equation*}
	holds for all $\bbY$-predictable processes $(C_t)_{t \geq 0}$. Since $\displaystyle \biggl(\int_{(0, t]} C_s \hat \mu_s(B) \, \dd s\biggr)_{t \geq 0}$ is $\bbY$-predictable for all $B \in \cO$, we get that this process is the $\bbY$-dual predictable projection of $m\bigl((0,t] \times B\bigr)$, for all $B \in \cO$, and the second part of the claim follows immediately.
\end{proof}


Another important step to solve the filtering problem is to represent the process to be filtered (in this case $\phi(X_t)$, for some $\phi \in \dB_b(I)$) as a semimartingale and then use a martingale representation theorem to obtain an expression for the filtering process $\pi(\phi)$.

Let us fix $\phi \in \dB_b(I)$. A semimartingale representation for $\phi(X_t)$ is easily obtained by using Dynkin's formula
\begin{equation} \label{eq:semimartrepr}
\phi(X_t) = \phi(X_0) + \int_0^t \cL\phi(X_s) \, \dd s + M_t, \quad t \geq 0
\end{equation}
where $\cL$ is the infinitesimal generator associated to the process $X$, \ie
\begin{equation}\label{eq:infinitesimalgen}
\cL f(x) \coloneqq \int_I \bigl[f(z) - f(x)\bigr]\lambda(x, \, \dd z), \quad f \in \dB_b(I)
\end{equation}
and $(M_t)_{t \geq 0}$ is a $\bbX$-martingale whose expression is given by
\begin{equation} \label{eq:Dynkinmartingale}
M_t = \int_{(0, t] \times I} \bigl[\phi(z) - \phi(X_{s-})\bigr]\tilde n(\dd s \, \dd z), \quad t \geq 0.
\end{equation}

The following Lemma is required to check that the assumptions needed to get the expression for the filtering process provided in the next Proposition are satisfied. Its proof is omitted, since it is based on routine computations.
\begin{lemma}\label{lemma:filteringhp}
	Let $\phi \in \dB_b(I)$ be fixed. The next properties hold true:
	\begin{enumerate}
		\item The process $\bigl(\cL\phi(X_t)\bigr)_{t \geq 0}$ is $\bbX$-progressive and for all $t > 0$
		$$\bbE \int_0^t \Bigl| \cL\phi(X_s) \Bigr| \, \dd s < \infty.$$
		\item The $\bbX$-martingale $(M_t)_{t \geq 0}$ is such that for all $t > 0$ 
		$$\bbE \int_{(0,t]} \Bigl| \dd M_s \Bigr| < \infty.$$
		\item The process $\bigl(\phi(X_t)\bigr)_{t \geq 0}$ is bounded.
	\end{enumerate}
\end{lemma}

\begin{proposition}
	Let $\phi \in \dB_b(I)$ be fixed. Then, for all $t \geq 0$, $\pi_t(\phi) = \bbE[\phi(X_t) \mid \cY_t]$ satisfies $\bbP$-a.s. the following \emph{Kushner-Stratonovich equation}:
	\begin{equation} \label{eq:filtering1}
	\pi_t(\phi) = \pi_0(\phi) + \int_0^t \int_I \cL\phi(x) \, \pi_{s-}(\dd x) \, \dd s + \int_{(0,t] \times O} \biggl\{\Psi_s(y) - \pi_{s-}(\phi)\biggr\} \, \tilde m(\dd s \, \dd y).
	\end{equation}
	Here $\tilde m(\dd s \, \dd y) = m(\dd s \, \dd y) - \hat \mu_s(\dd y) \, \dd s$, $\bigl(\Psi_t(y)\bigr)_{t \geq 0, \, y \in O}$ is the $\bbY$-predictable random field defined as
	\begin{equation}\label{eq:Psifield}
		\Psi_t(\omega, y) \coloneqq \frac{\psi_t(\omega; \, \dd y) \, \dd t \, \bbP(\dd \omega)}{\hat \mu_t(\omega; \, \dd y) \, \dd t \, \bbP(\dd \omega)}
	\end{equation}
	and $\psi_t(\omega, \, \dd y) \, \dd t$ is the $\bbY$-predictable random measure given by
	\begin{equation*}
	\psi_t(\omega, B) \, \dd t = \int_I \int_I \phi(z) \ind_{h^{-1}(B \setminus \{Y_{t-}(\omega)\})}(z) \lambda(x, \, \dd z) \pi_{t-}(\omega; \, \dd x) \, \dd t, \quad B \in \cO.
	\end{equation*}
\end{proposition}

\begin{proof}
	Thanks to Lemma \ref{lemma:filteringhp} we can apply \citep[Ch. VIII, Th. T9]{bremaud:pp} and we can write
	\begin{equation*}
	\pi_t(\phi) = \pi_0(\phi) + \int_0^t \hat f_s \, \dd s + \hat m_t, \quad t \geq 0, \quad \bbP\text{-a.s.},
	\end{equation*}
	where $(\hat f_t)_{t \geq 0}$ is a $\bbY$-progressive version of $\bigl(\bbE[\cL\phi(X_t) \mid \cY_t]\bigr)_{t \geq 0}$ and $(\hat m_t)_{t \geq 0}$ is the $\bbY$-martingale given by
	\begin{equation*}
	\hat m_t = \int_{(0,t] \times O} \{\Psi^1_s(y) - \Psi^2_s(y) + \Psi^3_s(y)\} \, \tilde m(\dd s \, \dd y), \quad t \geq 0,
	\end{equation*}
	with the $\bbY$-predictable fields $\bigl(\Psi^i_t(y)\bigr)_{t \geq 0, \, y \in O}$, $i = 1,2,3$ defined by the following equalities:
	\begin{align}
	\bbE \int_0^t \int_O C_s(y) \,\Psi^1_s(y) \,  \hat\mu_s(\dd y) \, ds &=
	\bbE \int_0^t \int_O C_s(y) \, \phi(X_s) \, \mu_s(\dd y) \, \dd s, \label{eq:Psi1definition} \\
	\bbE \int_0^t \int_O C_s(y) \,\Psi^2_s(y) \,  \hat\mu_s(\dd y) \, \dd s &=
	\bbE \int_0^t \int_O C_s(y) \,\phi(X_s) \,  \hat \mu_s(\dd y) \, \dd s, \label{eq:Psi2definition} \\
	\bbE \int_{(0,t] \times O} C_s(y) \, \Psi^3_s(y) \, \hat\mu_s(\dd y) \, \dd s &=
	\bbE \int_{(0,t] \times O} C_s(y) \, [\phi(X_s) - \phi(X_{s-})] \, m(\dd s \, \dd y), \label{eq:Psi3definition}
	\end{align}
	holding for all $t \geq 0$ and all non-negative $\bbY$-predictable fields $\bigl(C_t(y)\bigr)_{t \geq 0, \, y \in O}$.
	
	It is clear that $\hat f_t = \int_I \cL\phi(x)\pi_t(\dd x)$, $t \geq 0$, as a straightforward computation shows.
	It is also immediate to notice that $\Psi^2_t(y) = \pi_{t-}(\phi)$, $t > 0$.
	
	We now proceed to compute $\Psi^3_t(y)$, $t \geq 0, \, y \in O$. We will see that it is not necessary to compute the term $\Psi^1_t(y)$.
	Let us elaborate the right hand side of (\ref{eq:Psi3definition})
	\begin{align*}
	&\quad \, \, \bbE \int_{(0,t] \times O} C_s(y) \, [\phi(X_s) - \phi(X_{s-})] \, m(\dd s \, \dd y) \\ 
	&=\bbE \int_I C_s(h(z)) \, [\phi(z) - \phi(X_{s-})] \, \ind_{h^{-1}(Y_{s-})^c}(z) \, n(\dd s \, \dd z) \\
	&=\bbE \int_0^t \int_I C_s(h(z)) \, [\phi(z) - \phi(X_{s-})] \, \ind_{h^{-1}(Y_{s-})^c}(z) \, \lambda(X_{s-}, \, \dd z) \, \dd s
	\end{align*}
	where the former passage is justified by the same reasoning as in the beginning of the proof of Lemma \ref{lemma:dualpredproj} and the latter one is due to the fact that $\lambda(X_{s-}, \, \dd z) \, \dd s$ is the $\bbX$-compensator of $n(\dd s \, \dd z)$ and that the integrand is a $\bbX$-predictable process.
	
	It is easy to check that the term
	$$\bbE \int_0^t \int_I C_s(h(z)) \, \phi(X_{s-}) \, \ind_{h^{-1}(Y_{s-})^c}(z) \, \lambda(X_{s-}, \, \dd z) \, \dd s $$
	leads to the expression of the process $\bigl(\Psi^1_t(y)\bigr)_{t \geq 0, \, y \in O}$, obtainable by elaborating the right hand side of (\ref{eq:Psi1definition}). Hence defining $\Psi_t(\omega, y) = \Psi^1_t(\omega, y) + \Psi^3_t(\omega, y)$, $\omega \in \Omega, \, t \geq 0, \, y \in O$, we are left to characterize the following equality, holding for all $t \geq 0$ and all non-negative $\bbY$-predictable random fields $\bigl(C_t(y)\bigr)_{t \geq 0, \, y \in O}$.
	\begin{equation}\label{eq:Psiequality}
	\bbE \int_0^t \int_O C_s(y) \, \Psi_s(y) \, \hat\mu_s(\dd y) \, \dd s =
	\bbE \int_0^t \int_I C_s(h(z)) \, \phi(z) \, \ind_{h^{-1}(Y_{s-})^c}(z) \, \lambda(X_{s-}, \, \dd z) \, \dd s. 
	\end{equation}
	
	Applying the freezing lemma and the Fubini-Tonelli theorem to the right hand side of (\ref{eq:Psiequality}) and noticing that $X_s = X_{s-}$ and $\pi_s = \pi_{s-}, \, \dd s$-a.e., we get that
	\begin{align*}
	&\quad \, \, \bbE \int_0^t \int_I C_s(h(z)) \, \phi(z) \, \ind_{h^{-1}(Y_{s-})^c}(z) \, \lambda(X_{s-}, \, \dd z) \, \dd s \\
	&= \bbE \int_0^t \int_I \int_I C_s(h(z)) \, \phi(z) \, \ind_{h^{-1}(Y_{s-})^c}(z) \, \lambda(x, \, \dd z) \, \pi_{s-}(\dd x) \, \dd s \\
	&= \bbE \int_0^t \int_O C_s(y) \, \psi_s(\dd y) \, \dd s = \bbE \int_0^t \int_O C_s(y) \, \Psi_s(y) \, \hat\mu_s(\dd y) \, \dd s.
	\end{align*}
	Therefore, if we define the random field $\bigl(\Psi_t(y)\bigr)_{t \geq 0, \, y \in O}$ as in (\ref{eq:Psifield}) we get (\ref{eq:filtering1}). The random field $\bigl(\Psi_t(y)\bigr)_{t \geq 0, \, y \in O}$ is well defined and $\bbY$-predictable since $\psi_t(\omega; \, \dd y) \, \dd t \, \bbP(\dd \omega) \ll \hat \mu_t(\omega; \dd y) \, \dd t \, \bbP(\dd \omega)$, as it is straightforward to verify, and both are measures on $(\Omega \times (0,+\infty) \times O)$ equipped with the product $\sigma$-algebra given by the $\bbY$-predictable $\sigma$-field and $\cO$.
\end{proof}

Before stating an explicit equation for the filtering process $\pi(\phi)$, we need to define an operator, denoted by $H$. The notation adopted is due to \citep{confortola:filt}, where the authors discuss the case of a signal process $X$ given by a finite-state Markov chain. The aim is to characterize, for each $n \in \N$, the probability measure $\pi_{\tau_n}$, \ie the filtering process evaluated at each jump time $\tau_n$ of the observed process $Y$ (we will use $H$ also to characterize the initial value $\pi_0$). This will be done by identifying $\pi_{\tau_n}$ as a probability measure obtained via this operator and depending on the position $Y_{\tau_n}$ of the observed process at the $n$-th jump time and on a specific random measure. This measure will be determined by the values $\pi_{\tau_n-}$ and $Y_{\tau_n-}$ and by the rate transition measure $\lambda$. 

For the operator $H$ to be well defined, we need the following Proposition.

\begin{proposition}\label{prop:Hprobmeasure}
	Let $\mu \in \cM_+(I)$ be fixed. Then there exists a probability measure $\rho_y$ on $(I, \cI)$ such that for all $\phi \in \dB_b(I)$ and $\mu \circ h^{-1}$-almost all $y \in O$ it holds
	\begin{equation*}
	\frac{\phi \mu \circ h^{-1}(\dd \upsilon)}{\mu \circ h^{-1}(\dd \upsilon)}(y) = \int_I \phi(z) \, \rho_y(\dd z).
	\end{equation*}
	Moreover the set
	\begin{equation*}
	A \coloneqq \{z \in I \colon \rho_{h(z)}(G) = \ind_G(z), \, \forall G \in \cH\}, \quad \cH \coloneqq h^{-1}(\cO)
	\end{equation*}
	is such that $\mu(A^c) = 0$ and 
	\begin{equation*}
	\rho_y(h^{-1}(y)) = 1, \quad \mu \circ h^{-1}\text{-a.e.}
	\end{equation*}
\end{proposition}

\begin{proof}
	The scheme followed in this proof is the same used to prove the existence of a regular version of a conditional probability (see \eg \citep[Th. 89.1]{williams:diffusions}).
	To start, let us prove the first claim in the case where $I$ is a compact metric space. Thanks to Lemma~\ref{lemma:Hprobmeasure} (see Appendix~\ref{app:operatorH}) there exists a unique probability measure $\rho_y$ on $(I, \cI)$ such that
	\begin{equation*}
		\frac{\phi \mu \circ h^{-1}(\dd \upsilon)}{\mu \circ h^{-1}(\dd \upsilon)}(y) = \int_I \phi(z) \, \rho_y(\dd z), \quad \phi \in \dC(I)
	\end{equation*}	
	for all $y \in \supp(\mu \circ h^{-1})$. Clearly the set $\supp(\mu \circ h^{-1})^c$ has $\mu \circ h^{-1}$-measure zero and on this set we can define $\rho_y = \nu_y$, where $\nu_y$ is an arbitrary but fixed probability measure on $(I, \cI)$, such that $\nu_y\bigl(h^{-1}(y)\bigr) = 1$.
	Then, we get that for $\mu \circ h^{-1}$-almost all $y \in O$
	\begin{equation*}
		\frac{\phi \mu \circ h^{-1}(\dd \upsilon)}{\mu \circ h^{-1}(\dd \upsilon)}(y) = \int_I \phi(z) \, \rho_y(\dd z), \quad \phi \in \dC(I).
	\end{equation*}
	By a monotone class argument, the same equality holds for all $\phi \in \dB_b(I)$.
	
	To show that the second assertion holds, let us notice, first, that the $\sigma$-algebra $\cH$ is countably generated. In fact, since $O$ is a complete and separable metric space, its Borel $\sigma$-algebra $\cO$ can be written as $\cO = \sigma(C_1, C_2, \dots)$, for some countable collection $\cC = (C_i)_{i \in \N}$ of subsets of $O$. By a standard fact from measure theory (see \eg \citep[Corollary 1.2.9]{bogachev:measth}) we have that
	\begin{equation*}
		\cH = h^{-1}(\cO) =  h^{-1}\bigl(\sigma(\cC)\bigr) = \sigma\bigl(h^{-1}(\cC)\bigr)
	\end{equation*}
	hence the collection $h^{-1}(\cC) = \bigl(h^{-1}(C_1), h^{-1}(C_2), \dots\bigr)$ forms a countable class generating $\cH$.
	
	Now, let $\cK$ be the $\pi$-system formed by all finite intersections of sets in $h^{-1}(\cC)$ (notice that $\sigma(\cK) = \cH$) and define
	\begin{equation*}
		A_1 \coloneqq \{z \in I \colon \rho_{h(z)}(K) = \ind_K(z), \, \forall K \in \cK\}.
	\end{equation*}
	Fix $K \in \cK$. By definition of Radon-Nikodym derivative we have that for all $B \in \cO$
	\begin{multline*}
		\int_{h^{-1}(B)} \ind_K(z) \, \mu(\dd z) = \ind_K \mu \circ h^{-1} (B) = \int_B \int_I \ind_K(z) \, \rho_y(\dd z) \, \mu \circ h^{-1}(\dd y) = \\
		\int_B \rho_y(K) \, \mu \circ h^{-1}(\dd y) = \int_{h^{-1}(B)} \rho_{h(z)}(K) \, \mu(\dd z)
	\end{multline*}
	therefore $\rho_{h(z)}(K) = \ind_K(z), \, \mu$-a.e. on $\cH$.
	This equality holds for all $K \in \cK$ and since this is a countable collection of sets, we get that $\mu(A_1^c) = 0$.
	
	Next, for fixed $z \in A_1$, let $\cD \coloneqq \{G \in \cH \colon \rho_{h(z)}(G) = \ind_G(z)\}$. Obviously $I \in \cD$ and it is immediate to see that for any $D_1, D_2 \in \cD$ with $D_1 \subset D_2$
	\begin{equation*}
		\rho_{h(z)}(D_2 \setminus D_1) = \rho_{h(z)}(D_2) - \rho_{h(z)}(D_1) = \ind_{D_2}(z) - \ind_{D_1}(z) = \ind_{D_2 \setminus D_1}(z).
	\end{equation*}
	In addition, for any sequence $(D_n)_{n \in \N} \subset \cD, \, D_n \uparrow D$
	\begin{equation*}
			\rho_{h(z)}(D) = \lim_{n \to \infty} \rho_{h(z)}(D_n) = \lim_{n \to \infty} \ind_{D_n}(z) = \ind_D(z).
	\end{equation*}
	Hence, $\cD$ is a d-system, clearly containing the $\pi$-system $\cK$. Therefore, by Dynkin's $\pi-\lambda$ theorem we get that $\sigma(\cK) = \cH \subset \cD$. This implies that the equality
	\begin{equation*}
		\rho_{h(z)}(G) = \ind_G(z), \quad \text{for all } G \in \cH
	\end{equation*} 
	holds for all $z \in A_1$, hence $\mu(A^c) = 0$.
	
	Finally, fix $y \in O$. Clearly $h^{-1}(y) \in \cH$ and from the previous discussion we have that for $\mu$-almost all $z \in I$
	\begin{equation*}
		\rho_{h(z)}(h^{-1}(y)) = \ind_{h^{-1}(y)}(z) =
		\begin{dcases}
			1, &\text{if } z \in h^{-1}(y) \\
			0, &\text{if } z \notin h^{-1}(y)
		\end{dcases}
		.
	\end{equation*}
	Notice, also, that
	\begin{equation*}
		\rho_{h(z)}(h^{-1}(y)) =
		\begin{dcases}
			\rho_y(h^{-1}(y)),	&\text{if } z \in h^{-1}(y) \\
			\rho_\upsilon(h^{-1}(y)),	&\text{if } z \notin h^{-1}(y)
		\end{dcases}
	\end{equation*}
	for some $\upsilon \in O$, $\upsilon \ne y$.
	Therefore, $\rho_y(h^{-1}(y)) = 1$ for $\mu \circ h^{-1}$-almost all $y \in O$.
	
	To prove the claim in the case where $I$ is a complete and separable metric space, it suffices to remember (see \eg \citep[Theorem A.7]{bain:fundofstochfilt}) that $I$ is homeomorfic to a Borel subset of some compact metric space $J$ (in particular, $I \in \cB(J)$). After extending the measure $\mu$ to $(J, \cB(J))$ in the usual way, one can apply the result just shown to the measure space $(J, \cB(J), \mu)$, considering $\cH \coloneqq \sigma\bigl(h^{-1}(\cO)\bigr) \subset \cB(J)$. To conclude, it is enough to set $\rho_y$, for each $y \in O$, to be the restriction to $h^{-1}(\cO)$ of the probability measure found with the above procedure.
\end{proof}

Thanks to this result we can provide the following Definition of the operator $H$.

\begin{definition}[Operator $H$]\label{def:operatorH}
	For each $y \in O$ the operator $H_y \colon \cM_+(I) \to \cP(I)$ is given by
	\begin{equation*}
	H_y[\mu] \coloneqq 
	\begin{dcases}
	\rho_y,	&\text{if } y \in \supp(\mu \circ h^{-1})\\
	\nu_y,	&\text{if } y \notin \supp(\mu \circ h^{-1})
	\end{dcases}
	\end{equation*}
	where $\rho_y$ is the unique probability measure on $(I, \cI)$ satisfying
	\begin{equation*}
	\frac{\phi \mu \circ h^{-1}(\dd y)}{\mu \circ h^{-1}(\dd y)}(y) = \int_I \phi(z) \, \rho_y(\dd z), \quad \phi \in \dB_b(I)
	\end{equation*}
	and $\nu_y$ is an arbitrary probability measure on $(I, \cI)$, such that $\nu_y\bigl(h^{-1}(y)\bigr) = 1$.
\end{definition}

\begin{rem}
	The operator $H$ presents strong analogies with the regular version of a conditional probability, not only from a technical perspective as pointed out in the proof of Proposition~\ref{prop:Hprobmeasure}. Consider $(I, \cI) \eqqcolon (\Omega, \cF)$ as the measurable sample space; take $\mathbf{P}\in \cP(I)$ and define the sub-$\sigma$-algebra $h^{-1}(\cO) \eqqcolon \cG$.
	Then what we will denote by $H_{h(z)}[\mathbf{P}]$ (in this example, we should write $H_{h(\omega)}[\mathbf{P}]$) corresponds to a regular version of the conditional distribution $\mathbf{P}$ given $\cG$, \ie a function $(\mathbf{P} \mid \cG) \colon \cF \times \Omega \to [0,1]$ such that
	\begin{itemize}
		\item the map $\omega \mapsto (\mathbf{P} \mid \cG)(F, \omega)$ is a version of $\mathbf{P}(F \mid \cG)$ for all $F \in \cF$
		\item the function $F \mapsto (\mathbf{P} \mid \cG)(F, \omega)$ is a probability measure on $(\Omega, \cF)$ for all $\omega \in \Omega$.
	\end{itemize}
	 However, in the definition of the operator $H$ we do not use as "parameter" space (\ie the sample space) the set $I$, but the state space $O$ of the observed process. Moreover, the operator $H$ acts on the larger space $\cM_+(I)$.
\end{rem}

\begin{rem}\label{rem:operatorH}
	A more explicit definition of the operator $H$ can be obtained whenever this operator acts on a measure $\mu \in \cM_+(I)$ such that $\mu \circ h^{-1}$ is discrete. It is clear that in this case we have that
	$$H_y[\mu](\phi) = \frac{\phi \mu\bigl(h^{-1}(y)\bigr)}{\mu\bigl(h^{-1}(y)\bigr)} = \frac{1}{\mu\bigl(h^{-1}(y)\bigr)} \int_{h^{-1}(y)} \phi(z) \, \mu(\dd z), \quad y \in O$$
	under the usual assumption $\frac{0}{0} \coloneqq 0$.
	Otherwise said, the probability measure $H_y[\mu]$ is
	$$H_y[\mu](\dd z) = \frac{1}{\mu\bigl(h^{-1}(y)\bigr)} \ind_{h^{-1}(y)}(z) \, \mu(\dd z), \quad y \in O.$$
	As an example of such a setting, see \citep{calvia:optcontrol, confortola:filt}, where both the state spaces $I$ and $O$ of the unobserved and observed processes are assumed to be finite sets.
	
	This simplification can be interpreted in a Bayesian setting, by looking at this simple dominated model. We can view $(I, \cI)$ as the parameter space and $(O, \cO)$ as the data one. If we fix a \emph{prior distribution} $\mu$ on $(I, \cI)$ of our unknown parameter $X$, such that $\mu \circ h^{-1}$ is discrete, then we can interpret $\mu \circ h^{-1}$ as the \emph{likelihood}, \ie the law on $(O, \cO)$ of the datum $Y$ given $X$ and, finally, we can see $H_Y[\mu]$ as the \emph{posterior distribution} of $X$ given $Y$. 
	
	Unfortunately, this setting cannot be generalized to encompass the range of possible cases covered by our model and, as it is known, the Bayesian framework fails in a non-dominated setting. 
\end{rem}

We are now ready to state the final version of the filtering equation, giving the dynamics of the process $\pi(\phi)$.

\begin{theorem}[Filtering equation]\label{thm:filteringequation}
	Let $\phi \in \dB_b(I)$ be fixed. Let us define, for each fixed $y \in O$, the linear operator $\cA_y \colon \dB_b(I) \to \dB_b(I)$ as
	\begin{equation}\label{eq:operatorAy}
	\cA_y \phi(x) \coloneqq \cL \phi(x) - \int_I \ind_{h^{-1}(y)^c}(z) \phi(z) \lambda(x, \dd z), \quad x \in I
	\end{equation}
	and let us denote by $\mathsf{1} \colon I \to \R$ the function identically equal to $1$.
	
	The process $\pi(\phi)$ satisfies for all $t \geq 0$ and $\bbP$-a.s. the following equation
	\begin{equation} \label{eq:filteringintegral}
	\begin{split}
	\pi_t(\phi) &= H_{Y_0}[\mu](\phi) \\ &+
	\int_0^t \int_I \cA_{Y_{s-}} \, \phi(x) \, \pi_{s-}(\dd x) \, \dd s -
	\int_0^t \pi_{s-}(\phi) \int_I \cA_{Y_{s-}} \, \one(x) \, \pi_{s-}(\dd x) \, \dd s \\
	&+\sum_{0 < \tau_n \leq t} \biggl\{H_{Y_{\tau_n}}[\mu_n](\phi) - \pi_{\tau_n-}(\phi)\biggr\},
	\end{split}
	\end{equation}
	or, in differential form
	\begin{equation} \label{eq:filteringdifferential}
	\begin{dcases}
	\frac{\dd}{\dd t} \pi_t(\phi) = \pi_t(\cA_{Y_{\tau_n}}\phi) - \pi_t(\phi) \pi_t(\cA_{Y_{\tau_n}} \one), & t \in [\tau_n, \tau_{n+1}), \, n \in \N_0 \\
	\pi_{\tau_n}(\phi) = H_{Y_{\tau_n}}[\mu_n](\phi), & n \in \N_0
	\end{dcases}
	\end{equation}
	where the (random) measures $\mu_n$ on $(I, \cI)$ are given by
	\begin{equation}\label{eq:measuresmun}
	\mu_n(\dd z) \coloneqq
	\begin{dcases}
	\mu(\dd z), & n = 0\\
	\ind_{h^{-1}(Y_{\tau_n-})^c}(z) \int_I \lambda(x, \dd z) \, \pi_{\tau_n-}(\dd x), & n \in \N
	\end{dcases}.
	\end{equation}  
\end{theorem}

\begin{proof}
	To prove this Theorem it suffices to elaborate the terms appearing in (\ref{eq:filtering1}). We show that the integral form (\ref{eq:filteringintegral}) holds true. The differential form (\ref{eq:filteringdifferential}) follows immediately.
	
	Let us start by computing $\pi_0(\phi) = \bbE [\phi(X_0) \mid Y_0]$. By definition of conditional expectation the equality
	$$\bbE [Z \pi_0(\phi)] = \bbE [Z \phi(X_0)]$$
	holds for all bounded and $\sigma(Y_0)$ measurable random variables $Z$. Otherwise said, we have that
	$$\bbE[g(Y_0) f(Y_0)] = \bbE[g(Y_0) \phi(X_0)]$$
	for all bounded and measurable functions $g \colon O \to \R$, where $f \colon O \to \R$ is a measurable function such that $f(Y_0) = \pi_0(\phi), \, \bbP$-a.s. (notice that $\pi_0(\phi)$ is $\sigma(Y_0)$ measurable).
	Then we can write
	$$\bbE [g(Y_0) f(Y_0)] = \int_I g\bigl(h(x)\bigr) f\bigl(h(x)\bigr) \, \mu(\dd x) = \int_O g(y) f(y) \, \mu \circ h^{-1}(\dd y)$$
	on one hand. On the other hand
	$$\bbE[g(Y_0) \phi(X_0)] = \int_I g\bigl(h(x)\bigr) \phi(x) \, \mu(\dd x) = \int_O g(y) \, \phi \mu \circ h^{-1}(\dd y).$$
	Therefore
	$$\int_O g(y) f(y) \, \mu \circ h^{-1}(\dd y) = \int_O g(y) \, \phi \mu \circ h^{-1}(\dd y)$$
	for all bounded and measurable functions $g \colon O \to \R$, whence
	$$f(y) = \frac{\phi \mu \circ h^{-1}(\dd y)}{\mu \circ h^{-1}(\dd y)}(y) = \int_I \phi(z) \, H_y[\mu](\dd z) = H_y[\mu](\phi), \quad y \in O$$
	and finally $\pi_0(\phi) = f(Y_0) = H_{Y_0}[\mu](\phi), \, \bbP$-a.s.\,.
	
	Let us now analyze the term
	$$ \int_{(0,t] \times O} \biggl\{\Psi_s(y) - \pi_{s-}(\phi)\biggr\} \, \bigl[ m(\dd s \, \dd y) - \hat \mu_s(\dd y) \, \dd s \bigr] $$
	appearing in (\ref{eq:filtering1}).
	
	From the definition of the field $\bigl(\Psi_t(y)\bigr)_{t \geq 0, \, y \in O}$ given in (\ref{eq:Psifield}) and recalling that $\cL \, \one = 0$, we easily get that $\bbP$-a.s.
	\begin{align*}
		&{ }\int_0^t \int_O \biggl\{\Psi_s(y) - \pi_{s-}(\phi)\biggr\} \, \hat \mu_s(\dd y) \, \dd s \\
		&= \int_0^t \int_I \biggl[\int_I \ind_{h^{-1}(Y_{s-})^c}(z) \phi(z) \lambda(x, \dd z) - \pi_{s-}(\phi) \cA_{Y_{s-}} \, \one(x) \biggr] \pi_{s-}(\dd x) \, \dd s.
	\end{align*}
	Therefore $\bbP$-a.s.
	\begin{align*}
	&\int_0^t \int_I \cL\phi(x) \, \pi_{s-}(\dd x) \, \dd s - \int_{(0,t] \times O} \biggl\{\Psi_s(y) - \pi_{s-}(\phi)\biggr\} \, \hat \mu_s(\dd y) \, \dd s \\
	= &\int_0^t \int_I \cA_{Y_{s-}} \, \phi(x) \, \pi_{s-}(\dd x) \, \dd s -
	\int_0^t \pi_{s-}(\phi) \int_I \cA_{Y_{s-}} \, \one(x) \, \pi_{s-}(\dd x) \, \dd s.
	\end{align*}
	
	We are left to elaborate the term
	$$\int_{(0,t] \times O} \biggl\{\Psi_s(y) - \pi_{s-}(\phi)\biggr\} \, m(\dd s \, \dd y) = \sum_{0 < \tau_n \leq t} \biggl\{\Psi_{\tau_n}(Y_{\tau_n}) - \pi_{\tau_n-}(\phi)\biggr\}.$$
	Let us recall that the $\bbY$-predictable random field $\bigl(\Psi_t(y)\bigr)_{t \geq 0, \, y \in O}$ satisfies
	$$\bbE \int_0^t \int_O C_s(y) \, \psi_s(\dd y) \, \dd s = \bbE \int_0^t \int_O C_s(y) \, \Psi_s(y) \, \hat\mu_s(\dd y) \, \dd s$$
	for all non-negative $\bbY$-predictable random fields $\bigl(C_t(y)\bigr)_{t \geq 0, \, y \in O}$.
	A simple computation involving just the Fubini-Tonelli theorem shows that we can rewrite the previous equation as
	$$\bbE \int_0^t \int_O C_s(y) \, \phi \nu_t \circ h^{-1}(\dd y) \, \dd s = \bbE \int_0^t \int_O C_s(y) \, \Psi_s(y) \, \nu_t \circ h^{-1}(\dd y) \, \dd s$$
	where the $\bbY$-predictable random measure $\nu_t(\omega; \, \dd z) \, \dd t$ is given by
	$$\nu_t(\dd z) \, \dd t = \ind_{h^{-1}(Y_{t-})^c}(z) \int_I \lambda(x, \dd z) \, \pi_{t-}(\dd x) \, \dd t.$$
	Therefore, we have that
	\begin{equation*}
		\Psi_t(\omega, y) = H_y[\nu_t(\omega)](\phi), \quad \nu_t(\omega) \circ h^{-1}(\dd y) \, \dd t \, \bbP(\dd \omega)\text{-a.e.}
	\end{equation*}
	or, equivalently
	\begin{equation*}
		\Psi_t(\omega, y) = H_y[\nu_t(\omega)](\phi), \quad  m(\omega, \dd t \, \dd y) \, \bbP(\dd \omega)\text{-a.e.}
	\end{equation*}
	whence we deduce that $\Psi_{\tau_n}(Y_{\tau_n}) = H_{Y_{\tau_n}}[\mu_n](\phi), \, \bbP$-a.s, for all $n \in \N$.
\end{proof}

The differential form (\ref{eq:filteringdifferential}) of the filtering equation gives an important insight on the structure of the filtering process $\pi(\phi)$. 
In fact, in each time interval $[\tau_n, \tau_{n+1})$, $n \in \N_0$, the filtering process satisfies $\bbP$-a.s. a deterministic differential equation (observe that since $Y_t = Y_{\tau_n}$ for all $t \in [\tau_n, \tau_{n+1})$, the operator $\cA_{Y_t}$ is defined and fixed at each jump time $\tau_n$).
This will be a crucial fact in the characterization of $\pi$ as a PDP.

The final and most important Theorem of this Section shows that, starting from the filtering equation (\ref{eq:filteringintegral}) we can obtain an explicit equation for the measure-valued filtering process $\pi$. It provides the evolution equation satisfied by the filtering process $\pi$ on the space $\cP(I)$.

\begin{theorem}\label{thm:filteringproc}
	For each fixed $y \in O$ let $\cB_y \colon \cM(I) \to \cM(I)$ be the operator
	\begin{equation}\label{eq:operatorBy}
	\cB_y \nu(\dd z) \coloneqq \ind_{\pre(y)}(z) \int_I \lambda(x, \dd z) \, \nu(\dd x) - \lambda(z) \nu(\dd z), \quad \nu \in \cM(I).
	\end{equation}
	The filtering process $\pi = (\pi_t)_{t \geq 0}$ satisfies for all $t \geq 0$ and $\bbP$-a.s. the following SDE
	\begin{equation} \label{eq:filteringprocess}
	\begin{dcases}
	\frac{\dd}{\dd t} \pi_t = \cB_{Y_t} \pi_t - \pi_t \, \cB_{Y_t}\pi_t(I), \quad t \in [\tau_n, \tau_{n+1}), \, n \in \N_0 \\
	\pi_{\tau_n} = H_{Y_{\tau_n}}[\mu_n], \quad n \in \N_0
	\end{dcases}
	\end{equation}
	where the random measures $\mu_n, \, n \in \N_0$, were defined in (\ref{eq:measuresmun}).
\end{theorem}

\begin{proof}
	Let us notice first that from (\ref{eq:filteringdifferential}) we have that for all $n \in \N_0$ and $\P$-a.s.
	\begin{equation*}
	\pi_{\tau_n}(\phi) = \int_I \phi(z) \, H_{Y_{\tau_n}}[\mu_n](\dd z), \quad \P\text{-a.s.}, \quad \phi \in \dB_b(I).
	\end{equation*}
	Since we also have that $\pi_{\tau_n}(\phi) = \int_I \phi(z) \, \pi_{\tau_n}(\dd z), \, \P$-a.s., we get
	\begin{equation*}
		\pi_{\tau_n} = H_{Y_{\tau_n}}[\mu_n], \quad \P\text{-a.s.}, \, n \in \N_0.
	\end{equation*}
	
	We need to prove that the filtering process satisfies for each $n \in \N_0$ the ODE
	$$\frac{\dd}{\dd t} \pi_t = \cB_{Y_t} \pi_t - \pi_t \, \cB_{Y_t}\pi_t(I), \quad t \in [\tau_n, \tau_{n+1}), \, \P\text{-a.s.}$$
	It suffices to show that for each fixed $y \in O$ the operator $\cB_y$ is the restriction to $\cM(I)$ of the adjoint $\cA_y^\star$ of the operator $\cA_y$ introduced in Theorem \ref{thm:filteringequation}. To see this, we have to recall that the dual space $\dB_b(I)^\star$ of $\dB_b(I)$ is isometrically isomorfic to the space $\mathrm{ba}(I)$ of bounded finitely additive regular measures defined on the algebra generated by open sets in $I$. Notice that $\cM(I) \subset \mathrm{ba}(I)$. 
	
	Denote by $\inprod{\phi}{\nu} \coloneqq \int_I \phi(x) \, \nu(\dd x)$ the duality pairing between $\phi \in \dB_b(I)$ and $\nu \in \dB_b(I)^\star$ and fix $n \in \N_0$. Then (\ref{eq:filteringdifferential}) can be written as
	\begin{equation*}
	\inprod{\phi}{\dot \pi_t} = \inprod{\cA_{Y_t} \phi}{\pi_t} - \inprod{\phi}{\pi_t} \inprod{\cA_{Y_t} \one}{\pi_t}, \quad t \in [\tau_n, \tau_{n+1}), \, \P\text{-a.s.}\,.
	\end{equation*}
	The claim follows if we are able to show that
	\begin{equation*}
	\inprod{\phi}{\dot \pi_t} = 
	\inprod{\phi}{\cA_{Y_t}^\star \pi_t - \pi_t \cA_{Y_t}^\star \pi_t(I)}, \quad t \in [\tau_n, \tau_{n+1}), \, \P\text{-a.s.}
	\end{equation*}
	
	This fact follows from a repeated application of the \mbox{Fubini-Tonelli} theorem in the following chain of equalities, holding for all $\nu \in \cM(I)$ and $\phi \in \dB_b(I)$.
	\begin{align*}
	&\inprod{\cA_y\phi}{\nu} = 
	\int_I \cL\phi(x) \, \nu(dx) - \int_I \int_I \ind_{h^{-1}(y)^c} \phi(z) \lambda(x, dz) \, \nu(dx) =\\
	&\int_I \phi(z) \biggl\{ \ind_{h^{-1}(y)}(z) \int_I \lambda(x,dz) \, \nu(dx) - \lambda(z) \, \nu(dz) \biggr\} =
	\inprod{\phi}{\cB_y \nu}.
	\end{align*}	
	So, clearly $\cB_y = \cA_y^\star \rvert_{\cM(I)}$.
\end{proof}
	
\section{The filtering process as a PDMP} \label{sec:filterchar}
In this Section we want to investigate the properties of the filtering process $\pi$. The core of this Section will be devoted to prove that this is a \emph{Piecewise Deterministic Markov Process}, or PDMP for short. This class of processes, whose study has been started by M.H.A.~Davis (see \citep{davis:markovmodels} or \citep{jacobsen:pointprocesstheory}), has gained a lot of attention in applications, since it provides a framework to describe a vast range of phenomena whose behavior does not fit any diffusive model.

The main feature of PDMPs is that their dynamic, as the name suggests, is deterministic in specific time intervals, given by the occurrence of random jumps. In each of these time windows the evolution of the process is governed by a \emph{flow}, determined by a \emph{vector field}, satisfying an ODE. The distribution of the time passing between two consecutive jump times is given by an exponential-like law. The position of the process after a jump, i.e., its post jump location, is provided by another specified probability measure. Notice that, in general, a PDMP may jump whenever hitting the boundary of its state space. However, this will not happen in our problem.

It is clear that, in our situation, the filtering process $\pi$ appears to be a PDMP. The main task is to identify the characteristic triple that uniquely determines a PDMP, composed by the flow, the distribution of sojourn times and the law of the post jump locations.

Let us start by studying the flow. We are concerned with well-posedness
of the initial value problem described by equation (\ref{eq:filteringprocess}) between two consecutive jump times. From now on we consider the set $\cM(I)$ endowed with the total variation norm, making it a Banach space.
We define, for each fixed $y \in O$, the set $\Delta_y$ as the family of probability measures on $(I, \cI)$ concentrated on $h^{-1}(y)$, \ie
\begin{equation}\label{eq:deltaydefinition}
	\Delta_y \coloneqq \{\nu \in \cP(I) \colon \nu\bigl(h^{-1}(y)^c\bigr) = 0\}, \quad y \in O.
\end{equation}
This is a closed subset of $\cM(I)$ since it can be written as the intersection of the closed sets $\cP(I)$ and $\cM_y(I)$, where $\cM_y(I) \coloneqq \{\mu \in \cM(I) \colon \mu\bigl(h^{-1}(y)\bigr) = 1\}$. In particular, $\cM_y(I)$ is closed since the functional $\mu \mapsto \mu\bigl(h^{-1}(y)\bigr)$ is continuous on $\cM(I)$ for all $y \in O$.

Let us define, for each fixed $y \in O$ the vector field $F_y \colon \cM(I) \to \cM(I)$
\begin{equation}\label{eq:vectorfield}
 	F_y(\nu) \coloneqq \cB_y \nu - \nu \cB_y \nu(I), \quad \nu \in \cM(I), \, y \in O
\end{equation}
where $\cB_y$ is the operator defined in (\ref{eq:operatorBy}).
We already know that the solution to the following ODE
\begin{equation}\label{eq:filterODE}
\begin{dcases}
\frac{\dd}{\dd t} z_t = F_y(z_t), \quad t \geq 0 \\
z_0 = \rho, \quad \rho \in \Delta_y,
\end{dcases}
\end{equation}
exists for each fixed $y \in O$ (for instance, consider the fact that the filtering process $\pi$ satisfies it $\P$-a.s. in the time interval $[0, \tau_1)$ when $y = Y_0$).
Its uniqueness is established by the following Theorem. Its proof is provided in Appendix \ref{app:Fvectorfield}

\begin{theorem}\label{th:ODEuniqueness}
	Under Assumption \ref{hp:lambda}, for each fixed $y \in O$ the ODE (\ref{eq:filterODE}) admits a unique global solution $z \in \dC^1\bigl([0, +\infty); \Delta_y\bigr)$.
\end{theorem}

\begin{rem}\label{rem:flow}
	In what follows, we will denote the solution $z$ by $\varphi_{y, \rho}(\cdot)$, 
	to stress the dependence on $y \in O$ and $\rho \in \Delta_y$.
	By standard results, $(t, \rho) \mapsto \varphi_{y, \rho}(t)$ is continuous for each $y \in O$
	and it enjoys the flow property, \ie $\varphi_{y, \varphi_y(s, \rho)}(t) = \varphi_{y, \rho}(t+s)$, for $t, s \geq 0$.
	The function $y \mapsto \varphi_{y, \cdot}(\cdot)$ is called the \emph{flow} associated with the vector field $F_y$ on $\Delta_y$.
	To simplify the notation, it is convenient to define the set $\Delta_e = \bigcup_{y \in O} \Delta_y$, named the \emph{effective simplex} to preserve the terminology used in \citep{confortola:filt}.
	Notice that the union is disjoint, as is immediate to prove from the definition given in (\ref{eq:deltaydefinition}). In this way we can define a \emph{global flow} $\varphi$ on $\Delta_e$ setting $\varphi_\rho(t) = \varphi_{y, \rho}(t)$, if $\rho \in \Delta_y$.
	For all fixed $t \geq 0, \, \rho \mapsto \varphi_\rho(t)$ is a function mapping $\Delta_e$ into itself and leaving each set $\Delta_y$ invariant.
	Finally, we can associate to the global flow a \emph{global vector field} $F \colon \Delta_e \to \Delta_e$ defined as
	\begin{equation}\label{eq:globalvectorfield}
		F(\nu) \coloneqq F_y(\nu) = \cB_y \nu - \nu \, \cB_y \nu(I), \quad \nu \in \Delta_y.
	\end{equation}
	The effective simplex bears this name because of its relationship with the canonical simplex on euclidean spaces. In fact, if we consider the state spaces $I$ and $O$ of the signal and observed processes as finite sets (see \eg \citep{calvia:optcontrol, confortola:filt}), then the effective simplex is made of pairwise disjoint faces of the canonical simplex on $\R^{\abs{I}}$, where $\abs{I}$ denotes the cardinality of the set $I$. The shape of these faces (points, segments, triangles, tetrahedra, etc \dots) depends on the function $h$. The evolution of the filtering process takes place only on parts of the boundary of the canonical simplex.
\end{rem}

Before moving on to prove the characterization of the filtering process as a PDMP, let us precise that, as far as topology is concerned, the effective simplex will be regarded as a topological space under the relative topology inherited from the Banach space $\cM(I)$. In this way, we can also consider the effective simplex as a measurable space, endowing it with the Borel $\sigma$-algebra $\cB(\Delta_e)$.

In order to prove that the filtering process is a PDMP, it is convenient to put ourselves in a canonical setting for our filtering problem with respect to the unobserved process $X$. This construction will have a fundamental role in studying the optimal control problem.

Let us define $\Omega$ as the set
\begin{multline*}
\Omega = \{\omega = (i_0, t_1, i_1, t_2, i_2, \ldots) \colon i_0 \in I, i_n \in I, \\ t_n \in (0, +\infty], t_n \leq t_{n+1}, t_n < +\infty \Rightarrow t_n < t_{n+1}, n \in \N\}.
\end{multline*}
For each $n \in \N$ we introduce the following random variables 
\begin{align*}
T_0(\omega) &= 0, & T_n(\omega) &= t_n, & T_\infty(\omega) &= \lim_{n \to \infty} T_n(\omega), & \xi_0(\omega) &= i_0, & \xi_n(\omega) &= i_n,
\end{align*}
and we define the random measure on $\bigl((0, +\infty) \times I, \, \cB\bigl((0, +\infty)\bigr) \otimes \cI\bigr)$
\begin{equation*}
n(\omega, \dd t \, \dd z) = \sum_{n \in \N} \delta_{\bigl(T_n(\omega), \, \xi_n(\omega)\bigr)} (\dd t \, \dd z) \ind_{\{T_n < +\infty\}}(\omega), \quad \omega \in \Omega,
\end{equation*}
with associated natural filtration $\cN_t = \sigma\bigl(n\bigl((0, s] \times A \bigr)\bigr), \, 0 \leq s \leq t, \, A \in \cI\bigr)$.
Finally, let us specify the \mbox{$\sigma$-algebras}
\begin{align*}
\cX_0^\circ &= \sigma(\xi_0), & \cX_t^\circ &= \sigma(\cX_0^\circ \cup \cN_t), & \cX^\circ &= \sigma\Bigl(\bigcup_{t \geq 0} \cX_t\Bigr).
\end{align*}

The unobserved process $X$ is defined, for all $t \geq 0$, as
\begin{equation*}
X_t(\omega) =
\begin{cases}
\xi_n(\omega),	& t \in \bigl[T_n(\omega), T_{n+1}(\omega)\bigr), \, n \in \N_0, \, T_n(\omega) < +\infty \\
i_\infty,		& t \in \bigl[T_\infty(\omega), +\infty), \, T_\infty(\omega) < +\infty
\end{cases}
\end{equation*}
where $i_\infty \in I$ is an arbitrary state, that is irrelevant to specify.
Next, we define the observed process $Y$ and its natural filtration $(\cY_t^\circ)_{t \geq 0}$ as
\begin{align*}
Y_t(\omega) &= h(X_t(\omega)), \, t \geq 0, \, \omega \in \Omega; & \cY_t^\circ &= \sigma\bigl(Y_s,\, 0 \leq s \leq t\bigr), \, t \geq 0.
\end{align*} 
It is clear that we can equivalently describe this process (as is the case for $X$) via a MPP $(\eta_n, \tau_n)_{n \in \N}$ together with the initial condition $\eta_0 = h(\xi_0) = Y_0$.
Accordingly, the $\sigma$-algebras of the natural filtration of $Y$ are the smallest $\sigma$-algebras generated by the union of $\sigma(\eta_0)$ and the $\sigma$-algebras of the natural filtration of the MPP $(\eta_n, \tau_n)_{n \in \N}$.

\begin{rem}
	Notice that here we constructed the unobserved process $X$ starting from its MPP counterpart $(T_n, \xi_n)_{n \in \N}$, whereas in Section \ref{sec:filtering} we were given a pure jump Markov process and we associated its corresponding MPP. The two approaches are clearly equivalent.
\end{rem}

Next, for every $\mu \in \cP(I)$ let $\p_\mu$ be the unique probability measure on $(\Omega, \cX^\circ)$ such that $X$ is a $(\cX^\circ, \p_\mu)$-Markov process with state space $I$, initial law $\mu$ and generator $\cL$, defined in (\ref{eq:infinitesimalgen}). This means that for all $A \in \cI$, all $s, t \geq 0$ and all $f \in \dB_b(I)$ it holds
\begin{align*}
	\p_\mu(X_0 \in A) &= \mu(A), \quad \p_\mu\text{-a.s.}\\
	\e_\mu\bigl[f(X_{t+s}) \mid \cX_t^\circ\bigr] &= e^{s \cL}f(X_t), \quad \p_\mu\text{-a.s.}
\end{align*}
It follows from Assumption \ref{hp:lambda} and by standard arguments that the point process $n$ is $\p_\mu$-a.s. \mbox{non-explosive}, \ie that $T_\infty = +\infty$, \mbox{$\p_\mu$-a.s.}

To conclude the previous construction, for a fixed probability measure $\mu$ on $I$ we define
\begin{itemize}
	\item $\cX^\mu$ the \mbox{$\p_\mu$-completion} of $\cX^\circ$
	($\p_\mu$ is extended to $\cX^\mu$ in the natural way).
	\item $\cZ^\mu$ the family of elements of $\cX^\mu$
	with zero $\p_\mu$~-~probability.
	\item $\cY_t^\mu = \sigma(\cY_t^\circ, \cZ^\mu)$, for $t \geq 0$.
\end{itemize}
$(\cY_t^\mu)_{t \geq 0}$ is called the \emph{natural completed filtration} of $Y$.

Next, for each $\mu \in \cP(I)$ and $\omega \in \Omega$ we define $\pi^\mu(\omega)$ to be the unique solution to
\begin{equation}\label{eq:filteringproctraj}
\begin{dcases}
	\frac{\dd}{\dd t} \pi^\mu_t(\omega) = F_{Y_t(\omega)}\bigl(\pi^\mu_t(\omega)\bigr), \quad t \in [\tau_n(\omega), \tau_{n+1}(\omega)), \, n \in \N_0 \\
	\pi_0^\mu(\omega) = H_{Y_0(\omega)}[\mu] \\
	\pi_{\tau_n(\omega)}^\mu(\omega) = H_{Y_{\tau_n(\omega)}}\bigl[\Lambda\bigl(\pi_{\tau_n^-}^\mu(\omega)\bigr)\bigr], \, n \in \N
\end{dcases}
\end{equation}
where $\Lambda \colon \Delta_e \to \cM_+(I)$ is the function given by
\begin{equation}\label{eq:operatorLambda}
	\Lambda(\nu) \coloneqq \ind_{h^{-1}(y)^c}(z) \int_I \lambda(x, \dd z) \, \nu(\dd x), \quad \nu \in \Delta_y
\end{equation}
and the quantity $\pi_{\tau_n^-(\omega)}(\omega)$ is defined as
$$\pi_{\tau_n^-(\omega)}(\omega) \coloneqq \lim_{t \to \tau_n(\omega)^-} \pi_t(\omega), \quad \text{on } \{\omega \in \Omega \colon \tau_n(\omega) < +\infty\}.$$

Thanks to Theorem \ref{th:ODEuniqueness}, Equation (\ref{eq:filteringproctraj}) uniquely determines a $(\cY_t^\circ)_{t \geq 0}$-adapted c\'adl\'ag and $\Delta_e$-valued process. By Theorem \ref{thm:filteringproc}, we deduce that it is a modification of the filtering process, \ie for all $t \geq 0$ and all $A \in \cI$ it holds
\begin{equation*}
	\pi^\mu_t(A) = \p_\mu(X_t \in A \mid \cY^\mu_t), \quad \p_\mu\text{-a.s.}
\end{equation*}
Since the filtering process is $(\cY_t^\mu)_{t \geq 0}$-adapted and the filtration $(\cY_t^\mu)_{t \geq 0}$ is right-continuous, we can choose (and we will, whenever needed) a $(\cY_t^\mu)_{t \geq 0}$-progressive version of the filtering process itself.

We are now ready to state the Markov property for the filtering process $\pi^\mu$ with respect to the natural completed filtration of the observed process $Y$, for each fixed $\mu \in \cP(I)$. This is the content of Proposition \ref{prop:filtermarkov}, preceded by the useful technical Lemma \ref{lemma:prediction}. We omit their proof, being slight generalizations of \citep[Proposition 3.3 and Proposition 3.4]{confortola:filt}.

\begin{lemma}\label{lemma:prediction}
	For fixed $t \geq 0$, let us denote by $X_t^\infty$ the future trajectory of the process $X$ starting at time $t$. For all $\mu \in \cP(I), \, t \geq 0$ and $C \in \cX^\mu$, it holds
	\begin{equation*}
		\p_\mu(X_t^\infty \in C \mid \cY^\mu_t) = \p_{\pi_t^\mu}(C), \quad \p_\mu\text{-a.s.}
	\end{equation*}
\end{lemma}

\begin{proposition}\label{prop:filtermarkov}
	For fixed $t \geq 0$ consider the transition kernel $p_t$ from $\bigl(\Delta_e, \, \cB(\Delta_e)\bigr)$ into itself given by
	$$p_t(\nu,  D) \coloneqq \p_\nu(\pi_t^\nu \in D), \quad \nu \in \Delta_e, \, D \in \cB(\Delta_e).$$
	Then $(p_t)_{t \geq 0}$ is a Markov transition function on $\bigl(\Delta_e, \, \cB(\Delta_e)\bigr)$. Moreover, for every fixed $\mu \in \cP(I)$, the process $\pi^\mu$ in the probability space $(\Omega, \cX^\mu, \p_\mu)$ is a $\Delta_e$-valued Markov process with respect to $(\cY_t^\mu)_{t \geq 0}$, having transition function $(p_t)_{t \geq 0}$. Otherwise said, the following equality holds, for all $s, t \geq 0$ and all $D \in \cB(\Delta_e)$
	\begin{equation*}
		\p_\mu(\pi^\mu_{t+s} \in D \mid \cY_t^\mu) = p_s(\pi^\mu_t, D), \quad \p_\mu\text{-a.s.}
	\end{equation*}
\end{proposition}

To prove that the filtering process $\pi^\mu$ is a PDMP we need to define its \emph{characteristic triple}, \ie the following quantities.
\begin{enumerate}
	\item The \emph{vector field} $F$ is the global vector field defined in Equation (\ref{eq:globalvectorfield}), \ie
	\begin{equation}\label{eq:PDPvectorfield}
	F(\nu) \coloneqq \cB_y \nu - \nu \, \cB_y \nu(I), \quad \nu \in \Delta_y.
	\end{equation}
	The global flow $\varphi$ previously introduced is associated to this vector field (see Remark \ref{rem:flow}). 
	\item The \emph{jump rate function} $r \colon \Delta_e \to [0, +\infty)$ is defined by
	\begin{equation}\label{eq:PDPjumprate}
	r(\nu) \coloneqq - \cB_y \nu(I) = \int_I \lambda\bigl(x, h^{-1}(y)^c\bigr) \, \nu(\dd x), \quad \nu \in \Delta_y.
	\end{equation}
	\item The \emph{transition probability} $R$ from $\bigl(\Delta_e, \cB(\Delta_e)\bigr)$ into itself is defined by
	\begin{equation}\label{eq:PDPtransition}
	R(\nu, D) \coloneqq \int_O \ind_D\bigl(H_\upsilon[\Lambda(\nu)]\bigr) \, \rho(\nu, \dd \upsilon), \quad \nu \in \Delta_y, \, D \in \cB(\Delta_e)
	\end{equation}
	where $\rho$ is a transition probability from $\bigl(\Delta_e, \cB(\Delta_e)\bigr)$ into $(O, \cO)$ defined for all $\nu \in \Delta_y$ and all $B \in \cO$ as
	\begin{equation}\label{eq:rhotransition}
		\rho(\nu, B) \coloneqq 
		\begin{dcases}
			\frac{1}{r(\nu)} \int_I \lambda\bigl(x, h^{-1}(B \setminus \{y\})\bigr) \, \nu(\dd x), &\text{if } r(\nu) > 0 \\
			q_y(B), &\text{if } r(\nu) = 0 
		\end{dcases}
	\end{equation}
	where $(q_y)_{y \in O}$ is a family of probability measures, each concentrated on the level set $h^{-1}(y), \, y \in O$, whose exact values are irrelevant.
	
	Since for any given $\nu \in \Delta_y$ the probability $\rho(\nu, \cdot)$ is concentrated on the set $O \setminus \{y\}$, the probability $R(\nu, \cdot)$ is concentrated on $\Delta_e \setminus \Delta_y$, as we expected to be, given the structure of the filtering process.
\end{enumerate}

\begin{rem}\label{rem:rpositive}
	Note that if $r(\nu) > 0$ then $r(\nu) \, \rho(\nu, \dd y) = \Lambda(\rho) \circ h^{-1}(\dd y)$.
\end{rem}

The law of the observed process $Y$ can be expressed via the filtering process itself. This fact will be used to characterize the filtering process as a PDMP. It is clear that since the process $Y$ is piecewise constant, its law is uniquely determined by the finite dimensional distributions of the process $\{Y_0, \tau_1, Y_{\tau_1}, \dots\}$. These in turn are completely characterized by the law of $Y_0$, which is obvious, and by the following distributions
\begin{align*}
	&\p_\mu(\tau_{n+1} - \tau_n > t, \, \tau_n < +\infty \mid \cY_{\tau_n}^\mu), & &t \geq 0, \, n \in \N_0 \\
	&\p_\mu(Y_{\tau_{n+1}} \in B, \, \tau_{n+1} < +\infty \mid \cY_{\tau_n}^\mu), & &B \in \cO, \, n \in \N_0.
\end{align*}

\begin{proposition}\label{prop:lawY}
	For all fixed $\mu \in \cP(I)$ the distributions of the sojourn times and the post jump locations of the observed process $Y$ are given by the following equalities, holding $\p_\mu$-a.s. for all $t \geq 0, \, B \in \cO$ and all $n \in \N_0$
	\begin{align}
		\p_\mu(\tau_{n+1} - \tau_n > t, \, \tau_n < +\infty \mid \cY_{\tau_n}^\mu) &= \exp\biggl\{-\int_0^t r\bigl(\varphi_{\pi_{\tau_n}^\mu}(s)\bigr) \, \dd s\biggr\} \ind_{\tau_n < +\infty} \label{eq:Ysojourntimes}, \\
		\p_\mu(Y_{\tau_{n+1}} \in B, \, \tau_{n+1} < +\infty \mid \cY_{\tau_{n+1}^-}^\mu) &= \rho\bigl(\varphi_{\pi_{\tau_n}^\mu}(\tau_{n+1}^- - \tau_n), B\bigr) \ind_{\tau_{n+1} < +\infty}. \label{eq:Ypostjump}
	\end{align}
\end{proposition}

\begin{proof}
	Let us fix $n \in \N_0$.
	To start, we will look for an expression for the joint distribution of the jump times and post jump locations of the process $Y$, \ie for all $T \geq 0$ and all $B \in \cO$ the quantity
	\begin{equation*}
		\p_\mu(\tau_{n+1} \leq T, \, Y_{\tau_{n+1}} \in B \mid \cY_{\tau_n}^\mu), \quad \text{on } \{\tau_n < +\infty\}.
	\end{equation*}
	Notice, first, that for all fixed $T \geq 0$ and $B \in \cO$ we can write 
	$$Z_T(B) \coloneqq \ind_{\tau_{n+1} \leq T} \ind_{Y_{\tau_{n+1}} \in B} = m\bigl((0, T \land \tau_{n+1}] \times B\bigr) - m\bigl((0, T \land \tau_n] \times B\bigr)$$ where $m$ is the random counting measure associated to $Y$ and defined in (\ref{eq:mMPP}). To see this, it suffices to observe that
	\begin{equation*}
		m\bigl((0, T \land \tau_n] \times B\bigr) = \sum_{k = 1}^n \ind_{\tau_k \leq T} \ind_{Y_{\tau_k} \in B}.
	\end{equation*}
	Clearly, $(Z_T(B))_{T \geq 0}$ is a $(\cY_T^\mu)_{T \geq 0}$-adapted point process and, thanks to Lemma \ref{lemma:dualpredproj}, a straightforward computation shows that its $(\cY_T^\mu)_{T \geq 0}$-compensator is given by
	\begin{align*}
		\zeta_T(B) &\coloneqq \int_{T \land \tau_n}^{T \land \tau_{n+1}} \int_I \lambda\bigl(x, h^{-1}(B \setminus \{Y_{s-}\})\bigr) \, \pi_{s-}^\mu(\dd x) \, \dd s\\
		&= \int_0^T \ind_{\tau_n \leq s < \tau_{n+1}} \int_I \lambda\bigl(x, h^{-1}(B \setminus \{Y_s\})\bigr) \, \pi_s^\mu(\dd x) \, \dd s.
	\end{align*}
	Moreover, since for all $k \in \N$ the stopped process 
	$$\biggl(m\bigl((0, T \land \tau_k] \times B\bigr) - \int_{(0, T \land \tau_k]} \int_I \lambda\bigl(x, h^{-1}(B \setminus \{Y_{s-}\})\bigr) \, \pi_{s-}^\mu(\dd x) \, \dd s\biggr)_{T \geq 0}$$
	is a uniformly integrable $(\cY_T^\mu)_{T \geq 0}$-martingale (cfr. \citep[(2.6)]{jacod:mpp}), the compensated process $(Z_T(B) - \zeta_T(B))_{T \geq 0}$ is a uniformly integrable $(\cY_T^\mu)_{T \geq 0}$-martingale. Hence by applying Doob's optional sampling theorem we get that for all $T \geq 0$
	\begin{equation*}
		\e_\mu[Z_T(B) \mid \cY_{\tau_n}^\mu] = \e_\mu[\zeta_T \mid \cY_{\tau_n}^\mu], \quad \p_\mu\text{-a.s.}
	\end{equation*}
	or otherwise written
	\begin{multline*}
		\p_\mu(\tau_{n+1} \leq T, \, Y_{\tau_{n+1}} \in B \mid \cY_{\tau_n}^\mu) = \\
		\e_\mu\biggl[\int_0^T \ind_{\tau_n \leq s < \tau_{n+1}} \int_I \lambda\bigl(x, h^{-1}(B \setminus \{Y_s\})\bigr) \, \pi_s^\mu(\dd x) \, \dd s \biggm| \cY_{\tau_n}^\mu\biggr], \quad \p_\mu\text{-a.s.}
	\end{multline*}
	Noting that for $\tau_n \leq s < \tau_{n+1}$ we have that $Y_s = Y_{\tau_n}$ and $\pi_s = \varphi_{\pi_{\tau_n}^\mu}(s-\tau_n)$ we can write the previous equation as
	\begin{multline}\label{eq:jumpsojournlaw}
		\p_\mu(\tau_{n+1} \leq T, \, Y_{\tau_{n+1}} \in B \mid \cY_{\tau_n}^\mu) = \\
		\e_\mu\biggl[\int_0^T \int_I \ind_{\tau_n \leq s < \tau_{n+1}} \lambda\bigl(x, h^{-1}(B \setminus \{Y_{\tau_n}\})\bigr) \, \varphi_{\pi_{\tau_n}^\mu}(s-\tau_n; \,\dd x) \, \dd s \biggm| \cY_{\tau_n}^\mu\biggr], \, \p_\mu\text{-a.s.}
	\end{multline}
	
	Now let $t \geq 0$ be fixed. Since the random variable $t + \tau_n$ is $Y_{\tau_n}^\mu$-measurable, we immediately get from (\ref{eq:jumpsojournlaw})
	\begin{multline}\label{eq:sojournjumplaw}
		\p_\mu(\tau_{n+1} -\tau_n \leq t, \, Y_{\tau_{n+1}} \in B \mid \cY_{\tau_n}^\mu) = \\
		\e_\mu\biggl[\int_0^{t+\tau_n} \int_I \ind_{\tau_n \leq s < \tau_{n+1}} \lambda\bigl(x, h^{-1}(B \setminus \{Y_{\tau_n}\})\bigr) \, \varphi_{\pi_{\tau_n}^\mu}(s-\tau_n; \,\dd x) \, \dd s \bigm| \cY_{\tau_n}^\mu\biggr] = \\
		\e_\mu\biggl[\int_0^t \int_I \ind_{\tau_{n+1} - \tau_n > s} \lambda\bigl(x, h^{-1}(B \setminus \{Y_{\tau_n}\})\bigr) \, \varphi_{\pi_{\tau_n}^\mu}(s; \,\dd x) \, \dd s \bigm| \cY_{\tau_n}^\mu\biggr], \, \p_\mu\text{-a.s.}
	\end{multline}
	We need to exchange the conditional expectation and the time integral appearing in the last line of (\ref{eq:sojournjumplaw}). Let us consider a regular version $G^n$ of the conditional distribution $\p_\mu(\tau_{n+1} -\tau_n \in \cdot, Y_{\tau_n} \in \cdot \mid \cY_{\tau_n}^\mu)$, which always exists in our setting. Define the function $g^n \colon \Omega \times [0, +\infty) \to [0,1]$ as
	\begin{equation*}
		g^n(\omega, t) \coloneqq G^n\bigl(\omega, (t, +\infty], O\bigr).
	\end{equation*}
	Clearly $g^n$ enjoys the following properties
	\begin{itemize}
		\item $g^n(\omega, t) = \p_\mu(\tau_{n+1} - \tau_n > t \mid \cY_{\tau_n}^\mu), \, \p_\mu$-a.s.
		\item $g^n$ is $\bigl(\cY_{\tau_n}^\mu \otimes \cB\bigl([0, +\infty)\bigr)\bigr)$-measurable.
		\item For all $\omega \in \Omega$ the map $t \mapsto g^n(\omega, t)$ is non-increasing and right-continuous.
	\end{itemize}
	Applying the Fubini-Tonelli theorem to the last line of (\ref{eq:sojournjumplaw}) we get
	\begin{multline}
		\p_\mu(\tau_{n+1} -\tau_n \leq t, \, Y_{\tau_{n+1}} \in B \mid \cY_{\tau_n}^\mu) = G\bigl((0,t], B\bigr) \\
		\int_0^t g^n(s) \int_I \lambda\bigl(x, h^{-1}(B \setminus \{Y_{\tau_n}\})\bigr) \, \varphi_{\pi_{\tau_n}^\mu}(s; \,\dd x) \, \dd s, \quad \p_\mu\text{-a.s.}
	\end{multline}
	Thanks to \citep[Ch. VIII, Th. T7]{bremaud:pp}, we obtain that $g^n$ satisfies
	\begin{equation*}
		g^n(t) = 1 - \int_0^t g^n(s) \int_I \lambda\bigl(x, h^{-1}(Y_{\tau_n})^c\bigr) \, \varphi_{\pi_{\tau_n}^\mu}(s; \,\dd x) \, \dd s, \quad t \in (0, \tau_{n+1}-\tau_n].
	\end{equation*}
	This equality implies that on $(0, \tau_{n+1}-\tau_n]$ the function $g^n$ is absolutely continuous for each $\omega \in \Omega$ and solves $\omega$-by-$\omega$ the following ODE
	\begin{equation*}
		\begin{dcases}
			\frac{\dd}{\dd t} g^n(t) = -g^n(t) \int_I \lambda\bigl(x, h^{-1}(Y_{\tau_n})^c\bigr) \, \varphi_{\pi_{\tau_n}^\mu}(t; \,\dd x), & t \in (0, \tau_{n+1}-\tau_n] \\
			g_n(0) = 1
		\end{dcases}
	\end{equation*}
	whose solution is clearly $g_n(t) = \exp\bigl\{-\int_0^t \int_I \lambda\bigl(x, h^{-1}(Y_{\tau_n})^c\bigr) \, \varphi_{\pi_{\tau_n}^\mu}(s; \,\dd x) \, \dd s\bigr\}$ for $t \in (0, \tau_{n+1}-\tau_n]$.
	Therefore we get (\ref{eq:Ysojourntimes}).
	
	Finally, (\ref{eq:Ypostjump}) follows from an immediate application of \citep[Ch. VIII, T7 and T16]{bremaud:pp}. In fact, we have that on $\{\tau_{n+1} < +\infty\}$
	\begin{multline*}
		\p_\mu(Y_{\tau_{n+1}} \in B \mid \cY_{\tau_{n+1}^-}^\mu) = \\
		= \frac{1}{r\bigl(\varphi_{\pi_{\tau_n}^\mu}(\tau_{n+1}^- - \tau_n)\bigr)} \int_I \lambda\bigl(x, h^{-1}(B \setminus \{Y_{\tau_n}\})\bigr) \, \varphi_{\pi_{\tau_n}^\mu}(\tau_{n+1}^- - \tau_n; \,\dd x), \quad \p_\mu\text{-a.s.}
	\end{multline*}
	whence the desired equality. Notice that the fraction is well defined since, by \citep[Ch. VIII, Th. T12]{bremaud:pp}, $r\bigl(\varphi_{\pi_{\tau_n}^\mu}(\tau_{n+1}^- - \tau_n)\bigr) > 0$ on $\{\tau_{n+1} < +\infty\}$. 
\end{proof}

We are now ready to prove the main Theorem of this Section, that is the characterization of the filtering process as a PDMP.

\begin{theorem} \label{th:uncontrolledfilteringprocessPDP}
	For every $\nu \in \Delta_e$ the filtering process $\pi^\nu = (\pi^\nu_t)_{t \geq 0}$
	defined on the probability space $(\Omega, \cX^\nu, \p_\nu)$ and taking values in $\Delta_e$
	is a \emph{Piecewise Deterministic Markov Process} with respect to the filtration $(\cY^\nu_t)_{t \geq 0}$ and the triple $(F, r, R)$ defined in (\ref{eq:PDPvectorfield})-(\ref{eq:PDPtransition}) and with starting point $\nu$.
	
	More specifically, we have that $\pi^\nu$ is a $(\cY^\nu_t)_{t \geq 0}$-Markov process and the following equalities hold $\p_\nu\text{-a.s.}$
	\begin{equation}\label{eq:PDPflow}
	\pi_t^\nu = \varphi_{\pi^\nu_{\tau_n}}(t-\tau_n), \quad t \in [\tau_n, \tau_{n+1}), \, n \in \N_0,
	\end{equation}
	\begin{multline}\label{eq:PDPsojourntimes}
	\p_\nu(\tau_{n+1} - \tau_n > t, \, \tau_n < +\infty \mid \cY^\nu_{\tau_n}) = \\
	\ind_{\tau_n < +\infty} \exp\biggl\{-\int_0^t r\bigl(\varphi_{\pi_{\tau_n}^\mu}(s)\bigr) \, \dd s\biggr\}, \quad t \geq 0, \, n \in \N_0,
	\end{multline}
	\begin{multline}\label{eq:PDPpostjumplocations}
	\p_\nu(\pi^\nu_{\tau_{n+1}} \in D, \, \tau_{n+1} < +\infty \mid \cY^\nu_{\tau_{n+1}^-}) = \\
	\ind_{\tau_{n+1} < +\infty} R\bigl(\varphi_{\pi^\nu_{\tau_n}}(\tau_{n+1}^- - \tau_n); D\bigr), \quad D \in \cB(\Delta_e), \, n \in \N_0,
	\end{multline}
	where, for each $n \in \N_0$, $\varphi_{\pi^\nu_{\tau_n}}$ is the flow starting from $\pi^\nu_{\tau_n}$ and determined by the vector field $F$.
\end{theorem}

\begin{proof}
	Fix $\nu \in \Delta_e$, hence $\nu \in \Delta_y$ for some $y \in O$. It is clear that $\p_\nu(Y_0 = y) = 1$ and that $H_y[\nu] = \nu$. Hence $\p_\nu(\pi^\nu_0 = \nu) = 1$, \ie the filtering process $\pi^\nu$ starts from $\nu$. The $(\cY^\nu_t)_{t \geq 0}$-Markov property for the process $\pi^\nu$ has already been established in Proposition \ref{prop:filtermarkov}. The deterministic dynamic between consecutive jump times expressed by (\ref{eq:PDPflow}) easily follows from (\ref{eq:filteringproctraj}). Moreover (\ref{eq:PDPsojourntimes}) coincides with (\ref{eq:Ysojourntimes}).
	
	It remains to prove (\ref{eq:PDPpostjumplocations}). From (\ref{eq:filteringproctraj}) we have that on $\{\tau_{n+1} < +\infty\}$ and for all $D \in \cB(\Delta_e)$
	\begin{equation*}
		\p_\nu\Bigl(\pi^\nu_{\tau_{n+1}} \in D \Bigm| \cY^\nu_{\tau_{n+1}^-}\Bigr) = \p_\nu\Bigl(H_{Y_{\tau_{n+1}}}\bigl[\Lambda\bigl(\pi_{\tau_{n+1}^-}^\nu\bigr)\bigr] \in D \Bigm| \cY^\nu_{\tau_{n+1}^-}\Bigr).
	\end{equation*}
	Observing that $\Lambda(\pi_{\tau_{n+1}-}^\nu) = \Lambda\bigl(\varphi_{\pi_{\tau_n}^\nu}(\tau_{n+1}^- - \tau_n)\bigr)$ is a $\cY^\nu_{\tau_{n+1}^-}$-measurable random variable (with values on $\Delta_e$), an easy application of the freezing lemma to the last displayed equation entails that
	\begin{equation*}
		\p_\nu\Bigl(\pi^\nu_{\tau_{n+1}} \in D \Bigm| \cY^\nu_{\tau_{n+1}^-}\Bigr) \!=\! \int_O \!\!\ind_D\bigl(H_\upsilon\bigl[\Lambda\bigl(\varphi_{\pi_{\tau_n}^\nu}(\tau_{n+1}^- - \tau_n)\bigr)\bigr]\bigr) \rho\bigl(\varphi_{\pi_{\tau_n}^\nu}(\tau_{n+1}^- - \tau_n), \dd \upsilon \bigr),
	\end{equation*}
	hence the desired result.
\end{proof}

We conclude this Section stating the following \emph{weak Feller} property of the transition kernel $R$. This is a corollary of Proposition~\ref{prop:weakFeller} that the reader can find in the following Section. However, we state it also in the uncontrolled case because it is useful on its own, for instance as a fundamental tool in optimal stopping or optimal switching problems (see, e.g., its counterpart in the Markov chain case in \citep[Prop. 5.3]{confortola:filt}).

\begin{proposition}\label{prop:weakfeller}
	Let Assumption \ref{hp:lambda} hold. Then for every bounded and continuous function $w \colon \Delta_e \to \R$ the function $\rho \mapsto r(\rho) \int_{\Delta_e} w(p) R(\rho; \dd p)$ is bounded and continuous on $\Delta_e$.
\end{proposition}

\section{Optimal control problem with noise-free partial observation}\label{sec:jmpoptcontrol}

In the present Section we address an optimal control problem for a pure jump Markov process with noise-free partial observation. We are given a pair $(X,Y) = (X_t,Y_t)_{t \geq 0}$ of continuous-time processes, that will be precisely defined below. As in the previous Sections, we call $X$ the \emph{unobserved} process and $Y$ the \emph{observed} process. Their respective state spaces, denoted by $I$ and $O$, are supposed to be complete and separable metric spaces, equipped with their respective Borel $\sigma$-algebras $\cI$ and $\cO$.
The aim is to control the dynamics of a pure jump Markov process by deciding the control actions on the basis of the observation provided by $Y$. The control is described by third stochastic process $\bfu = (u_t)_{t \geq 0}$, the \emph{control process}, with values in the set of Borel probability measures $\cP(U)$ on a measurable space $(U, \cU)$, the \emph{space of control actions}.

Optimal control problems with partial observation require a two step procedure to be studied. First, we need to express the available informations about the unobserved process contained in the observed one. This is done through the \emph{filtering process}, that we already studied in Sections \ref{sec:filtering} and \ref{sec:filterchar}, although in the "uncontrolled" setting. The second step is to introduce a \emph{separated control problem}, in which the state variable becomes the filtering process itself, so that we can study a control problem with complete observation. As one expects, the original control problem with partial observation and the separated one are connected through their respective value functions. Therefore, it is possible to analyze the latter to have a complete study of the first one.

\subsection{Formulation of the control problem}\label{sec:optcontrolform}
In this Subsection we introduce our optimal control problem with noise-free partial observation. The results in this Subsection are stated without proof, since they are slight modifications or generalizations of those proved in Sections \ref{sec:filtering} and \ref{sec:filterchar}.

The unobserved process $X$ is a \emph{pure jump Markov process}. In our setting, as we will shortly see, to control its dynamics means to control its law. This is described by the initial law of $X$ and a \emph{controlled rate transition measure} $\lambda$ from $(I \times U, \cI \otimes \cU)$ into $(I, \cI)$ such that
\begin{equation*}
	\lambda(x, u, \{x\}) = 0, \quad x \in I, \, u \in U.
\end{equation*}
To simplify the notation it is convenient to define the \emph{controlled jump rate function} $\lambda \colon I \times U \to [0, +\infty)$ as $\lambda(x, u) \coloneqq \lambda(x, u, I), \quad x \in I, \, u \in U$.
It will always be clear from the context if $\lambda$ refers to the rate transition measure or the jump rate function.

We introduce the following Assumption.
\begin{assumption} \label{assumption:lambdainf}
	\mbox{}
	\begin{enumerate}[1.]
	\item The space of control actions $U$ is a compact metric space. \label{assumption:Ucompact}
	\item For each $x \in I$ the map $u \mapsto \lambda(x, u, \cdot)$ is continuous from $U$ to $\cM(I)$, hence bounded and uniformly continuous on $U$. \label{assumption:lambdacont}
	\item $\displaystyle \sup_{(x,u) \in I \times U} \lambda(x, u) < +\infty.$ \label{assumption:lambdabdd}
	\end{enumerate}
\end{assumption}
\begin{rem}
	We recall that at the beginning of Section \ref{sec:filterchar} we endowed the space $\cM(I)$ with the total variation norm. Notice, also, that compactness of $U$ entails that $\cP(U)$ is a compact metrizable space, too.
\end{rem}

To define the triple of processes $(X,Y,\bfu)$ presented above, we need to put ourselves in a canonical framework for a continuous time pure jump Markov process. as follows
\begin{itemize}
	\item Let $\bar{\Omega} = \{\bar{\omega} \colon [0, +\infty) \to I, \text{ c\'adl\'ag}\}$ be the canonical space for \mbox{$I$~-~valued} pure jump Markov processes. We define $X_t(\omega) = \omega(t)$, for $\omega \in \Omega, \, t \geq 0$.
	\item The family of $\sigma$-algebras $(\cX_t^\circ)_{t \geq 0}$ given by
	\begin{equation*}
	\cX_t^\circ = \sigma(X_s, 0 \leq s \leq t), \quad
	\cX^\circ = \sigma(X_s, s \geq  0),
	\end{equation*}
	is the natural filtration of the process $X = (X_t)_{t \geq 0}$.
	\item The observed process $Y$ and its natural filtration $(\cY_t^\circ)_{t \geq 0}$ are defined as
	\begin{align*}
	Y_t(\omega) &= h(X_t(\omega)), \, t \geq 0, \, \omega \in \Omega; & \cY_t^\circ &= \sigma\bigl(Y_s,\, 0 \leq s \leq t\bigr), \, t \geq 0.
	\end{align*} 
	We can equivalently describe this process via a marked point process $(\eta_n, \tau_n)_{n \in \N}$ together with the initial condition $\eta_0 = h(\xi_0) = Y_0$.
	\item The class of \emph{admissible controls} is given by
	\begin{equation}\label{eq:admissiblecontrolsinf}
		\cU_{ad} \coloneqq \Bigl\{\bfu \colon \Omega \times [0,+\infty) \to \cP(U), \, (\cY_t^\circ)_{t \geq 0}\text{~-~predictable}\Bigr\}.
	\end{equation}
	Regarding the choice of $\cP(U)$-valued control processes, see Remark \ref{rem:relaxedcontrols} below. 
	\item For every $\mu \in \cP(I)$ and all $\bfu \in \cU_{ad}$ we denote by $\p_\mu^\bfu$ the unique probability measure on $(\Omega, \cX^\circ)$ such that the process $X$ is a pure jump Markov process, with initial law $\mu$ and $(\cX_t^\circ)_{t \geq 0}$-compensator
	given by
	\begin{equation}\label{eq:predprojinf}
		\nu^\bfu(\omega;\, \dd t \, \dd z) \coloneqq 
		\ind_{t < T_{\infty}(\omega)} \int_U \lambda(X_{t-}(\omega), \fru, \dd z) \, u_t(\omega \, ; \dd \fru) \, \dd t, \quad \omega \in \Omega
	\end{equation}
	where $T_\infty(\omega), \, \omega \in \Omega$ is the accumulation point of the sequence of jump times of $X$, sometimes called the \emph{explosion point}.
	Uniqueness of the probability measure $\p_\mu^\bfu$ is granted by \citep[Th. 3.6]{jacod:mpp}.
	\item Let $\cX^{\mu, \bfu}$ be the $\p_\mu^\bfu$-completion of $\cX^\circ$. We still denote by $\p_\mu^\bfu$ the measure naturally extended to this new $\sigma$-algebra.
	Let $\cZ^{\mu, \bfu}$ be the family of sets in $\cX^{\mu, \bfu}$ with zero $\p_\mu^\bfu$-probability and define $\cY_t^{\mu,\bfu} \coloneqq \sigma(\cY_t^\circ \cup \cZ^{\mu,\bfu}), \quad t \geq 0$.
\end{itemize}

\begin{rem}\label{rem:relaxedcontrols}
	Considering $\cP(U)$-valued control processes (the so called \emph{relaxed controls}), instead of ordinary $U$-valued processes, has considerable technical benefits that will allow us to prove the existence of an optimal control. At this stage such a choice has almost no impact on the problem itself (except for a slightly more complicated notation), being both $U$ and $\cP(U)$ compact. It is important to notice that relaxed controls can be approximated by ordinary ones, and the latter are included in this formulation by considering the isomorphism between $U$ and the subset of $\cP(U)$ given by Dirac probability measures. 
\end{rem}

Two consequences can be deduced from point process theory, summarized in the following two Remarks.

\begin{rem}\label{rem:Xnonexplosive}
	From point \ref{assumption:lambdabdd} of Assumption \ref{assumption:lambdainf} it follows that the processes $X$ and $Y$ are both $\p_\mu^\bfu$-a.s. \mbox{non-explosive}, \ie that the accumulation points of their respective sequences of jump times are equal to $+\infty, \, \p_\mu^\bfu$-a.s.
	For this reason we will drop the term $\ind_{t < T_\infty}$ appearing in (\ref{eq:predprojinf}).
\end{rem}

\begin{rem}\label{rem:controlrepr}	
	Predictable processes with respect to the natural filtration of a point process admit a precise description (see \eg \citep[Lemma 3.3]{jacod:mpp} or \citep[Appendix A2, Theorem T34]{bremaud:pp}).
	A control process $\bfu \in \cU_{ad}$ is completely determined by a sequence of \mbox{Borel-measurable} functions $(u_n)_{n \in \N_0}$, with $u_n \colon [0, +\infty) \times O \times \bigl((0, +\infty) \times O \bigr)^n \to \cP(U)$ for each $n \in \N_0$, so that we can write
	\begin{multline}\label{eq:controlsrepresentationinf}
	u_t(\omega) = u_0(t, Y_0(\omega)) \ind(0 \leq t \leq \tau_1(\omega)) +\\
	\sum_{n = 1}^\infty u_n(t, Y_0(\omega), \tau_1(\omega), Y_{\tau_1}(\omega), \dots, \tau_n(\omega), Y_{\tau_n}(\omega)) \ind(\tau_n(\omega) < t \leq \tau_{n+1}(\omega)).
	\end{multline}
	This kind of decomposition will be of fundamental importance in the analysis of the optimal control problem. Moreover, we will frequently switch between the notation $(u_t)_{t \geq 0}$ and $(u_n)_{n \in \bar \N_0}$ and, to simplify matters, we will often use the more compact writing $u_n(\cdot)$ instead of $u_n(\cdot, Y_0(\omega), \dots, \tau_n(\omega), Y_{\tau_n}(\omega)), \, n \in \N_0$.
\end{rem}

Our optimal control problem aims at minimizing, for all possible initial laws $\mu \in \cP(I)$ of the process $X$, the following \emph{cost functional}
\begin{equation} \label{eq:costfunctionalinf}
J(\mu, \bfu) = \e_\mu^{\bfu} \biggl[ \int_0^\infty e^{-\beta t} \int_U f(X_t, \fru) \, u_t(\dd \fru) \, \dd t \biggr], \quad \bfu \in \cU_{ad}
\end{equation}
where $f$ is called \emph{cost function} and $\beta > 0$ is a fixed constant called \emph{discount factor}. The optimal cost is represented by the \emph{value function}
\begin{equation} \label{eq:valuefunctioninf}
V(\mu) = \inf_{\bfu \in \cU_{ad}} J(\mu, \bfu).
\end{equation}

We make the following assumption on the cost function $f$, ensuring that the functional $J$ is well defined (and also bounded).
\begin{assumption} \label{assumption:costfunctioninf}
	The function $f \colon I \times U \to \R$ is bounded and uniformly continuous. In particular, for some constant $C_f > 0$ it holds that
	\begin{equation}\label{eq:fboundedinf}
	\sup_{(x,u) \in I \times U} \abs{f(x,u)} \leq C_f,
	\end{equation}
\end{assumption}

We can transform the problem formulated above into a complete observation problem by means of the \emph{filtering process}. This is defined as the $\cP(I)$-valued process given by
\begin{equation*}
	\p_\mu^\bfu(X_t \in A \mid \cY_t^{\mu, \bfu}), \quad t \geq 0, \, A \in \cI.
\end{equation*}
In our case, its state space is the subset of $\cP(I)$ called \emph{effective simplex} $\Delta_e$ and defined as
\begin{equation}\label{eq:effectivesimplexinf}
	\Delta_e \coloneqq \bigcup_{y \in O} \Delta_y, \quad \Delta_y \coloneqq \{\nu \in \cP(I) \colon \nu\bigl(h^{-1}(y)^c\bigr) = 0\}, \, y \in O
\end{equation}
It is worth noticing that the filtering process is a $(\cY_t^{\mu, \bfu})_{t \geq 0}$~-~adapted process and since $(\cY_t^{\mu, \bfu})_{t \geq 0}$ is right continuous we can choose a $(\cY_t^{\mu, \bfu})_{t \geq 0}$~-~progressive version. We will assume this whenever needed.

The filtering equation and the characterization of the filtering process as a PDP are generalizations of the corresponding results in Sections \ref{sec:filtering} and \ref{sec:filterchar}. Since their proofs are mere adaptations of those already provided, we state these results without proofs.

\begin{theorem}\label{th:PDPflowinf}
	For each $y \in O$, let $F_y \colon \cM(I) \times U \to \cM(I)$ be the map defined as
	\begin{equation} \label{eq:controlledvectorfieldinf}
	F_y(\nu, u) \coloneqq \cB_y^u \nu - \nu \, \cB_y^u \nu(I), \quad \nu \in \cM(I), \, u \in U
	\end{equation}
	where the controlled operator $\cB_y^u$ is defined for all $y \in O$ and $u \in U$ by
	\begin{equation}\label{eq:operatorByinf}
	\cB_y^u \nu(\dd z) \coloneqq \ind_{\pre(y)}(z) \int_I \lambda(x, u, \dd z) \, \nu(\dd x) - \lambda(z, u) \nu(\dd z), \quad \nu \in \cM(I).
	\end{equation}
	
	For all $y \in O, \, \rho \in \Delta_y$ and all measurable $m \colon [0, +\infty) \to \cP(U)$, the ODE
	\begin{equation} \label{eq:relaxedODEinf}
	\begin{dcases}
	\frac{\dd}{\dd t} z(t) = \int_U F_y(z(t), u) \, m(t \, ;\dd u), \quad t \geq 0, \\
	z(0) = \rho,
	\end{dcases}
	\end{equation}
	admits a unique strong solution $z \colon [0, +\infty) \to \cM(I)$.
	Moreover $z(t) \in \Delta_y$ for all $t \geq 0$.
\end{theorem}

\phivarphi

\begin{rem}\label{rem:flowinf}
	Similarly to what we did following Remark \ref{rem:flow}, for all measurable control functions $m$ as above we can define a \emph{global controlled flow} $\phi^m$ on $\Delta_e$, so that the solution $z$ to \eqref{eq:relaxedODEinf} can be denoted by $\phi_\rho^m$.
	In this way, for all fixed control functions $m$ and $t \geq 0, \, \rho \mapsto \phi_\rho^m(t)$ is a function $\Delta_e \to \Delta_e$ leaving each set $\Delta_y$ invariant.
	We also associate to the global flow a \emph{global controlled vector field} $F \colon \Delta_e \times U \to \Delta_e$ defined as
	\begin{equation}\label{eq:globalvectorfieldinf}
		F(\nu, u) \coloneqq F_y(\nu, u) = \cB_y^u \nu - \nu \, \cB_y^u \nu(I), \quad \nu \in \Delta_y, \, u \in U.
	\end{equation}
\end{rem}

\begin{theorem}[Filtering equation]
	For all $\omega \in \Omega$ define $\tau_0(\omega) \equiv 0$ and for fixed $\bfu \in \cU_{ad}$ the stochastic process
	$\pi^{\mu,\bfu} = (\pi^{\mu,\bfu}_t)_{t \geq 0}$ as the unique solution 
	of the following system of ODEs
	\begin{equation}\label{eq:controlledfilteringprocessinf}
		\begin{dcases}
			\frac{\dd}{\dd t} \pi^{\mu,\bfu}_t(\omega) = \int_U F(\pi^{\mu,\bfu}_t(\omega), \fru) \, u_t(\omega \, ; \dd \fru), \quad t \in [\tau_n(\omega), \tau_{n+1}(\omega)), \, n \in \N_0 \\
			\pi^{\mu,\bfu}_0(\omega) = H_{Y_0(\omega)}[\mu] \\
			\pi^{\mu,\bfu}_{\tau_n}(\omega) = H_{Y_{\tau_n(\omega)}(\omega)}\bigl[\Lambda\bigl(\pi_{\tau_n^-(\omega)}^{\mu, \bfu}(\omega), u_{\tau_n^-(\omega)}(\omega)\bigr)\bigr], \, n \in \N.
		\end{dcases}
	\end{equation}
	where $F$ is the vector field defined in (\ref{eq:globalvectorfieldinf}), $H$ is the operator given in Definition \ref{def:operatorH}, $\Lambda \colon \Delta_e \times \cP(U) \to \cM_+(I)$ is defined as
	\begin{equation}\label{eq:operatorLambdainf}
		\Lambda(\nu, u) \coloneqq \ind_{h^{-1}(y)^c}(z) \int_I \int_U \lambda(x, \fru, \dd z) \, u(\dd \fru) \, \nu(\dd x), \quad \nu \in \Delta_y, \, u \in \cP(U)
	\end{equation} 
	and the quantity $\pi_{\tau_n^-(\omega)}^{\mu,\bfu}(\omega)$ is defined as
	$$\pi_{\tau_n^-(\omega)}^{\mu,\bfu}(\omega) \coloneqq \lim_{t \to \tau_n(\omega)^-} \pi_t^{\mu,\bfu}(\omega), \quad \text{on } \{\omega \in \Omega \colon \tau_n(\omega) < +\infty\}.$$
	
	Then, $\pi^{\mu,\bfu}$ is $(\cY_t^\circ)_{t \geq 0}$-adapted and is a modification of the filtering process,
	\ie 
	\begin{equation*}
		\pi_t^{\mu,\bfu}(A) = \p_\mu^\bfu(X_t \in A \mid \cY_t^{\mu, \bfu}), \quad \p_\mu^\bfu\text{-a.s.}, \, t \geq 0, \, A \in \cI.
	\end{equation*}
\end{theorem}

\begin{rem}
	Thanks to the structure of admissible controls shown in (\ref{eq:controlsrepresentationinf}) we can write (\ref{eq:controlledfilteringprocessinf}) as
	\begin{equation*} \label{eq:controlledfilteringprocess2inf}
		\begin{dcases}
		\frac{\dd}{\dd t} \pi^{\mu,\bfu}_t = \int_U F(\pi^{\mu,\bfu}_t, \fru) \, u_n(t,Y_0, \dots, \tau_n, Y_{\tau_n} \, ; \dd \fru), \quad t \in [\tau_n, \tau_{n+1}), \, n \in \N_0, \\
		\pi^{\mu,\bfu}_0 = H_{Y_0}[\mu], \\
		\pi^{\mu,\bfu}_{\tau_n} = H_{Y_{\tau_n}}\bigl[\Lambda\bigl(\pi_{\tau_n^-}^{\mu, \bfu}, u_{n-1}(\tau_n^-, Y_0, \dots, \tau_{n-1}, Y_{\tau_{n-1}})\bigr)\bigr], \, n \in \N.
		\end{dcases}
	\end{equation*}
\end{rem}

The filtering process is again a \emph{Piecewise Deterministic Process} (PDP). Notice that,in general, it is no more a Markov process, due to the past-dependent structure of admissible controls. Its characteristic triple $(F, r, R)$ is given by the controlled vector field $F$ defined in (\ref{eq:globalvectorfieldinf}), a controlled jump rate function $r \colon \Delta_e \times U \to [0, +\infty)$ and a controlled stochastic kernel $R$, \ie a probability transition kernel from $(\Delta_e \times U, \cB(\Delta_e) \otimes \cU)$ to $(\Delta_e, \cB(\Delta_e))$.
More precisely
\begin{equation} \label{eq:PDPcharacteristictripleinf}
\begin{split}
	F(\nu, u) &\coloneqq F_y(\nu, u) = \cB_y^u \nu - \nu \, \cB_y^u \nu(I), \quad \nu \in \Delta_y, \, u \in U,\\
	r(\nu, u) &\coloneqq - \cB_y^u \nu(I) = \int_I \lambda\bigl(x, u,  h^{-1}(y)^c\bigr) \, \nu(\dd x), \quad \nu \in \Delta_y, \, u \in U, \\
	R(\nu, u, D) &\coloneqq \int_O \ind_D\bigl(H_\upsilon[\Lambda(\nu, u)]\bigr) \, \rho(\nu, u, \dd \upsilon), \quad \nu \in \Delta_y, \, D \in \cB(\Delta_e), \, u \in U,
\end{split}
\end{equation}
where $\rho$ is a transition probability from $\bigl(\Delta_e \times U, \cB(\Delta_e) \otimes \cU\bigr)$ into $(O, \cO)$ defined for all $\nu \in \Delta_y, \, u \in U$ and all $B \in \cO$ as
\begin{equation}\label{eq:rhotransitioninf}
	\rho(\nu, u, B) \coloneqq 
	\begin{dcases}
		\frac{1}{r(\nu, u)} \int_I \lambda\bigl(x, u, h^{-1}(B \setminus \{y\})\bigr) \, \nu(\dd x), &\text{if } r(\nu, u) > 0, \\
		q_y(B), &\text{if } r(\nu, u) = 0,
	\end{dcases}
\end{equation}
where $(q_y)_{y \in O}$ is a family of probability measures, each concentrated on the level set $h^{-1}(y), \, y \in O$, whose exact values are irrelevant.
Since for any given $\nu \in \Delta_y$ and $u \in U$ the probability $\rho(\nu, u, \cdot)$ is concentrated on the set $O \setminus \{y\}$, the probability $R(\nu, u, \cdot)$ is concentrated on $\Delta_e \setminus \Delta_y$.

\begin{rem}
Under Assumption \ref{assumption:lambdainf} we have that $\sup_{(\rho, u) \in \Delta_e \times U} |r(\rho, u)| \leq C_r$ for some $C_r > 0$.
\end{rem}

\begin{theorem} \label{th:filteringprocessPDPinf}
	For every $\nu \in \Delta_e$ and all $\bfu \in \cU_{ad}$ 
	the filtering process $\pi^{\nu,\bfu} = (\pi^{\nu,\bfu}_t)_{t \geq 0}$
	defined on the probability space $(\Omega, \cX^\circ, \p_\nu^\bfu)$ and taking values in $\Delta_e$
	is a controlled \emph{Piecewise Deterministic Process} with respect to the triple $(F, r, R)$ defined in (\ref{eq:PDPcharacteristictripleinf}) and with starting point $\nu$.
	
	More specifically, we have that for all $n \in \N_0$ and $\p_\nu^\bfu\text{-a.s.}$
	\begin{equation}\label{eq:controlledflowinf}
		\pi_t^{\nu, \bfu} = \phi_{\pi^{\nu, \bfu}_{\tau_n}}^{u_n}(t-\tau_n), \quad \text{on } \{\tau_n < +\infty\}, \, t \in [\tau_n, \tau_{n+1})
	\end{equation}
	\begin{multline}\label{eq:controlledsojourntimesinf}
		\p_\nu^\bfu(\tau_{n+1} - \tau_n > t, \, \tau_n < +\infty \mid \cY^{\nu,\bfu}_{\tau_n}) = \\
		\ind_{\tau_n < +\infty} \exp\biggl\{-\int_0^t \int_U r\bigl(\phi_{\pi^{\nu,\bfu}_{\tau_n}}^{u_n(\cdot \, + \, \tau_n)}(s), \fru\bigr) \, u_n(s + \tau_n \, ; \dd \fru) \, \dd s\biggr\}, \quad t \geq 0
	\end{multline}
	\begin{multline}\label{eq:controlledpostjumplocationsinf}
		\p_\nu^\bfu(\pi^{\nu, \bfu}_{\tau_{n+1}} \in D, \, \tau_{n+1} < +\infty \mid \cY^{\nu,\bfu}_{\tau_{n+1}^-}) = \\
		\ind_{\tau_{n+1} < +\infty} \int_U R\bigl(\phi_{\pi^{\nu, \bfu}_{\tau_n}}^{u_n(\cdot \, + \, \tau_n)}(\tau_{n+1}^- - \tau_n), \fru; D\bigr) \, u_n(\tau_{n+1}^- \, ; \dd \fru), \quad D \in \cB(\Delta_e)
	\end{multline}
	where, for each $n \in \N_0$, $\phi_{\pi^{\nu, \bfu}_{\tau_n}}^{u_n}$ is the flow starting from $\pi^{\nu, \bfu}_{\tau_n}$ and determined by the controlled vector field $F$ under the action of the control function $u_n(\cdot, Y_0, \dots, \tau_n, Y_{\tau_n})$.
\end{theorem}

\begin{rem}
	In optimal control problems for PDPs (see \eg \citep{davis:markovmodels}) it is customary to define the class of admissible controls as \emph{piecewise open-loop controls}. These control functions, first studied by Vermes in \citep{vermes:optcontrol}, depend at any time $t \geq 0$ on the position of the PDP at the last jump-time prior to $t$ and on the time elapsed since the last jump. 
	
	In (\ref{eq:admissiblecontrolsinf}) we specified a different class of admissible controls, more suited to our problem and imposed by the fact that we are dealing with partial observation, hence equations (\ref{eq:controlledflowinf})-(\ref{eq:controlledpostjumplocationsinf}) are changed with respect to the standard formulation with piecewise open-loop controls.

	Another difference with respect to the usual definition of a PDP is the absence of a boundary behavior, \ie the specification of a transition kernel giving the post-jump position of the process in case it touches the boundary.
\end{rem}

Now that we have a precise description of the filtering process, we can use it to rewrite the cost functional $J$ defined in (\ref{eq:costfunctionalinf}). In fact, since control processes $\bfu$ are 
$(\cY_t^\circ)_{t \geq 0}$~-~predictable and we know that the filtering process $\pi^{\mu, \bfu}$
provides us with the conditional law of $X_t$ given $\cY_t^{\mu,\bfu}$, for all $t \geq 0$, an easy application of the Fubini-Tonelli Theorem, of the freezing lemma and Remark \ref{rem:controlrepr} shows that
\begin{equation} \label{eq:functionalJdiscreteinf}
\begin{split}
&\quad \, \, J(\mu, \bfu) = \e_\mu^{\bfu} \biggl[ \int_0^\infty e^{-\beta t} \int_U \pi_t^{\mu, \bfu}(f; \, \fru) \, u_t(\dd \fru) \, \dd t \biggr] \\
&=\e_\mu^\bfu
\biggl[ \sum_{n = 0}^{+\infty} \int_{\tau_n}^{\tau_{n+1}} e^{-\beta t} \int_U \pi^{\mu, \bfu}_t(f; \, \fru) \, u_n(t \, ; \dd \fru) \, \dd t \biggr] \\
&= \e_\mu^\bfu
\biggl[ \sum_{n = 0}^{+\infty} e^{-\beta \tau_n} \int_0^{+\infty} e^{-\beta t} \chi_{\pi^{\mu, \bfu}_{\tau_n}}^{u_n(\cdot \, + \, \tau_n)}(t) \int_U \phi_{\pi^{\mu, \bfu}_{\tau_n}}^{u_n(\cdot \, + \, \tau_n)}(f; \, \fru, t) \, u_n(t+\tau_n \, ; \dd \fru) \, \dd t \biggr]\\
&= \e_\mu^\bfu
\biggl[\sum_{n = 0}^{+\infty} e^{-\beta \tau_n} g\bigl(\pi^{\mu,\bfu}_{\tau_n}, u_n(\cdot+\tau_n, Y_0, \tau_1, Y_{\tau_1}, \dots, \tau_n, Y_{\tau_n})\bigr) \biggr]
\end{split}
\end{equation}
where the function $g$, that will be defined precisely in \eqref{eq:costfunctiondiscreteinf}, represents the double integral appearing in the third line and $\chi_{\pi^{\mu, \bfu}_{\tau_n}}^{u_n(\cdot \, + \, \tau_n)}$ is the survival distribution appearing in (\ref{eq:controlledsojourntimesinf}). The notations $\phi_\nu^m(f; \, u, t) \coloneqq \int_I f(x,u) \, \phi_\nu^m(t; \, \dd x)$ and $\pi_t^{\mu, \bfu}(f; \, u) \coloneqq \int_I f(x,u) \, \pi_t^{\mu,\bfu}(\dd x)$ are adopted.

At this point, as is customary in optimal control problem with partial observation, a reformulation of the original optimal control problem is necessary. We have to introduce a \emph{separated} control problem with full observation, having as new state variable the filtering process. More specifically, as \eqref{eq:functionalJdiscreteinf} suggests, we can address a discrete-time problem for the process given by the pair of jump times and post-jump locations of the filtering process. We will show that the original control problem and the separated one are equivalent, providing also a precise description of the relationship between their respective control strategies.

\subsection{The separated optimal control problem}\label{sec:pdpoptcontrol}
In this Subsection we will reformulate the original optimal control problem into a discrete-time one based on the filtering process. This reformulation will fall in the framework of \citep{bertsekas:stochoptcontrol} (from which we will borrow some terminology), a fact that enables us to use known results to study the value function $V$ defined in (\ref{eq:valuefunctioninf}). We will prove the equivalence between the original control problem and the separated one. In particular, we will show that the value function $V$ can be indirectly studied by its counterpart in the separated problem, thanks to the equality proved in Theorem \ref{th:valuefunctionsidentifinf}. We will prove that the value function of the separated problem is lower semicontinuous and we will characterize it as the unique fixed point of a suitable contraction mapping in Theorem \ref{th:vfixedpoint}.

Let us now introduce the separated problem, starting to choose the \emph{action space} as the space of \emph{relaxed controls}
\begin{equation} \label{eq:actionspaceinf}
\cM = \{m \colon [0, +\infty) \to \cP(U), \text{ measurable}\}.
\end{equation}
We recall that $\cM$ is compact under the \emph{Young topology} (see \eg \citep{davis:markovmodels}).
The set of \emph{ordinary controls}
\begin{equation}\label{eq:ordinarycontrolsinf}
A = \{\alpha \colon [0, +\infty) \to U, \text{ measurable}\}
\end{equation}
can be identified as a subset of $\cM$ via the function $t \mapsto \delta_{\alpha(t)}$, $\alpha \in A$, where $\delta_u$ denotes the Dirac probability measure concentrated at the point $u \in U$. Thanks to \citep[Lemma 1]{yushkevich:jumpmodel}, $A$ is a \emph{Borel space} when endowed with the coarsest $\sigma$-algebra such that the maps
\begin{equation*}
\alpha \mapsto \int_0^{+\infty} e^{-t} \psi(t, \alpha(t)) \, \dd t
\end{equation*}
are measurable for all $\psi \colon [0, +\infty) \times U \to \R$, bounded and measurable.
The class of \emph{admissible policies} $\cA_{ad}$ for the separated problem is given by
\begin{equation}
\cA_{ad} = \{\bfa = (a_n)_{n \in \bar \N_0}, a_n \colon \Delta_e \times \bigl((0, +\infty) \times \Delta_e\bigr)^n \to \cM \text{ measurable } \forall n \in \bar \N_0\}.
\end{equation}

In the separated problem the state to be controlled is represented by the filtering process, therefore we put ourselves in a canonical framework for this process.
\begin{itemize}
	\item $\bar{\Omega} = \{\bar{\omega} \colon [0, +\infty) \to \Delta_e, \text{ c\'adl\'ag}\}$ denotes the canonical space for \mbox{$\Delta_e$~-~valued} PDPs. We define $\bar{\pi}_t(\bar{\omega}) = \bar{\omega}(t)$, for $\bar{\omega} \in \bar{\Omega}$, $t \geq 0$, and
	\begin{align*}
	\bar \tau_0(\bar \omega) &= 0, \\
	\bar \tau_n(\bar \omega) &= \inf\{t > \bar \tau_{n-1}(\bar \omega) \text{ s.t. } \bar \pi_t(\bar \omega) \ne \bar \pi_{t^-}(\bar \omega)\}, \quad n \in \N, \\
	\bar \tau_\infty(\bar \omega) &= \lim_{n \to \infty} \bar \tau_n(\bar \omega).
	\end{align*}
	\item The family of $\sigma$-algebras $(\bar{\cF}_t^\circ)_{t \geq 0}$ given by
	\begin{equation*}
	\bar{\cF}_t^\circ = \sigma(\bar{\pi}_s, 0 \leq s \leq t), \quad
	\bar{\cF}^\circ = \sigma(\bar{\pi}_s, s \geq  0),
	\end{equation*}
	is the natural filtration of the process $\bar{\pi} = (\bar{\pi}_t)_{t \geq 0}$.
	\item For every $\nu \in \Delta_e$ and all $\bfa \in \cA_{ad}$ we denote by $\bar \p_\nu^\bfa$ the probability measure on $(\bar \Omega, \bar{\cF}^\circ)$ such that the process $\bar \pi$ is a PDP, starting from the point $\nu$ and with characteristic triple $(F, r, R)$. With this, we mean that for all $n \in \N_0$ and $\bar \p_\nu^\bfa$-a.s.
	\begin{equation}\label{eq:controlledflowPDPinf}
	\bar \pi_t = \phi_{\bar \pi_{\bar \tau_n}}^{a_n}(t-\bar \tau_n), \quad \text{on } \{\bar \tau_n < +\infty\}, \, t \in [\bar \tau_n, \bar \tau_{n+1}).
	\end{equation}
	\begin{multline}\label{eq:controlledsojourntimesPDPinf}
	\bar \p_\nu^\bfa(\bar \tau_{n+1} - \bar \tau_n > t, \, \bar \tau_n < +\infty \mid \bar \cF^\circ_{\bar \tau_n}) = \\
	\ind_{\bar \tau_n < +\infty} \exp\biggl\{-\int_0^t \int_U r(\phi_{\bar \pi_{\bar \tau_n}}^{a_n}(t),\fru) \, a_n(s \, ; \dd \fru) \, \dd s\biggr\}, \quad t \geq 0.
	\end{multline}
	\begin{multline}\label{eq:controlledpostjumplocationsPDPinf}
	\bar \p_\nu^\bfa(\bar \pi_{\bar \tau_{n+1}} \in D, \, \bar \tau_{n+1} < +\infty \mid \bar \cF^\circ_{\bar \tau_{n+1}^-}) = \\
	\ind_{\bar \tau_{n+1} < +\infty} \int_U R(\phi_{\bar \pi_{\bar \tau_n}}^{a_n}(\bar \tau_{n+1}^- - \bar \tau_n), \fru; D) \, a_n(\bar\tau_{n+1}^- - \bar\tau_n \, ; \dd \fru),  \quad D \in \cB(\Delta_e).
	\end{multline}
	where we simplified the notation by indicating $a_n = a_n(\bar \pi_0, \dots, \bar \tau_n, \bar \pi_{\bar \tau_n})$ and, for each $n \in \N_0$, we denoted by $\phi_{\bar \pi_{\bar \tau_n}}^{a_n}$ the flow starting from $\bar \pi_{\bar \tau_n}$ and determined by the controlled vector field $F$ under the action of the relaxed control $a_n(\bar \pi_0, \dots, \bar \tau_n, \bar \pi_{\bar \tau_n})$.
	We recall that the probability measure $\bar \p_\nu^\bfa$ always exists by the canonical construction of a PDP.
	\item For every $Q \in \cP(\Delta_e)$ and every $\bfa \in \cA_{ad}$ we define a probability $\bar{\p}_Q^\bfa$ on $(\bar{\Omega}, \bar{\cF^\circ})$ by 
	$\bar{\p}_Q^\bfa(C) = \int_{\Delta_e} \bar{\p}_\nu^\bfa(C) \, Q(\dd\nu)$ for $C \in \bar{\cF^\circ}$. This means that $Q$ is the initial distribution of $\bar{\pi}$ under $\bar{\p}_Q^\bfa$.
	\item Let $\bar\cF^{Q, \bfa}$ be the $\bar{\p}_Q^\bfa$-completion of $\bar\cF^\circ$. We still denote by $\bar{\p}_Q^\bfa$ the measure naturally extended to this new $\sigma$-algebra.
	Let $\bar{\cZ}^{Q, \bfa}$ be the family of sets in $\bar{\cF}^{Q, \bfa}$ with zero $\bar{\p}_Q^\bfa$-probability and define
	\begin{equation*}
	\bar{\cF}_t^{Q, \bfa} = \sigma(\bar{\cF}_t^\circ \cup \bar{\cZ}^{Q,\bfa}), \quad 
	\bar\cF_t = \intertwo{Q \in \cP(\Delta_e)}{\bfa \in \cA_{ad}} \bar{\cF}_t^{Q,\bfa}, \quad
	t \geq 0.
	\end{equation*}
	$(\bar{\cF}_t)_{t\geq 0}$ is called the \emph{natural completed filtration} of $\bar{\pi}$.
	By a slight generalization of Theorem \citep[Th. 25.3]{davis:markovmodels} it is \mbox{right-continuous}.
\end{itemize}
The PDP $(\bar{\Omega}, \bar{\cF}, (\bar{\cF}_t)_{t \geq 0}, (\bar{\pi}_t)_{t \geq 0}, (\bar{\p}_\nu^\bfa)_{\nu \in \Delta_e}^{\bfa \in \cA_{ad}})$ constructed as above admits the controlled characteristic triple $(F, r, R)$ defined in (\ref{eq:PDPcharacteristictripleinf}).
To simplify the notation, let us introduce the function $\chi_\rho^m$, depending on $\rho \in \Delta_e$ and $m \in \cM$, given by
\begin{equation}
\chi_\rho^m(t) = \exp\biggl\{-\int_0^t \int_U r(\phi_\rho^m(s), \fru) \, m(s \, ; \dd \fru) \, \dd s\biggr\}, \quad t \geq 0.
\end{equation}
In this way, with a slight abuse of notation, we can write (\ref{eq:controlledsojourntimesPDPinf}) as
\begin{equation*}
\bar \p_\nu^\bfa(\bar \tau_{n+1} - \bar \tau_n > t \mid \bar \cF^\circ_{\bar \tau_n}) = \chi_\nu^{a_n}(t), \quad t \geq 0, \, \text{on } \{\bar \tau_n < +\infty\}.
\end{equation*}
Notice that $\chi_\rho^m$ solves the ODE
\begin{equation} 
\begin{dcases}
\frac{\dd}{\dd t} z(t) = -z(t) \int_U r(\phi_\rho^m(t), \fru)  \, m(t \, ;\dd \fru), \quad t \geq 0, \\
z(0) = 1.
\end{dcases}
\end{equation} 

The observed process $\bar Y$ can be defined on $\bar \Omega$ as follows. Let us introduce the measurable function $\projY \colon \Delta_e \to O$ given by
\begin{equation*}
\projY(p) = y, \quad \text{if } p \in \Delta_y, \text{ for some } y \in O
\end{equation*}
and set
\begin{equation*}
\bar Y_t(\bar\omega) =
\begin{cases}
\projY(\bar\pi_0(\bar\omega)),	& t \in \bigl[0, \bar\tau_1(\bar\omega)\bigr), \\
\projY(\bar\pi_{\bar\tau_n(\bar{\omega})}(\bar\omega)),	& t \in \bigl[\bar\tau_n(\bar\omega), \bar\tau_{n+1}(\bar\omega)\bigr), \, n \in \N, \, \bar\tau_n(\bar\omega) < +\infty, \\
o_\infty,						& t \in \bigl[\bar\tau_\infty(\bar\omega), +\infty), \, \bar\tau_\infty(\bar\omega) < +\infty,
\end{cases}
\end{equation*}
where $o_\infty \in O$ is an arbitrary state, that is irrelevant to specify, since under Assumption \ref{assumption:lambdainf} for each fixed $\nu \in \Delta_e$ and $\bfa \in \cA_{ad}$ we have that $\bar \tau_\infty = +\infty$, \mbox{$\bar \p_\nu^\bfa$-a.s.} In other words, the observed process is \mbox{$\bar \p_\nu^\bfa$-a.s.} non explosive.

Next, we define the cost functional $\bar J$ associated to the separated optimal control problem. In analogy to the form of the original cost functional $J$ shown in (\ref{eq:functionalJdiscreteinf}), we define
\begin{equation} \label{eq:pdpcostfunctionalinf}
\bar J(\nu, \bfa) \coloneqq \bar{\e}_\nu^\bfa 
\biggl[ \sum_{n = 0}^{+\infty} e^{-\beta \bar \tau_n} g\bigl(\bar \pi_{\bar \tau_n}, a_n(\bar \pi_{\bar \tau_0}, \dots, \bar \tau_n, \bar \pi_{\bar \tau_n})\bigr) \biggr], \quad \nu \in \Delta_e, \, \bfa \in \cA_{ad}
\end{equation}
where $g \colon \Delta_e \times \cM \to \R$ is the \mbox{discrete-time} one-stage cost function
\begin{equation}\label{eq:costfunctiondiscreteinf}
g(\nu, m) \coloneqq \int_0^{+\infty} e^{-\beta t} \chi_\nu^m(t) \int_U \phi_\nu^m(f; \, \fru, t) \, m(t \, ; \dd \fru) \, \dd t.
\end{equation} 
Notice that we again adopted the notation $\phi_\nu^m(f; \, u, t) = \int_I f(x, u) \, \phi_\nu^m(t ; \, \dd x)$.
The optimal cost is given by the value function of the separated problem, defined as
\begin{equation}\label{eq:pdpvaluefunctioninf}
v(\nu) \coloneqq \inf_{\bfa \in \cA_{ad}} \bar J(\nu, \bfa).
\end{equation} 

A first step to prove the equivalence between the original control problem and the separated one is to establish a connection between the cost functionals $J$ and $\bar J$, respectively given in (\ref{eq:costfunctionalinf}) and (\ref{eq:pdpcostfunctionalinf}).

\begin{theorem}\label{th:costfunctionalidentifinf}
	Fix $\mu \in \cP(I)$ and let $Q \in \cP(\Delta_e)$ the Borel probability measure on $\Delta_e$ defined as
	\begin{equation} \label{eq:probabilityQinf}
	Q \coloneqq \mu \circ h^{-1} \circ \cH^{-1}, \quad \cH(y) \coloneqq H_y[\mu].
	\end{equation}
	
	For all $\bfu \in \cU_{ad}$ there exists an admissible policy $\bfa \in \cA_{ad}$ such that the laws of $\pi^{\mu, \bfu}$ under $\p_\mu^\bfu$ and of $\bar \pi$ under $\bar \p_Q^\bfa$ are the same. Moreover, for such an admissible policy
	\begin{equation} \label{eq:functionalJequalityinf}
	J(\mu, \bfu) = \int_{\Delta_e} \bar J(\nu, \bfa) \, Q(\dd \nu) = \int_O \bar J(H_y[\mu], \bfa) \, \mu \circ h^{-1}(\dd y).
	\end{equation}
	Vice versa, for all $\bfa = (a_n)_{n \in \N} \in \cA_{ad}$ there exists an admissible control $\bfu \in \cU_{ad}$ such that the same conclusions hold.
\end{theorem}

\begin{proof}
	Let us start from the first part of the Theorem. Given an admissible control $\bfu \in \cU_{ad}$ we are able to construct a corresponding admissible policy in the following way. Let us define the functions $a_n \colon \Delta_e \times \bigl((0, +\infty) \times \Delta_e\bigr)^n \to \cM$ as
	\begin{equation*}
	a_n(\nu_0, \dots, s_n, \nu_n)(t \, ; \dd \fru) = u_n\bigl(t + \, s_n, \projY(\nu_0), \dots, s_n, \projY(\nu_n) \, ; \dd \fru \bigr)
	\end{equation*}
	for all possible sequences $(\nu_i)_{i = 0}^n \subset \Delta_e$ and $(s_i)_{i = 1}^n \subset (0, +\infty)$.
	
	Thanks to the fact that $\projY$ is Borel-measurable and that $\cM$ is a Borel space, we can apply \citep[Lemma 3(i)]{yushkevich:jumpmodel} to conclude that each function $a_n$ is measurable. Therefore we have that $\bfa = (a_n)_{n \in \N_0} \in \cA_{ad}$.
	
	The laws of $\pi^{\mu, \bfu}$ under $\p_\mu^\bfu$ and $\bar \pi$ under $\bar \p_Q^\bfa$ are determined respectively by the distributions of the stochastic processes $\{\pi_0^{\mu, \bfu}, \tau_1, \pi^{\mu,\bfu}_{\tau_1}, \dots\}$ and $\{\bar \pi_0, \bar \tau_1, \bar \pi_{\tau_1}, \dots\}$ and by the flows associated to the controlled vector fields $F^\bfu$ and $F^\bfa$. These laws, in turn, can be expressed via the initial distributions of $\pi^{\mu, \bfu}_0$ and $\bar \pi_0$ and the conditional distributions of the sojourn times and \mbox{post-jump} locations, \ie for $t \geq 0$, $D \in \cB(\Delta_e)$ and $n \in \N$ the quantities
	\begin{align}
	&\p_\mu^\bfu(\tau_n - \tau_{n-1} > t, \, \tau_{n-1} < +\infty \mid \pi^{\mu,\bfu}_0, \dots, \tau_{n-1}, \pi^{\mu,\bfu}_{\tau_{n-1}}); \label{eq:sojournlawmccontrolinf}\\
	&\bar\p_Q^\bfa(\bar\tau_n - \bar\tau_{n-1} > t, \, \bar\tau_{n-1} < +\infty \mid \bar\pi_0, \dots, \bar \tau_{n-1}, \bar{\pi}_{\bar\tau_{n-1}});  \label{eq:sojournlawpdpcontrolinf} \\
	&\p_\mu^\bfu(\pi^{\mu,\bfu}_{\tau_n} \in D, \, \tau_n < +\infty \mid \pi^{\mu,\bfu}_0, \dots, \pi^{\mu,\bfu}_{\tau_{n-1}}, \tau_n); \label{eq:jumplawmccontrolinf}\\
	&\bar\p_Q^\bfa(\bar\pi_{\bar\tau_n} \in D, \, \bar\tau_n < +\infty \mid \bar\pi_0, \dots, \bar{\pi}_{\bar\tau_{n-1}}, \bar \tau_n). \label{eq:jumplawpdpcontrolinf}
	\end{align}
	We now prove that under the two different probability measures $\p_\mu^\bfu$ and $\bar \p_Q^\bfa$ the initial laws of $\pi_0^{\mu,\bfu}$ and $\bar{\pi}_0$ and the distributions (\ref{eq:sojournlawmccontrolinf})~-~(\ref{eq:jumplawpdpcontrolinf}) are equal.
	
	\noindent \textbf{Initial distribution.} 
	This is immediate, as for fixed $D \in \cB(\Delta_e)$ we have that
	\begin{align*}
	&\p_\mu^\bfu(\pi^{\mu,\bfu}_0 \in D) = 
	\p_\mu^\bfu(H_{Y_0}[\mu] \in D) = 
	\p_\mu^\bfu\bigl(Y_0 \in \cH^{-1}(D)\bigr) =\\
	&\p_\mu^\bfu\bigl(X_0 \in h^{-1}(\cH^{-1}(D))\bigr) =
	\mu\bigl(h^{-1}(\cH^{-1}(D))\bigr) = Q(D)
	\end{align*}
	while $\bar \p_Q^\bfa(\bar \pi_0 \in D) = Q(D)$, by definition of $\bar \p_Q^\bfa$.
	
	\noindent \textbf{Sojourn times.} Let us analyze first the conditional law
	(\ref{eq:sojournlawmccontrolinf}). Notice that since we are considering (\ref{eq:sojournlawmccontrolinf}) on the set ${\tau_{n-1} < +\infty}$, $\pi^{\mu, \bfu}_{\tau_{n-1}}$ is well defined and the law of $\tau_n - \tau_{n-1}$ is not trivial. Fix $p_0, \dots, p_{n-1} \in \Delta_e$, where for each $i = 0, \dots, n-1$, $p_i \in \Delta_{b_i}$ for some $b_0 \ne b_1 \ne \dots \ne b_{n-1} \in O$; fix also $0 < s_1 < \dots < s_{n-1} < +\infty$. 
	Since a trajectory of the observed process $Y$ uniquely determines a trajectory of the filtering process $\pi^{\mu,\bfu}$ and vice versa, we can immediately deduce that,
	up to $\p_\mu^\bfu$-null sets
	\begin{align*}
	\cY^{\mu, \bfu}_{\tau_{n-1}} &= \sigma(\pi^{\mu,\bfu}_0, \dots, \tau_{n-1}, \pi^{\mu,\bfu}_{\tau_{n-1}}) & &\text{and} &
	\cY^{\mu, \bfu}_{\tau_n^-} &= \sigma(\pi^{\mu,\bfu}_0, \dots, \pi^{\mu,\bfu}_{\tau_{n-1}}, \tau_n).
	\end{align*}
	From this fact and (\ref{eq:controlledsojourntimesPDPinf}) we can write for $t \geq 0$
	\begin{align*}
	&\p_\mu^\bfu(\tau_n - \tau_{n-1} > t, \, \tau_{n-1} < +\infty \mid \pi^{\mu,\bfu}_0 = p_0, \dots, \tau_{n-1} = s_{n-1}, \pi^{\mu,\bfu}_{\tau_{n-1}} = p_{n-1})
	\\
	= &\chi_{p_{n-1}}^{u_{n-1}}(t) = \exp\biggl\{ - \int_0^t \int_U r\bigl(\phi_{p_{n-1}}^{u_{n-1}}(s), \fru \bigr) \, u_{n-1}(s + s_{n-1}, \, b_0, \, \dots, \, s_{n-1}, \, b_{n-1} \, ; \dd \fru)\, \dd s \biggr\}.
	\end{align*}
	The function $u_{n-1}=u_{n-1}(\cdot + s_{n-1}, \, b_0, \, \dots, \, s_{n-1}, \, b_{n-1})$ can be clearly expressed as 
	\begin{equation*}
	u_{n-1}(\cdot + s_{n-1}, \, \projY(p_0), \, \dots, \, s_{n-1}, \, \projY(p_{n-1})).
	\end{equation*}
	Therefore, if we compare the previous computation with
	\begin{align*}
	&\bar\p_Q^\bfa(\bar\tau_n - \bar\tau_{n-1} > t, \, \bar\tau_{n-1} < +\infty \mid \bar\pi_0 = p_0, \dots, \bar\tau_{n-1} = s_{n-1}, \bar\pi_{\bar \tau_{n-1}} = p_{n-1}) \\
	= &\chi_{p_{n-1}}^{a_{n-1}}(t)
	= \exp\biggl\{ - \int_0^t \int_U r\bigl(\phi_{p_{n-1}}^{a_{n-1}}(s), \fru \bigr) \, a_{n-1}(p_0, \, \dots, \, s_{n-1}, \, p_{n-1})(s \, ; \dd \fru) \, \dd s \biggr\}
	\end{align*}
	we get the desired result, by definition of $\bfa$.
	
	\noindent \textbf{Post-jump locations.} Continuing with the notation previously introduced (to which we simply add a new value $s_n$ such that $0 < s_1 < \dots < s_n < +\infty$), we can write (\ref{eq:jumplawmccontrolinf}) as
	\begin{align*}
	&\p_\mu^\bfu(\pi^{\mu,\bfu}_{\tau_n} \in D, \, \tau_n < +\infty \mid \pi^{\mu,\bfu}_0 = p_0, \dots, \pi^{\mu,\bfu}_{\tau_{n-1}} = p_{n-1}, \tau_n = s_n) \\
	= &\int_U R\bigl(\phi_{p_{n-1}}^{u_{n-1}}(s_n^- - s_{n-1}), \fru \, ; D\bigr) \, u_{n-1}(s_n^-, \, b_0, \, \dots, \, s_{n-1}, \, b_{n-1} \, ; \dd \fru).
	\end{align*}
	On the other hand, we know from (\ref{eq:controlledpostjumplocationsPDPinf})
	\begin{align*}
	&\bar\p_Q^\bfa(\bar\pi_{\bar \tau_n} \in D, \, \bar\tau_n < +\infty \mid \bar\pi_0 = p_0, \dots, \bar\pi_{\bar \tau_{n-1}} = p_{n-1}, \bar\tau_n = s_n) \\
	= &\int_U R\bigl(\phi_{p_{n-1}}^{a_{n-1}}(s_n^- - s_{n-1}), \fru \, ; D\bigr) \, a_{n-1}(p_0, \, \dots, \, s_{n-1}, \, p_{n-1})(s_n^- - s_{n-1} \, ; \dd \fru).
	\end{align*}
	Hence again by definition of $\bfa$ we get the equality of the conditional laws (\ref{eq:jumplawmccontrolinf}) and \eqref{eq:jumplawpdpcontrolinf}.
	
	We are left to prove (\ref{eq:functionalJequalityinf}). Fix $\mu \in \cP(I)$ and $\bfu \in \cU_{ad}$ with corresponding $\bfa \in \cA_{ad}$ defined as above. Let us define the function $\Phi \colon \bar \Omega \to \R$ as
	\begin{align*}
	\Phi(\bar \omega) &= 
	\sum_{n = 0}^{+\infty} e^{-\beta \bar \tau_n(\bar \omega)} g\bigl(\bar \pi_{\bar \tau_n(\bar \omega)} (\bar \omega), a_n(\bar \pi_{0}(\bar \omega),\dots, \bar \tau_n(\bar \omega), \bar \pi_{\bar \tau_n(\bar \omega)}(\bar \omega)\bigr)\\
	&= \sum_{n = 0}^{+\infty} e^{-\beta \bar \tau_n(\bar \omega)} g\bigl(\bar \pi_{\bar \tau_n(\bar \omega)} (\bar \omega), u_n(\cdot + \bar\tau_n(\bar \omega), \dots, \bar \tau_n(\bar \omega), \projY(\bar \pi_{\bar \tau_n(\bar \omega)}(\bar \omega))\bigr).
	\end{align*}
	Thanks to Assumptions \ref{assumption:lambdainf} and \ref{assumption:costfunctioninf} this function is bounded. Since for each $n \in \N_0$ the functions $a_n$ (equivalently $u_n$) are measurable it is also $\bar \cF$-measurable.
	
	Now, take $\bar \omega = \pi^{\mu,\bfu}(\omega)$, $\omega \in \Omega$. It is clear that for all $t \geq 0$ we have $\bar \pi_t(\bar \omega) = \bar \omega(t) = \pi_t^{\mu, \bfu}(\omega)$ and also, by definition of the jump times $(\bar \tau_n)_{n \in \N_0}$, that $\bar \tau_n(\bar \omega) = \tau_n(\omega)$, \mbox{$\p_\mu^\bfu$-a.s.} Then, we get that \mbox{$\p_\mu^\bfu$-a.s.}
	\begin{align*}
	\Phi(\pi^{\mu,\bfu}(\omega)) &= \sum_{n = 0}^{+\infty} e^{-\beta \tau_n(\omega)} g\bigl(\pi_{\tau_n(\omega)}^{\mu,\bfu} (\omega), u_n(\cdot + \tau_n(\omega)\bigr) \\
	&= \sum_{n = 0}^{+\infty} e^{-\beta \tau_n(\omega)} g\bigl(\pi_{\tau_n(\omega)}^{\mu,\bfu} (\omega), u_n(\cdot + \tau_n(\omega), \dots, \tau_n(\omega), Y_{\tau_n(\omega)}(\omega))\bigr),
	\end{align*}
	hence, comparing this result with (\ref{eq:functionalJdiscreteinf}) and applying the Fubini-Tonelli Theorem we obtain
	\begin{align*}
	J(\mu, \bfu) &= \int_\Omega \Phi(\pi^{\mu,\bfu}(\omega)) \p_\mu^\bfu(\dd \omega) =
	\int_{\bar\Omega} \Phi(\bar \omega) \bar \p_Q^\bfa(\dd \bar \omega) = \int_{\bar\Omega} \Phi(\bar \omega) \! \! \int_{\Delta_e} \! \! \bar \p_\nu^\bfa(\dd \bar \omega) \, Q(\dd \nu)\\
	&= \int_{\Delta_e} \biggl\{\int_{\bar\Omega} \Phi(\bar \omega) \bar \p_\nu^\bfa(\dd \bar \omega) \biggr\} Q(\dd \nu)= \int_{\Delta_e} \bar J(\nu, \bfa) \, Q(\dd \nu)
	\end{align*}
	by definition of the functional $\bar J$.
	
	To prove the second part of the theorem, fix $\mu \in \cP(I)$ and $\bfa = (a_n)_{n \in \N} \in \cA_{ad}$. We define, for each possible sequence $b_0, b_1, \dots \in O$ and $s_1, \dots \in (0, +\infty)$ the following quantities by recursion for all $n \in \N$
	\begin{align*}
	&p_0 = p_0(b_0) = H_{b_0}[\mu] \\
	&p_n = p_n(b_0, s_1, \dots, s_n, b_n) =\\ 
	&\begin{dcases}
	H_{b_n}\bigl[\Lambda\bigl(\phi_{p_{n-1}}^{a_{n-1}}(s_n^- - s_{n-1}), a_{n-1}(p_0, \dots, s_{n-1}, p_{n-1})(s_n^- - s_{n-1}; \, \cdot \,)\bigr)\bigr], &\text{if } s_1 < \dots < s_n \\
	\rho, &\text{otherwise}
	\end{dcases}	
	\end{align*}
	Here $\Lambda$ is the function defined in \eqref{eq:operatorLambdainf}, $s_0 = 0$ and $\rho \in \Delta_e$ is an arbitrarily chosen value.
	
	For all $n \in \N_0$ we define the functions $u_n \colon [0, +\infty) \times O \times \bigl((0, +\infty) \times O\bigr)^n \to \cP(U)$ as
	\begin{equation*}
	u_n(t, b_0, \dots, s_n, b_n \, ; \dd \fru) = 
	\begin{cases}
	a_n(p_0, \dots, s_n, p_n)(t - s_n \, ; \dd \fru), 	&\text{if } t \geq s_n, \\
	\fru,													&\text{if } t < s_n,
	\end{cases}
	\end{equation*}
	where $\fru \in U$ is some fixed value that is irrelevant to specify.
	Thanks to the fact that each of the functions $(b_0, \dots, s_n, b_n) \mapsto p_n$ is Borel-measurable and that $\cM$ is a Borel space, we can use \citep[Lemma 3(ii)]{yushkevich:jumpmodel} to conclude that all the functions $u_n$ are Borel-measurable and therefore $\bfu = (u_n)_{n \in \N_0} \in \cU_{ad}$.
	
	Now the proof follows the same steps of the first part. Equality between the laws of $\pi^{\mu, \bfu}$ under $\p_\mu^\bfu$ and $\bar \pi$ under $\bar\p_Q^\bfa$ is established by proving equivalence between the initial distributions of the two processes (that have not changed from the first part of the proof) and of the conditional distributions	
	\begin{align*}
	&\p_\mu^\bfu(\tau_n - \tau_{n-1} > t, \, \tau_{n-1} < +\infty \mid Y_0, \dots, \tau_{n-1}, Y_{\tau_{n-1}});\\
	&\bar\p_Q^\bfa(\bar\tau_n - \bar\tau_{n-1} > t, \, \bar\tau_{n-1} < +\infty \mid \bar Y_0, \dots, \bar \tau_{n-1}, \bar{Y}_{\bar\tau_{n-1}}); \\
	&\p_\mu^\bfu(\pi^{\mu,\bfu}_{\tau_n} \in D, \, \tau_n < +\infty \mid Y_0, \dots, Y_{\tau_{n-1}}, \tau_n);\\
	&\bar\p_Q^\bfa(\bar\pi_{\bar\tau_n} \in D, \, \bar\tau_n < +\infty \mid \bar Y_0, \dots, \bar{Y}_{\bar\tau_{n-1}}, \bar \tau_n),
	\end{align*}
	where $t \geq 0$, $D \in \cB(\Delta_e)$ and $n \in \N$.
	
	Finally, to prove (\ref{eq:functionalJequalityinf}) it suffices to define $\Phi \colon \bar \Omega \to \R$ as
	\begin{equation*}
	\Phi(\bar \omega) = 
	\sum_{n = 0}^{+\infty} e^{-\beta \bar \tau_n(\bar \omega)} g\bigl(\bar \pi_{\bar \tau_n(\bar \omega)} (\bar \omega), u_n(\cdot + \bar\tau_n, \bar Y_0(\bar \omega), \dots, \bar \tau_n(\bar \omega), \bar Y_{\bar\tau_n(\bar \omega)}(\bar \omega)\bigr).
	\end{equation*}
	and notice that $p_n(\bar Y_0, \dots, \bar\tau_n, \bar Y_{\bar\tau_n}) = \bar \pi_{\bar \tau_n}$, so that we can write 
	\begin{equation*}
	\Phi(\bar \omega) = 
	\sum_{n = 0}^{+\infty} e^{-\beta \bar \tau_n(\bar \omega)} g\bigl(\bar \pi_{\bar \tau_n(\bar \omega)} (\bar \omega), a_n(\bar \pi_0(\bar \omega), \dots, \bar \tau_n(\bar \omega), \bar \pi_{\bar\tau_n(\bar \omega)}(\bar \omega)\bigr).
	\end{equation*}
	The desired equality follows from the same reasoning as in the first part of the proof.
\end{proof}

\begin{rem}
	It is clear that the class $\cA_{ad}$ is strictly larger than the corresponding class of \emph{piecewise open-loop controls}, commonly adopted in PDP optimal control problems. This is due to the fact that policies in $\cA_{ad}$ can be past-dependent, while a piecewise open-loop control depends only on the time elapsed since the last jump, on the last post jump location and (in our case) on the last jump time.
	Nonetheless, we will see in Theorem \ref{th:vfixedpoint} that an optimal policy exists and is a \emph{stationary policy}. These policies, hereunder defined, correspond to piecewise open-loop controls, as the reader may easily check.
\end{rem}

\begin{definition}
	A policy $\bfa \in \cA_{ad}$ is said to be \emph{stationary} if it is of the form $\bfa = (a_0, a, a, \dots)$, where $a_0 \colon \Delta_e \to \cM$ and $a \colon (0,+\infty) \times \Delta_e \to \cM$ are measurable functions, the former depending on the starting point of the filtering process and the latter depending, between two consecutive jump times, on the last jump time and post jump location.
\end{definition}

We conclude this paper by proving that the original and the separated optimal control problems are equivalent. We need, first, to provide a characterization of the value function $v$ of the separated problem as the unique fixed point of the contraction mapping $\cG \colon \dB_b(\Delta_e) \to \dB_b(\Delta_e)$, defined as
\begin{multline}\label{eq:operatorGinf}
	\cG w(\nu) \coloneqq \inf_{\alpha \in A} \int_0^\infty e^{-\beta t} \chi_\nu^{\alpha}(t) \biggl[\phi_\nu^{\alpha}(f; \, \alpha(t), t) \\
	+ r(\phi_\nu^{\alpha}(t), \alpha(t))
	\int_{\Delta_e} w(p) R(\phi_\nu^{\alpha}(t), \alpha(t); \dd p) \biggr] \, \dd t, \quad \nu \in \Delta_e
\end{multline}
where the infimum is taken among all possible ordinary controls in the set $A$, defined in (\ref{eq:ordinarycontrolsinf}). It is rather easy to show that $\cG$ is a contraction.

To prove this characterization of $v$, we need to show first a weaker form of the Feller property for the controlled transition kernel $R$ defined in~\eqref{eq:PDPcharacteristictripleinf}.

\begin{proposition}\label{prop:weakFeller}
Let Assumption~\ref{assumption:lambdainf} hold. Then, for every bounded and continuous function $w \colon \Delta_e \to \R$ the map $\eta \mapsto r(\eta, u) \int_{\Delta_e} w(p) R(\eta, u; \dd p)$ is continuous on $\Delta_e$ uniformly in $u \in U$.
\end{proposition}

\begin{proof}
Fix $u \in U$ and $\eta \in \Delta_e$, i.e., $\eta \in \Delta_y$ for some $y \in O$. We want to show that for any $\epsilon > 0$ we can find a number $\delta > 0$, not depending on $u$, such that:
\begin{equation}\label{eq:weakFellerstatement}
	\biggl|r(\eta, u) \int_{\Delta_e} w(p) R(\eta, u; \dd p) - r(\theta, u) \int_{\Delta_e} w(p) R(\theta, u; \dd p)\biggr| < \epsilon
\end{equation}
for all $\theta \in \Delta_e$ such that $\norm{\eta - \theta}_{TV} < \delta$.

Let us fix $\epsilon > 0$. We consider, first, the case where $\rho(\eta, u) > 0$. It is easy to see that the function $\eta \mapsto r(\eta,u)$ is continuous on $\Delta_e$ uniformly with respect to $u \in U$. Therefore, choosing $\epsilon$ small enough, we can consider $\rho(\theta,u) > 0$. In this case, \eqref{eq:weakFellerstatement} corresponds to:
\begin{equation}\label{eq:weakFellerstatement2}
\begin{aligned}
	&\mathrel{\phantom{=}} \biggl|\int_I \int_I \ind_{h^{-1}(y)^c}(z) \Bigl\{w\bigl(H_{h(z)}[\Lambda(\eta,u)]\bigr) \, \lambda(x,u,\dd z) \, \eta(\dd x) \Bigr. \biggr. \\
	&\qquad - \biggl. \Bigl. w\bigl(H_{h(z)}[\Lambda(\theta,u)]\bigr) \, \lambda(x,u,\dd z) \, \theta(\dd x)\Bigr\}\biggr| \\
	&= \biggl|\int_I \Bigl\{w\bigl(H_{h(z)}[\Lambda(\eta,u)]\bigr) \Lambda(\eta,u)(\dd z) - w\bigl(H_{h(z)}[\Lambda(\theta,u)]\bigr) \Lambda(\theta,u)(\dd z) \Bigr\}\biggr| < \epsilon
\end{aligned}
\end{equation}
Let $B \coloneqq \supp\bigl(\Lambda(\eta,u)\bigr)$ and $B_0 \coloneqq B^c \cap h^{-1}(y)^c$. Then, repeatedly applying the Fubini-Tonelli theorem, the last line~\eqref{eq:weakFellerstatement2} can be rewritten and estimated as follows:
\begin{equation}\label{eq:weakFellerstatement3}
\begin{aligned}
	&\mathrel{\phantom{\leq}} \biggl|\int_B w\bigl(H_{h(z)}[\Lambda(\eta,u)]\bigr) \int_I \lambda(x,u,\dd z) \, \bigl[\eta(\dd x) - \theta(\dd x) \bigr] \biggr. \\
	&\qquad + \biggl. \int_B \Bigl\{w\bigl(H_{h(z)}[\Lambda(\eta,u)]\bigr) - w\bigl(H_{h(z)}[\Lambda(\theta,u)]\bigr)\Bigr\} \Lambda(\theta,u)(\dd z) \biggr. \\
	&\qquad - \biggl. \int_{B_0} w\bigl(H_{h(z)}[\Lambda(\theta,u)]\bigr) \Lambda(\theta,u)(\dd z) \Bigr\}\biggr| \\
	&\leq 2 C_\lambda \sup_{p \in \Delta_e} \abs{w(p)} \, \norm{\eta- \theta}_{TV} \\
	&\qquad + \int_B \Bigl|w\bigl(H_{h(z)}[\Lambda(\eta,u)]\bigr) - w\bigl(H_{h(z)}[\Lambda(\theta,u)]\bigr)\Bigr| \Lambda(\theta,u)(\dd z),
\end{aligned}
\end{equation}
where $C_\lambda \coloneqq \sup_{(x,u) \in I \times U}\lambda(x,u) $ is finite thanks to point~\ref{assumption:lambdabdd} of Assumption~\ref{assumption:lambdainf}.

It is clear that the first summand can be made smaller than $\frac \epsilon 2$, choosing $\norm{\eta- \theta}_{TV} < \delta$, for some $\delta > 0$ depending only on the constant $C_\lambda$ and the function $w$. We need to show that this is also the case for the second summand. 

Fix $z \in B$. Since $w$ is continuous, it is clear that there exists $\tilde\delta >0$ such that, if $\norm{H_{h(z)}[\Lambda(\eta,u)] - H_{h(z)}[\Lambda(\theta,u)]}_{TV} < \tilde\delta$, then $\bigl|w\bigl(H_{h(z)}[\Lambda(\eta,u)]\bigr) - w\bigl(H_{h(z)}[\Lambda(\theta,u)]\bigr)\bigr| < \frac\epsilon 2$. \emph{A priori} $\tilde \delta$ may depend (apart from $\epsilon$) also on $z,\, u$. However, if we are able to show that the map $\eta \mapsto H_{h(z)}[\Lambda(\eta,u)]$ is continuous for all $z \in B$, uniformly with respect to $u$, then it follows that there exists $\delta > 0$ such that, if $\norm{\eta - \theta}_{TV} < \delta$, then $\norm{H_{h(z)}[\Lambda(\eta,u)] - H_{h(z)}[\Lambda(\theta,u)]}_{TV} < \tilde\delta$, and by dominated convergence the claim follows immediately.

To show this fact, let us call $\tilde y \coloneqq h(z)$ and notice that $\eta \mapsto H_{\tilde y}[\Lambda(\eta,u)]$ is a function defined from $\Delta_e$ into itself and that, in particular, $H_{\tilde y}[\Lambda(\eta,u)] \in \Delta_{\tilde y}$. Moreover, from Definition~\ref{def:operatorH} (recall that $z \in B$), we have that for all $A \in \cI$:
\begin{equation*}
\bigl|H_{\tilde y}[\Lambda(\eta,u)](A) - H_{\tilde y}[\Lambda(\theta,u)](A)\bigr| = \Bigl|\frac{\ind_A \Lambda(\eta,u) \circ h^{-1}(\dd \upsilon)}{\Lambda(\eta,u) \circ h^{-1}(\dd \upsilon)}(\tilde y) - \frac{\ind_A \Lambda(\theta,u) \circ h^{-1}(\dd \upsilon)}{\Lambda(\theta,u) \circ h^{-1}(\dd \upsilon)}(\tilde y)\Bigr|.
\end{equation*} 
Since $\norm{\eta}_{TV} \leq 2 \sup_{A \in \cI} \abs{\eta(A)}$ (see~\eqref{eq:totalvariationequiv}), we have to show that we can make the last absolute value small as we wish for any $A \in \cI$.

To prove this, we resort to this useful fact from measure theory. Since $O$ is a complete and separable metric space, its Borel $\sigma$-algebra $\cO$ is countably generated. Hence, by \citep[Th. 6.5.5]{bogachev:measth} there exists a measurable function $\psi \colon O \to [0,1]$ such that $\cO = \bigl\{\psi^{-1}(B) \colon B \in \cB([0,1])\bigr\}$. Therefore, it is enough to prove the last assertion in the case $O = [0,1]$, which is done in Lemma~\ref{lemma:Hcontinuity} in Appendix~\ref{app:operatorH}. In fact, we can reduce the general case to this one as follows. We define the measures
\begin{align*}
\nu_{A, \eta, u}(\dd \upsilon) &\coloneqq \ind_A \Lambda(\eta,u) \circ h^{-1} \circ \psi^{-1}(\dd \upsilon) \ll \Lambda(\eta,u) \circ h^{-1} \circ \psi^{-1}(\dd \upsilon) \eqqcolon \nu_{\eta, u}(\dd \upsilon) \\
\nu_{A, \theta, u}(\dd \upsilon) &\coloneqq \ind_A \Lambda(\theta,u) \circ h^{-1} \circ \psi^{-1}(\dd \upsilon) \ll \Lambda(\theta,u) \circ h^{-1} \circ \psi^{-1}(\dd \upsilon) \eqqcolon \nu_{\theta, u}(\dd \upsilon).
\end{align*}
Thanks to the respective absolute continuity we may define the Radon-Nikod\'ym derivatives $g_{A, \eta,u}(t) \coloneqq \frac{\nu_{A, \eta, u}(\dd \upsilon)}{\nu_{\eta, u}(\dd \upsilon)}$ and $g_{A, \theta,u}(t) \coloneqq \frac{\nu_{A, \theta, u}(\dd \upsilon)}{\nu_{\theta, u}(\dd \upsilon)}$ and use Lemma~\ref{lemma:Hcontinuity} for each fixed $A \in \cI$. We can now define the $\cO$-measurable functions $f_{A, \eta, u}(y) \coloneqq g_{A, \eta,u}(\psi(y))$ and $f_{A, \theta,u}(y) \coloneqq g_{A, \theta,u}(\psi(y))$. Since for any $B \in \cO$ there exists $C \in \cB([0,1])$ such that $B = \psi^{-1}(C)$, we get:
\begin{multline*}
\ind_A \Lambda(\eta,u) \circ h^{-1}(B) = \nu_{A, \eta, u}(C) = \int_{[0,1]} \ind_C(t) g_{A, \eta,u}(t) \, \nu_{\eta, u}(\dd t) \\
= \int_O \ind_C(\psi(y)) g_{A, \eta,u}(\psi(y)) \, \Lambda(\eta,u) \circ h^{-1}(\dd y) = \int_O \ind_B(y) f_{A, \eta,u}(y) \, \Lambda(\eta,u) \circ h^{-1}(\dd y).
\end{multline*}
Hence $H_{\tilde y}[\Lambda(\eta,u)](A) = f_{A, \eta,u}(\tilde y)$ and, similarly, $H_{\tilde y}[\Lambda(\theta,u)](A) = f_{A, \theta,u}(\tilde y)$.

Summing up, we have shown that for any $\tilde \epsilon > 0$, any $u \in U$ and any $A \in \cI$
\begin{equation*}
	\bigl|H_{\tilde y}[\Lambda(\eta,u)](A) - H_{\tilde y}[\Lambda(\theta,u)](A)\bigr| < \tilde \epsilon
\end{equation*}
if we choose $\tilde \delta > 0$ (not depending on $u, A$) such that $\norm{\eta - \theta}_{TV} < \tilde \delta$. Therefore, for each $z \in B$ the map $\eta \mapsto H_{h(z)}[\Lambda(\eta,u)]$ is continuous uniformly with respect to $u$ and the claim follows applying the dominated convergence theorem to the last line of~\eqref{eq:weakFellerstatement3}.

We are left to prove~\eqref{eq:weakFellerstatement} in the case $\rho(\eta,u) = 0$. If we have that also $\rho(\theta,u) = 0$, then the claim is trivial. If, instead, $\rho(\theta,u) > 0$, then
\begin{equation*}
\begin{aligned}
	&\mathrel{\phantom{=}} \biggl|r(\eta, u) \int_{\Delta_e} w(p) R(\eta, u; \dd p) - r(\theta, u) \int_{\Delta_e} w(p) R(\theta, u; \dd p)\biggr| \\
	&= \biggl|r(\theta, u) \int_{\Delta_e} w(p) R(\theta, u; \dd p)\biggr|
	=\biggl|\int_I \int_I \ind_{h^{-1}(y)^c}(z) w\bigl(H_{h(z)}[\Lambda(\theta,u)]\bigr) \, \lambda(x,u,\dd z) \, \theta(\dd x)\biggr| \\
	&=\biggl|\int_I w\bigl(H_{h(z)}[\Lambda(\theta,u)]\bigr) \Lambda(\theta,u)(\dd z) \biggr|
	\leq \sup_{p \in \Delta_e} \abs{w(p)} \, \Lambda(\theta,u)\bigl(h^{-1}(y)^c\bigr).
\end{aligned}
\end{equation*}
Since $0 = \rho(\eta,u) = \Lambda(\eta,u)\bigl(h^{-1}(y)^c\bigr)$, we easily get
\begin{equation*}
\Lambda(\theta,u)\bigl(h^{-1}(y)^c\bigr) = \bigl|\Lambda(\theta,u)\bigl(h^{-1}(y)^c\bigr) - \Lambda(\eta,u)\bigl(h^{-1}(y)^c\bigr)\bigr| \leq C_\lambda \, \norm{\eta-\theta}_{TV},
\end{equation*}
hence~\eqref{eq:weakFellerstatement} holds as long as we choose $\delta = \epsilon \bigl[C_\lambda \, \sup_{p \in \Delta_e} \abs{w(p)}\bigr]^{-1}$, which does not depend on $u$.
\end{proof}

%
%

We can now proceed and prove the characterization of the value function $v$ of the separated problem as the unique fixed point of the contraction mapping $\cG$ defined in~\eqref{eq:operatorGinf}.
We can reduce the separated problem to a \emph{Markov Decision Process}, by a construction analogous to that provided in \citep{calvia:optcontrol}. Therefore, we can invoke standard results from \citep{bertsekas:stochoptcontrol}, in particular those connected to the so called \emph{lower semicontinuous models}, of which our problem is an instance. Thanks to those results and the similar ones proved in \citep{calvia:optcontrol}, we can state the following Theorem, whose proof is omitted.
\begin{theorem}\label{th:vfixedpoint}
	Under Assumptions \ref{assumption:lambdainf} and \ref{assumption:costfunctioninf} we have that:
	\begin{enumerate}
		\item There exists an optimal stationary policy $\bfa^\star \in \cA_{ad}$, \ie a policy such that 
		$$v(\nu) = \bar J(\nu, \bfa^\star), \quad \text{for all } \nu \in \Delta_e.$$
		\item The value function $v$ is lower semicontinuous and is the unique fixed point of $\cG$.
	\end{enumerate}
\end{theorem}

We can finally prove the equivalence of the original and the separated optimal control problems.
\begin{theorem} \label{th:valuefunctionsidentifinf}
	Under Assumptions \ref{assumption:lambdainf} and \ref{assumption:costfunctioninf}, for all $\mu \in \cP(I)$ we have that
	\begin{equation}
	V(\mu) = \int_O v(H_y[\mu]) \, \mu \circ h^{-1}(\dd y).
	\end{equation}
\end{theorem}
\begin{proof}
	First of all, let us notice that, for all fixed $\mu \in \cP(I)$, the map $y \mapsto v(H_y[\mu])$ is integrable with respect to the probability measure $\mu \circ h^{-1}$, being bounded and measurable. In fact, its boundedness is granted by Assumption \ref{assumption:costfunctioninf} and its measurability follows from the lower semicontinuity of $v$, proved in (2) of Theorem \ref{th:vfixedpoint}, and the fact that $y \mapsto H_y[\mu]$ is measurable by construction of the operator $H$.
	
	Recall that we know from Theorem \ref{th:costfunctionalidentifinf} that for all $\mu \in \cP(I)$
	\begin{equation*}
	J(\mu, \bfu) = \int_O \bar J(H_y[\mu], \bfa) \, \mu \circ h^{-1}(\dd y),
	\end{equation*}
	where $\bfu \in \cU_{ad}$ and $\bfa \in \cA_{ad}$ are corresponding admissible controls and admissible policies.
	
	Let now $\mu \in \cP(I)$ be fixed. It is obvious that $V(\mu) \geq \int_O v(H_y[\mu]) \, \mu \circ h^{-1}(\dd y)$. In fact, since $\bar J(H_y[\mu], \bfa) \geq v(H_y[\mu])$ for all $\bfa \in \cA_{ad}$ and all $y \in O$, we get that for all $\bfu \in \cU_{ad}$
	\begin{equation*}
	J(\mu, \bfu) \geq \int_O v(H_y[\mu]) \, \mu \circ h^{-1}(\dd y),
	\end{equation*}
	and we get the desired inequality by taking the infimum on the left hand side with respect to all $\bfu \in \cU_{ad}$.
	
	The reverse inequality is easily obtained by taking an optimal policy $\bfa^\star \in \cA_{ad}$ (whose existence is guaranteed by (1) of Theorem \ref{th:vfixedpoint}) and considering its corresponding admissible control $\bfu^\star \in \cU_{ad}$. From Theorem \ref{th:costfunctionalidentifinf} we immediately get that
	\begin{equation*}
	V(\mu) \leq J(\mu, \bfu^\star) = \int_O \bar J(H_y[\mu], \bfa^\star) \, \mu \circ h^{-1}(\dd y) =  \int_O v(H_y[\mu]) \, \mu \circ h^{-1}(\dd y). \qedhere
	\end{equation*}
\end{proof}

\begin{rem}
As it is known from the literature, we could search for other characterizations of the value function $v$, most notably, as the unique solution of a \emph{Hamilton-Jacobi-Bellman} equation (or HJB for short). In this case, it is the following integro-differential infinite dimensional equation
\begin{equation}\label{eq:HJB}
	H(\nu, \dD v(\nu), v) + \beta v(\nu) = 0, \quad \nu \in \Delta_e,
\end{equation}
with the \emph{hamiltonian function} $H\colon \Delta_e \times \cM(I)^\star \times \mathrm{UC}_b(\Delta_e) \to \R$ defined as\footnote{$\inprod{\cdot}{\cdot}$ denotes the duality pairing between $\cM(I)$ and $\cM(I)^\star$ and $\mathrm{UC}_b(\Delta_e)$ is the space of uniformly continuous and bounded real-valued functions on $\Delta_e$.}
\begin{equation*}
	H(\nu, g, w) \coloneqq \sup_{u \in U} \biggl\{- \inprod{F(\nu,u)}{g} - \nu(f; \, u) -r(\nu,u) \int_{\Delta_e} \bigl[w(p) - w(\nu)\bigr] \, R(\nu, u; \dd p)\biggr\}.
\end{equation*}
As usual, one can expect that classical solutions to the HJB are hard (if not impossible) to find, i.e., continuous and Fr\'echet differentiable functions $v \colon \cM(I) \to \R$ that satisfy~\eqref{eq:HJB}. It is even more unlikely that the value function $v$ of the separated optimal control problem, defined in~\eqref{eq:pdpvaluefunctioninf}, is the unique classical solution. What is more viable is to try to prove that $v$ is the unique \emph{viscosity solution} to the HJB, in a suitable sense. The theory of viscosity solutions, pioneered by M.~G.~Crandall and P.-L.~Lions (see \eg \citep{crandall:usersguide, barles:viscosite} and \citep{flemingsoner:controlledmarkov} for connections with optimal control problems), has provided a sound tool to give meaning to solutions of HJB equations, otherwise ill-posed in the classical sense. The theory has been recently developed also in the infinite dimensional case, for instance, in Hilbert spaces (see, e.g. \citep{fabbri:soc}) and in connection with the mean-field games theory (see, e.g. \citep{bensoussan:meanfield}).

Another characterization that one may seek requires the use of \emph{Backward Stochastic Differential Equations} (or BSDEs for short). One can try to characterize the value function $v$ of the separated problem (or, even more directly, the value function $V$ of the original control problem, defined in~\eqref{eq:valuefunctioninf}) as the unique solution, in some suitable sense, of a BSDE through the so called \emph{randomization method}, introduced in~\citep{kharroubi2015:IPDE} for classical Markovian models. This method requires the introduction of a further \emph{randomized} control problem, to prove that it is equivalent to the original optimal control problem (in the sense that the two value functions are equal) and then to prove that its value function is the unique solution, in a prescribed sense, to the randomized BSDE. It is a powerful technique, because the assumptions on the data of the problem are rather weak. However, it has the disadvantage that no information about any optimal control can be deduced. This method has also been adopted in infinite dimensional settings (as in \citep{bandini2018:JD}), in optimal control problems for pure jump Markov processes (see, e.g. \citep{bandini:constBSDE}) and in optimal control problems with partial observation (as in \citep{bandini:randomizHJBwasserstein, bandini2018:randBSDE}).

These further characterizations of the value function $v$ of the separated optimal control problem are left for future research. 
\end{rem}

\bibliographystyle{plainnat} 
\bibliography{Bibliography}

\appendix

\phivarphi

\section{The operator $H$} \label{app:operatorH}

We first take a closer look at the construction of the operator $H$, that is based on the existence result provided in Proposition \ref{prop:Hprobmeasure}. To prove it we need the following Lemma.

\begin{lemma}\label{lemma:Hprobmeasure}
	Suppose that $I$ is a compact metric space and fix $\mu \in \cM_+(I)$. For each $\phi \in \dC(I)$ take a version (\ie any function in the equivalence class) of the Radon-Nikodym derivative of $\phi \mu \circ h^{-1}$ with respect to $\mu \circ h^{-1}$ and define, for fixed $y \in O$, the functional $L_y \colon \dC(I) \to \R$ as
	\begin{equation*}
		L_y(\phi) \coloneqq \frac{\phi \mu \circ h^{-1}(\dd \upsilon)}{\mu \circ h^{-1}(\dd \upsilon)}(y), \quad \phi \in \dC(I).
	\end{equation*}
	If $y \in \supp(\mu \circ h^{-1})$ there exists a unique probability measure $\rho_y$ on $(I, \cI)$ such that
	\begin{equation*}
		L_y(\phi) = \int_I \phi(z) \, \rho_y(\dd z), \quad \phi \in \dC(I).
	\end{equation*}
\end{lemma} 

\begin{rem}
	The hypothesis that the point $y$ belongs to the support of the measure $\mu \circ h^{-1}$ ensures that the functional $L$ is not zero on the whole space $\dB_b(I)$ (for instance, it takes value $1$ on the function $\phi = \one$, as will be proved). To see what happens if this is not the case, consider the example below.
\end{rem}

\begin{example}
	Let $\mu \in \cM_+(I), \, O = [0,1]$ and a point $y \in O, \, y \notin \supp(\mu \circ h^{-1})$. Take $(A_n)_{n \in \N} \subset \cO$ the sequence of open intervals given by
	$$A_n \coloneqq \Bigl(y - \frac{1}{n}, y + \frac{1}{n}\Bigr) \cap O, \quad n \in \N.$$
	By definition of support and set inclusion, we have that there exists a natural number $\bar n \in \N$ such that $\mu \circ h^{-1}(A_n) = 0$ for all $n \geq \bar n$. Since $\phi \mu \circ h^{-1}$ is absolutely continuous with respect to $\mu \circ h^{-1}$, we also have that $\phi \mu \circ h^{-1}(A_n) = 0$ for all $n \geq \bar n$ and for all $\phi \in \dB_b(I)$. Therefore, thanks to \citep[Theorem 5.8.8.]{bogachev:measth} and under the usual convention $\frac{0}{0} \coloneqq 0$ we have that
	$$\frac{\phi \mu \circ h^{-1}(\dd \upsilon)}{\mu \circ h^{-1}(\dd \upsilon)}(y) = \lim_{n \to \infty} \frac{\phi \mu \circ h^{-1}(A_n)}{\mu \circ h^{-1}(A_n)} = 0$$
	and we obtain $L_y(\phi) = 0$ for all $\phi \in \dB_b(I)$.
\end{example}

\begin{proof}[Proof of Lemma \ref{lemma:Hprobmeasure}]
	Fix $\mu \in \cM_+(I)$ and $y \in O$ such that $y$ is in the support of $\mu \circ h^{-1}$. Let $\cC$ denote a countable dense subset of the set $\dC(I)$ of continuous functions on $I$, containing the constant function equal to $1$ (denoted by $\one$) and such that $\cC$ is a vector space over $\Q$.
	The result will follow from an application of a slight modification of the Riesz Representation Theorem to the functional $L_y$ (see \citep[Par. 88]{williams:diffusions} for further details). What we need to prove is that (remember: $y \in \supp(\mu \circ h^{-1})$):
	\begin{enumerate}
		\item $L_y$ is a linear functional on $\cC$ (as a vector space over $\Q$).
		\item $L_y(\phi) \leq L_y(\psi)$, whenever $\phi \leq \psi, \, \phi, \psi \in \cC$.
		\item $L_y(\one) = 1$.
	\end{enumerate}	
	\noindent \textbf{Claim 1.} Let $\alpha, \beta \in \Q$ and $\phi, \psi \in \cC$ be fixed and let us define 
	$$g(\upsilon) \coloneqq \frac{(\alpha \phi + \beta \psi) \mu \circ h^{-1}(\dd \upsilon)}{\mu \circ h^{-1}(\dd \upsilon)}(\upsilon), \quad \upsilon \in O.$$ 
	By definition of Radon-Nikodym derivative we have that for all $B \in \cO$
	$$ \int_B g(\upsilon) \, \mu \circ h^{-1}(\dd \upsilon) = (\alpha \phi + \beta \psi) \mu \circ h^{-1}(B) = \alpha \phi \mu \circ h^{-1}(B) + \beta \psi \mu \circ h^{-1}(B).$$
	Therefore, setting $g_1 \coloneqq \frac{\alpha \phi \mu \circ h^{-1}(\dd \upsilon)}{\mu \circ h^{-1}(\dd \upsilon)}$ and $g_2 \coloneqq \frac{\beta \psi \mu \circ h^{-1}(\dd \upsilon)}{\mu \circ h^{-1}(\dd \upsilon)}$ we get that $g = g_1 + g_2$ except on a $(\mu \circ h^{-1})$-null measure set $C$. On this set we may redefine, for instance, $g_1(\upsilon) \coloneqq g(\upsilon)$ and $g_2(\upsilon) = 0, \, \upsilon \in C$ to have that $g = g_1 + g_2$ for all $\upsilon \in O$, whence $L_y(\alpha \phi + \beta \psi) = \alpha L_y(\phi) + \beta L_y(\psi)$.
	
	\noindent \textbf{Claim 2.} By linearity, this is equivalent to prove that $L_y(\phi) \geq 0$, for all $\phi \in \cC, \, \phi \geq 0$. It is immediate to see that, for all $\phi \geq 0$, we have that $\phi \mu \circ h^{-1} \in \cM_+(I)$, hence $g \coloneqq \frac{\phi \mu \circ h^{-1}(\dd \upsilon)}{\mu \circ h^{-1}(\dd \upsilon)} \geq 0, \, \mu \circ h^{-1}$-a.e.\,. Redefining $g$ to be zero on the $\mu \circ h^{-1}$-null measure set $C \in \cO$ where this does not happen, we get that $L_y(\phi) \geq 0$.
	
	\noindent \textbf{Claim 3.} If $\{y\}$ is an atom for $\mu$ then the result is obvious. Otherwise, we can consider $\mu$ to be atomless, without loss of generality.
	
	Consider first the case $O = [0,1]$. Define $(A_n)_{n \in \N} \subset \cO$ to be the sequence of open intervals given by
	$$A_n \coloneqq \Bigl(y - \frac{1}{n}, y + \frac{1}{n}\Bigr) \cap O, \quad n \in \N.$$
	Then, by definition of support, we have that for all $n \in \N$
	$$\mu \circ h^{-1}(A_n) = \mu\bigl(h^{-1}(A_n)\bigr) > 0.$$
	Therefore, thanks to \citep[Theorem 5.8.8.]{bogachev:measth} we have that
	$$\frac{\mu \circ h^{-1}(\dd \upsilon)}{\mu \circ h^{-1}(\dd \upsilon)}(y) = \lim_{n \to \infty} \frac{\mu\bigl(h^{-1}(A_n)\bigr)}{\mu\bigl(h^{-1}(A_n)\bigr)} = 1.$$
	
	The case of $O$ being a complete and separable metric space is treated by reducing it to the previous case. This is possible since we are taking the $\sigma$-algebra $\cO$ on $O$ as the Borel one, hence we know that $\cO$ is countably generated and countably separated. Then, by \citep[Theorem 6.5.5.]{bogachev:measth}, there exists a measurable function $\psi \colon O \to [0,1]$ such that
	$$\cO = \{\psi^{-1}(B) \colon B \in \cB\bigl([0,1]\bigr)\}$$
	and, by \citep[Theorem 6.5.7.]{bogachev:measth}, we know that this function is injective.

	We are now in a position to apply the Riesz Representation Theorem to the functional $L_y$ and say that there exists a unique probability measure $\rho_y$ on $(I, \cI)$ such that
	\begin{equation*}
	L_y(\phi) = \int_I \phi(z) \, \rho_y(\dd z), \quad \phi \in \cC.
	\end{equation*}
	We get the same equality for all $\phi \in \dC(I)$ by uniform convergence.
\end{proof}

We now focus on the continuity of the map $\eta \mapsto H_y[\Lambda(\eta,u)]$, where $\Lambda$ is defined in~\eqref{eq:operatorLambdainf}. We recall that the following Lemma is required in the proof of Proposition~\ref{prop:weakFeller}.
\begin{lemma}\label{lemma:Hcontinuity}
Let Assumption~\ref{assumption:lambdainf} hold. Assume, moreover, that $O = [0,1]$ and fix $y \in O$. Then, for each $z \in \supp\bigl(\Lambda(\eta,u)\bigr)$ the function $\eta \mapsto H_{h(z)}[\Lambda(\eta,u)]$ is continuous on $\Delta_y$, uniformly with respect to $u \in U$.
\end{lemma}
\begin{proof}
	Fix $\epsilon > 0$, $u \in U$ and $z \in \supp\bigl(\Lambda(\eta,u)\bigr)$. We need to show that $\norm{H_{h(z)}[\Lambda(\eta,u)] - H_{h(z)}[\Lambda(\theta,u)]}_{TV} < \epsilon$ as long as $\norm{\eta - \theta}_{TV} < \delta$ for some $\delta > 0$, not depending on $u$. Since $\norm{\mu}_{TV} \leq 2 \sup_{A \in \cI} |\mu(A)|$ (see~\eqref{eq:totalvariationequiv}), it is sufficient to prove that for any $A \in \cI$ we have that $\bigl|H_{h(z)}[\Lambda(\eta,u)](A) - H_{h(z)}[\Lambda(\theta,u)](A)\bigr| < \epsilon$. 
	
	Let us fix $A \in \cI$ and define $\tilde y = h(z)$. Recalling Definition~\ref{def:operatorH} and thanks to~\citep[Th. 5.8.8]{bogachev:measth}, we can write
\begin{multline}
H_{\tilde y}[\Lambda(\eta,u)](A) - H_{\tilde y}[\Lambda(\theta,u)](A) =
\frac{\ind_A \Lambda(\eta,u) \circ h^{-1}(\dd \upsilon)}{\Lambda(\eta,u) \circ h^{-1}(\dd \upsilon)}(\tilde y) - \frac{\ind_A \Lambda(\theta,u) \circ h^{-1}(\dd \upsilon)}{\Lambda(\theta,u) \circ h^{-1}(\dd \upsilon)}(\tilde y) \\
=
\lim_{n \to \infty} \Bigl[\frac{\ind_A \Lambda(\eta,u) \circ h^{-1}(B_n)}{\Lambda(\eta,u) \circ h^{-1}(B_n)} - \frac{\ind_A \Lambda(\theta,u) \circ h^{-1}(B_n)}{\Lambda(\theta,u) \circ h^{-1}(B_n)}\Bigr],
\end{multline}
where $B_n = \bigl(\tilde y - \frac 1n, \tilde y + \frac 1n\bigr) \cap O, \, n \in \N$. Notice that the fractions are all well defined since, by definition of support, the sets $h^{-1}(B_n)$ have all strictly positive measure, thanks to the fact that $z \in \supp\bigl(\Lambda(\eta,u)\bigr)$. It is clear that if the term inside the brackets can be made arbitrarily small uniformly with respect to $u,A$ and $n$, then the claim is proved. Hence, we concentrate ourselves on proving this last claim.

For any fixed $n \in \N$ we have, using the Fubini-Tonelli theorem
\begin{multline}\label{eq:LambdaRN}
\frac{\ind_A \Lambda(\eta,u) \circ h^{-1}(B_n)}{\Lambda(\eta,u) \circ h^{-1}(B_n)} - \frac{\ind_A \Lambda(\theta,u) \circ h^{-1}(B_n)}{\Lambda(\theta,u) \circ h^{-1}(B_n)}
= \frac{\Lambda(\eta,u)(C_n^A)}{\Lambda(\eta,u)(C_n)} - \frac{\Lambda(\theta,u)(C_n^A)}{\Lambda(\theta,u)(C_n)} \\
=
\frac{\Lambda(\eta,u)(C_n^A) \Lambda(\theta,u)(C_n) - \Lambda(\eta,u)(C_n) \Lambda(\theta,u)(C_n^A)}{\Lambda(\eta,u)(C_n) \Lambda(\theta,u)(C_n)},
\end{multline}
where $C_n \coloneqq h^{-1}(B_n \setminus \{y\})$ and $C_n^A \coloneqq C_n \cap A$.
Let us define $S \coloneqq \supp(\eta)$ and $S_0 \coloneqq S \setminus h^{-1}(y)$. By definition of the operator $\Lambda$ and adding and subtracting the measure $\theta$ from the integrals where the measure $\eta$ appears, we get that the numerator of the last fraction is equal to:
\begin{multline}\label{eq:Lambdaestimatenum}
\Lambda(\theta,u)(C_n) \int_S \lambda(x, u, C_n^A) \, [\eta - \theta](\dd x)
- \Lambda(\theta,u)(C_n^A) \int_S \lambda(x, u, C_n) \, [\eta - \theta](\dd x)\\
+ \int_{S_0} \lambda(x,u,C_n) \, \theta(\dd x) \int_S \lambda(x,u,C_n^A) \, \theta(\dd x)
- \int_{S_0} \lambda(x,u,C_n^A) \, \theta(\dd x) \int_S \lambda(x,u,C_n) \, \theta(\dd x).
\end{multline}
Taking the absolute value and noticing that, for any $A \in \cI$ and any $n \in \N$, $C_n^A \subset C_n$, we get the following estimate combining~\eqref{eq:LambdaRN} and~\eqref{eq:Lambdaestimatenum}:
\begin{multline}\label{eq:Lambdaestimate}
\Bigl|\frac{\Lambda(\eta,u)(C_n^A) \Lambda(\theta,u)(C_n) - \Lambda(\eta,u)(C_n) \Lambda(\theta,u)(C_n^A)}{\Lambda(\eta,u)(C_n) \Lambda(\theta,u)(C_n)}\Bigr| 
\\
\leq 2 \biggl[ \frac{\int_S \lambda(x,u,C_n) \, |\eta-\theta|(\dd x)}{\int_S \lambda(x,u,C_n) \, \eta(\dd x)} + \frac{\int_{S_0} \lambda(x,u,C_n) \, \theta(\dd x)}{\int_S \lambda(x,u,C_n) \, \eta(\dd x)}\biggr].
\end{multline}
Next, by definition of the total variation measure, i.e., $|\eta-\theta|(\dd x) = [\eta-\theta]^+(\dd x) + [\eta-\theta]^-(\dd x)$ (where, the superscripts \textsuperscript{$+$} and \textsuperscript{$-$} denote its positive and negative parts), we have that the restriction on $S_0$ of the negative part coincides with $\theta$, since this holds on any $\eta$-null measure set. Therefore, we can rewrite the term inside the brackets in the last line of~\eqref{eq:Lambdaestimate} as:
\begin{equation}\label{eq:Lambdabracketest}
\begin{aligned}
&\mathrel{\phantom{=}} \frac{\int_S \lambda(x,u,C_n) \, |\eta-\theta|(\dd x)}{\int_S \lambda(x,u,C_n) \, \eta(\dd x)} + \frac{\int_{S_0} \lambda(x,u,C_n) \, \theta(\dd x)}{\int_S \lambda(x,u,C_n) \, \eta(\dd x)} \\
&= \frac{\int_S \lambda(x,u,C_n) \, [\eta-\theta]^+(\dd x)}{\int_S \lambda(x,u,C_n) \, \eta(\dd x)} + \frac{\int_S \lambda(x,u,C_n) \, [\eta-\theta]^-(\dd x)}{\int_S \lambda(x,u,C_n) \, \eta(\dd x)} \\
&\qquad + \frac{\int_{S_0} \lambda(x,u,C_n) \, \theta(\dd x)}{\int_S \lambda(x,u,C_n) \, \eta(\dd x)} \\
&= \frac{\int_S \lambda(x,u,C_n) \, [\eta-\theta]^+(\dd x)}{\int_S \lambda(x,u,C_n) \, \eta(\dd x)} + \frac{\int_{h^{-1}(y)} \lambda(x,u,C_n) \, [\eta-\theta]^-(\dd x)}{\int_S \lambda(x,u,C_n) \, \eta(\dd x)} \\
&= \int_I \frac{\lambda(x,u,C_n)}{\int_S \lambda(x,u,C_n) \, \eta(\dd x)} \, |\eta - \theta|(\dd x).
\end{aligned}
\end{equation}
Suppose, for the time being, we proved that 
$$\displaystyle M \coloneqq \limsup_{n \to \infty} \, \esssup_{x \in I} \sup_{u \in U} \frac{\lambda(x,u,C_n)}{\int_S \lambda(x,u,C_n) \, \eta(\dd x)} < +\infty,$$
where the essential supremum is taken with respect to the total variation measure $|\eta-\theta|$.
Then, from~\eqref{eq:Lambdabracketest} we would get
\begin{equation}
\int_I \frac{\lambda(x,u,C_n)}{\int_S \lambda(x,u,C_n) \, \eta(\dd x)} \, |\eta - \theta|(\dd x) \leq M \norm{\eta - \theta}_{TV}.
\end{equation}
Therefore, collecting all the computations made so far, we would have
\begin{equation*}
\bigl|H_{\tilde y}[\Lambda(\eta,u)](A) - H_{\tilde y}[\Lambda(\theta,u)](A)\bigr| \leq 2 M \norm{\eta - \theta}_{TV} < \epsilon,
\end{equation*}
choosing $\delta = \frac{\epsilon}{2M}$. Since $M$ does not depend on $u$ and $A$, $\delta$ does not too, so that $\norm{H_{\tilde y}[\Lambda(\eta,u)] - H_{\tilde y}[\Lambda(\theta,u)]}_{TV} < \epsilon$, uniformly with respect to $u \in U$.

Let us notice that proving $M < +\infty$ is not immediate, because of the dependence on the sets $C_n$ of the term $\frac{\lambda(x,u,C_n)}{\int_S \lambda(x,u,C_n) \, \eta(\dd x)}$. In fact, thanks to Assumption~\ref{assumption:lambdainf} we know that for any $x \in I$ and any $u \in U$ this quantity is bounded, but we do not have any information about its behavior with respect to $n \in \N$. 

Suppose that $M = +\infty$. Then, for any $K > 0$, there exists a countable index set $J_K$ such that $\esssup_{x \in I} \sup_{u \in U} \dfrac{\lambda(x,u,C_{n_j})}{\int_S \lambda(x,u,C_{n_j}) \, \eta(\dd x)} > K$ for any $j \in J_K$. Moreover, for each $j \in J_K$ there exists a set $A_{j,K} \in \cI$ such that $|\eta-\theta|(A_{j,K}^c) = 0$ and $\sup_{u \in U} \dfrac{\lambda(x,u,C_{n_j})}{\int_S \lambda(x,u,C_{n_j}) \, \eta(\dd x)} > K$ for all $x \in A_{j,K}$. Finally, for each $x \in A_{j,K}$ there exists $\tilde u \in U$ such that 
\begin{equation}\label{eq:contradiction}
\lambda(x,\tilde u,C_{n_j}) > K \, \int_S \lambda(x,\tilde u,C_{n_j}) \, \eta(\dd x).
\end{equation}
There are three possible cases.

\noindent\textbf{Case 1.} If $\eta(A_{j,K}^c) > 0$, then from~\eqref{eq:contradiction} we immediately get
\begin{equation*}
\frac{\int_{A_{j,K}} \lambda(x,\tilde u,C_{n_j}) \, \eta(\dd x)}{\int_S \lambda(x,\tilde u,C_{n_j}) \, \eta(\dd x)} > K.
\end{equation*}
However, the fraction is always no greater than one, hence we reach a contradiction letting $K \uparrow +\infty$.

\noindent\textbf{Case 2.}
If $\eta(A_{j,K}^c) = 0$ and $\lambda(x,\tilde u,h^{-1}(\tilde y)) = 0$ for all $x \in A_{j,K}$ then $\int_S \lambda(x,\tilde u,h^{-1}(\tilde y)) \, \eta(\dd x) = 0$, hence, recalling that $C_n \downarrow h^{-1}(\tilde y)$, we have
\begin{equation*}
	\lambda(x,\tilde u,C_{n_j}) > K \, \int_S \lambda(x,\tilde u,C_{n_j}) \, \eta(\dd x) \geq K \int_S \lambda(x,\tilde u,h^{-1}(\tilde y)) \, \eta(\dd x) = 0
\end{equation*}
for any $j \in J_K$ and any $x \in A_{j,K}$. Therefore, for any $j \in J_K$
\begin{equation*}
\int_S \lambda(x,\tilde u,C_{n_j}) \, \eta(\dd x) = \int_{A_{j,K}} \lambda(x,\tilde u,C_{n_j}) \, \eta(\dd x) > 0,
\end{equation*}
whence, by dominated convergence
\begin{multline*}
0 < \limsup_{n \to \infty} \int_S \lambda(x,\tilde u,C_n) \, \eta(\dd x) = \lim_{n \to \infty} \int_S \lambda(x,\tilde u,C_n) \, \eta(\dd x) \\
= \int_S \lim_{n \to \infty} \lambda(x,\tilde u,C_n) \, \eta(\dd x) = \int_S \lambda(x,\tilde u,h^{-1}(\tilde y)) \, \eta(\dd x) = 0,
\end{multline*}
hence the contradiction.

\noindent\textbf{Case 3.}
If $\eta(A_{j,K}^c) = 0$ and there exists a set $\cI \ni B_{j,K} \subset A_{j,K}$ such that $\eta(B_{j,K}) > 0$ and $\lambda(x,\tilde u,h^{-1}(\tilde y)) > 0$ for all $x \in B_{j,K}$, then $\int_S \lambda(x,\tilde u,h^{-1}(\tilde y)) \, \eta(\dd x) > 0$ and so $H \coloneqq \sup_{u \in U} \int_S \lambda(x, u,h^{-1}(\tilde y)) \, \eta(\dd x) > 0$. Since from~\eqref{eq:contradiction} we have that $\lambda(x,\tilde u,C_{n_j}) > K H$ for all $j \in J_K$ and all $x \in B_{j,K}$, we get
\begin{equation*}
\sup_{(x,u) \in I \times U} \lambda(x,u) \geq \lambda(x,\tilde u,C_{n_j}) > K H
\end{equation*}
for any $K > 0$. Letting $K \uparrow +\infty$ we reach a contradiction with point~\ref{assumption:lambdabdd} of Assumption~\ref{assumption:lambdainf}.
\end{proof}

%
%
\section{Results on the flow of the filtering process}\label{app:Fvectorfield}
We recall that in Section \ref{sec:filterchar} we consider the set $\cM(I)$ endowed with the total variation norm, indicated by $\norm{\cdot}_{TV} \coloneqq |\cdot|(I)$ (where $|\cdot|$ denotes the total variation measure). It is worth to recall that this norm is equivalent to the one defined as $\cM(I) \ni \mu \mapsto \sup_{A \in \cI} |\mu(A)|$. In particular, from the Hahn decomposition of signed measures, we have that
\begin{equation}\label{eq:totalvariationequiv}
	\norm{\mu}_{TV} \leq 2 \sup_{A \in \cI} |\mu(A)| \leq 2 \norm{\mu}_{TV}, \quad \mu \in \cM(I).
\end{equation}

To prove Theorem \ref{th:ODEuniqueness} we need the following two results.

\begin{lemma}\label{lemma:Bycontinuous}
	Under Assumption \ref{hp:lambda}, for each fixed $y \in O$ the operator $\cB_y$ is linear and continuous on $\cM(I)$.
\end{lemma}
\begin{proof}
	Fix $y \in O$. Linearity is obvious. To prove continuity, fix $A \in \cI$ and $\mu, \nu \in \cM(I)$. Then, we get
	\begin{equation*}
	\begin{split}
		&{ } |\cB_y \nu(A) - \cB_y \mu(A)| \\
		&= \biggl|\int_A \ind_{h^{-1}(y)}(z) \int_I \lambda(x, \dd z) \, [\nu - \mu](\dd x) - \int_A \lambda(x) \, [\nu - \mu](\dd x)\biggr| \\
		&\leq \biggl| \int_I \lambda\bigl(x, A \cap h^{-1}(y)\bigr) \, [\nu - \mu](\dd x) \biggl| \, + \, \biggl|\int_A \lambda(x) \, [\nu - \mu](\dd x)\biggr| \\
		&\leq 2 \int_I \lambda(x) \, |\nu - \mu|(\dd x) \leq 2 \sup_{x \in I} \lambda(x) \norm{\nu - \mu}_{TV}.
	\end{split}
	\end{equation*}
	Since this inequality holds for all $A \in \cI$ we easily get
	\begin{equation*}
		\norm{\cB_y\nu - \cB_y\mu}_{TV} \leq 2 \sup_{A \in \cI} |\cB_y \nu(A) - \cB_y \mu(A)| \leq 4  \sup_{x \in I} \lambda(x) \norm{\nu - \mu}_{TV}
	\end{equation*}
	whence the continuity of the operator $\cB_y$.	
\end{proof}

\begin{proposition}\label{prop:Fylipschitz}
	Under Assumption \ref{hp:lambda}, for each fixed $y \in O$, the map $F_y$ is locally Lipschitz continuous on $\cM(I)$. In particular, it is Lipschitz continuous on $\cP(I)$.
\end{proposition}

\begin{proof}
	Fix $y \in O$ and $\mu, \nu \in \cM(I)$. Then, recalling (\ref{eq:totalvariationequiv}) and the result of Lemma \ref{lemma:Bycontinuous}, we have that
	\begin{align*}
	&\quad \, \norm{F_y(\nu) - F_y(\mu)}_{TV} = \norm{\cB_y \nu - \nu \cB_y \nu(I) - \cB_y \mu + \mu \cB_y \mu(I)}_{TV} \\
	&\leq \norm{\cB_y(\nu-\mu)}_{TV} + \norm{\cB_y \mu(I) [\nu-\mu]}_{TV} + \norm{\nu[\cB_y\nu(I) - \cB_y\mu(I)]}_{TV} \\
	&\leq 4\sup_{x \in I}\lambda(x) \norm{\nu-\mu}_{TV} + |\cB_y \mu(I)| \norm{\nu - \mu}_{TV} + \norm{\nu}_{TV} |\cB_y\nu(I) - \cB_y\mu(I)| \\
	&\leq (4 + \norm{\mu}_{TV}) \sup_{x \in I}\lambda(x) \norm{\nu-\mu}_{TV} + \norm{\nu}_{TV} \norm{\cB_y(\nu-\mu)}_{TV} \\
	&\leq (4 + \norm{\mu}_{TV} + 4\norm{\nu}_{TV}) \sup_{x \in I}\lambda(x) \norm{\nu-\mu}_{TV}
	\end{align*}
	whence the result. Notice that the term $|\cB_y \mu(I)| = \bigl|\int_I \lambda\bigl(x, h^{-1}(y)^c\bigr) \, \mu(\dd x)\bigr|$ is easily majorized by $\norm{\mu}_{TV} \sup_{x \in I} \lambda(x)$.
\end{proof}


\begin{theorem}
	Under Assumption \ref{hp:lambda}, for each fixed $y \in O$ the ODE (\ref{eq:filterODE}) admits a unique global solution $z \in \dC^1\bigl([0, +\infty); \Delta_y\bigr)$.
\end{theorem}

\begin{proof}
	Fix $y \in O$.
	The claim follows from \citep[Th. 4]{martin:differentialeq}. To apply it we have to verify the following assumptions (we point out in square brackets the reference to the corresponding hypotheses of the cited work. We invite the interested reader to consult it for further details).
	\begin{enumerate}
		\item $F_y$ is continuous from $\Delta_y$ into $\cM(I)$ [Condition C1].
		\item $\displaystyle \lim_{\epsilon \to 0^+} \frac{1}{\epsilon} \inf_{\nu \in \Delta_y}\norm{\mu + \epsilon F_y(\mu) - \nu}_{TV} = 0$ for all $\mu \in \Delta_y$ [Condition C2].
		\item For all $K > 0$ there exists $C_K > 0$ such that for all $\mu \in \Delta_y$ with $\norm{\mu}_{TV} \leq K$ it holds that $\norm{F_y(\mu)}_{TV} \leq C_K$ [(i) of Theorem 4].
		\item $\langle \mu - \nu, F_y(\mu) - F_y(\nu) \rangle_+ \leq C \norm{\mu - \nu}_{TV}^2$ for all $\mu, \nu \in \Delta_y$ and some $C > 0$ [(3.11)]
	\end{enumerate}
	where for all $\mu, \nu \in \cM(I)$ we define $\displaystyle \langle \mu, \nu \rangle_+ \coloneqq \suptwo{\Phi \in \cM(I)^\star}{\Phi \mu = \norm{\mu}_{TV}^2} \Phi \nu$. The set $\cM(I)^\star$ is the topological dual space of $\cM(I)$.
	
	\noindent \textbf{Claim 1.} In Proposition \ref{prop:Fylipschitz} we proved that the vector field $F_y$ is Lipschitz continuous on $\cP(I)$, hence we easily deduce its continuity on $\Delta_y$.
	
	\noindent \textbf{Claim 2.} Fix $\mu \in \Delta_y$. To get the claim, it suffices to prove that for $\epsilon > 0$ small enough $\mu + \epsilon F_y(\mu) \in \Delta_y$. 
	
	We prove, first that $\mu + \epsilon F_y(\mu) \in \cM_+(I)$ for $\epsilon > 0$ small enough. For all fixed $A \in \cI$ we have that
	\begin{align*}
		\bigl[\mu + \epsilon F_y(\mu)\bigr](A) 
		&= \mu(A) + \epsilon F_y(\mu; \, A) = \mu(A) + \epsilon \bigl[\cB_y \mu(A) - \mu(A) \cB_y \mu(I) \bigr] \\
		&= \mu(A) + \epsilon \biggl[ \int_I \lambda(x, A \cap h^{-1}(y)) \, \mu(\dd x) - \int_A \lambda(z) \, \mu(\dd z) \biggr. 
		\\ 
		&\qquad \biggl. {} + \mu(A) \int_I \lambda(x, h^{-1}(y)^c) \, \mu(\dd x) \biggr]
	\end{align*}
	Recalling that Assumption \ref{hp:lambda} is in force, we have the obvious estimate 
	\begin{equation*}
		\int_A \lambda(z) \, \mu(\dd z) \leq \sup_{x \in A} \lambda(x) \mu(A) \leq \sup_{x \in I} \lambda(x) \mu(A).
	\end{equation*}
	Hence, taking $\epsilon < [\sup\limits_{x \in I} \lambda(x)]^{-1}$ we get
	\begin{align*}
		\bigl[\mu + \epsilon F_y(\mu)\bigr](A)
		&> \epsilon \biggl[\int_I \lambda(x, A \cap h^{-1}(y)) \, \mu(\dd x) \biggr. \\
		&\qquad \biggl. {} + \mu(A) \int_I \lambda(x, h^{-1}(y)^c) \, \mu(\dd x) \biggr] \geq 0.
	\end{align*}
	
	Now it remains to prove that $\mu + \epsilon F_y(\mu)$ is a probability measure and that is concentrated on $h^{-1}(y)$.
	This can be easily shown, since it is immediately seen that $\bigl[\mu + \epsilon F_y(\mu)\bigr](I) = 1$ and we have that
	\begin{align*}
	\bigl[\mu + \epsilon F_y(\mu)\bigr]\bigl(h^{-1}(y)\bigr) 
	&= \mu\bigl(h^{-1}(y)\bigr) + \epsilon F_y\bigl(\mu; \, h^{-1}(y)\bigr) 
	\\
	&= 1 + \epsilon \biggl[ \int_{h^{-1}(y)} \int_I \lambda(x, \dd z) \, \mu(\dd x) - \int_{h^{-1}(y)} \lambda(z) \, \mu(\dd z) \biggr. 
	\\ 
	&\qquad \biggl. {} - \int_{h^{-1}(y)} \int_I \lambda(x, \dd z) \, \mu(\dd x) + \int_I \lambda(z) \, \mu(\dd z) \biggr] = 1
	\end{align*}
	thanks to the equality $\int_{h^{-1}(y)} \lambda(z) \, \mu(\dd z) = \int_I \lambda(z) \, \mu(\dd z)$, implied by the fact that $\mu \in \Delta_y$.
	
	\noindent \textbf{Claim 3.} Fix $\mu \in \Delta_y$. The claim is easily proved thanks to the following estimate, holding for all $A \in \cI$.
	\begin{align*}
		\abs{F_y(\mu; \, A)} 
		&= \abs{\cB_y \mu(A) - \mu(A) \cB_y(I)} \\
		&= \biggl| \int_I \lambda(x, A \cap h^{-1}(y)) \, \mu(\dd x) - \int_A \lambda(z) \, \mu(\dd z) \biggr. \\ 
		&\qquad \biggl. {} + \mu(A) \int_I \lambda(x, h^{-1}(y)^c) \, \mu(\dd x) \biggr| \leq 3 \sup_{x \in I} \lambda(x).
	\end{align*}
	From this inequality, it follows that $\norm{F_y(\mu)}_{TV} \leq 6 \sup_{x \in I} \lambda(x)$, whence the result.
	
	\noindent \textbf{Claim 4.} Fix $\mu, \nu \in \Delta_y$ and take $\Phi \in \cM(I)^\star$ such that $\Phi(\mu - \nu) = \norm{\mu - \nu}_{TV}^2 = \norm{\Phi}_\star^2$, where $\norm{\cdot}_\star$ denotes the norm in the dual space $\cM(I)^\star$. Thanks to Proposition \ref{prop:Fylipschitz} we have that
	\begin{equation*}
		\Phi\bigl(F_y(\mu) - F_y(\nu)\bigr) \leq \norm{\Phi}_\star \, \norm{F_y(\mu) - F_y(\nu)}_{TV} \leq 9 \sup_{x \in I} \lambda(x) \norm{\mu-\nu}_{TV}^2.
	\end{equation*} 
	Since this estimate holds for all required $\Phi$, we get the result taking the supremum.
\end{proof}

\end{document}